\documentclass[a4paper,10pt,psamsfonts]{amsart}

\usepackage[english]{babel}
\usepackage[english]{translator}
\usepackage[T1]{fontenc}
\usepackage[utf8]{inputenc}

\usepackage{amsmath}
\usepackage{amssymb}
\usepackage{amsfonts}
\usepackage{amstext}
\usepackage{amsthm}
\usepackage{bbm}

\usepackage[top=4cm, bottom=3cm, left=3.2cm, right=3.2cm]{geometry}

\usepackage{graphicx}
\usepackage{color}
\usepackage{psfrag}
\usepackage{subfigure}
\usepackage{float}

\usepackage{marvosym}

\usepackage{hyperref}

\hypersetup{
  pdffitwindow=false,
  pdfhighlight=/O,
  pdfnewwindow,
  colorlinks=true,
  citecolor=red,            
  linkcolor=blue,             
  menucolor=blue,            
  urlcolor=blue,             
  pdfpagemode=UseOutlines,
  bookmarksnumbered=true,
  linktocpage,
  pdftitle={A heat kernel for the complex Ornstein-Uhlenbeck operator and its semigroup in exponentially weighted function spaces},
  pdfsubject={heat kernel matrix, complex Ornstein-Uhlenbeck operator, exponentially weighted funtion spaces},
  pdfauthor={Denny Otten},
  pdfkeywords={heat kernel matrix, complex Ornstein-Uhlenbeck operator, exponentially weighted funtion spaces},
  pdfcreator={},
  pdfproducer={LaTeX mit hyperref}
}

\newcommand{\D}{\mathcal{D}}              
\newcommand{\K}{\mathbb{K}}               
\newcommand{\C}{\mathbb{C}}               
\newcommand{\R}{\mathbb{R}}               
\newcommand{\N}{\mathbb{N}}                
\renewcommand{\Re}{\mathrm{Re}\,}          
\newcommand{\trace}{\mathrm{tr}}           
\renewcommand{\L}{\mathcal{L}}             
\newcommand{\supp}{\mathrm{supp}}          
\newcommand{\cond}{\mathrm{cond}}          
\newcommand{\Cb}{C_{\mathrm{b}}}
\newcommand{\Cub}{C_{\mathrm{ub}}}
\newcommand{\Crub}{C_{\mathrm{rub}}}
\newcommand{\Cbt}{C_{\mathrm{b},\theta}}
\newcommand{\Cubt}{C_{\mathrm{ub},\theta}}
\newcommand{\Crubt}{C_{\mathrm{rub},\theta}}

\newcommand{\amin}{a_{\mathrm{min}}} 
\newcommand{\amax}{a_{\mathrm{max}}} 
\newcommand{\azero}{a_0} 
\newcommand{\bzero}{b_0}
\newcommand{\czero}{c_0} 

\newcommand{\aone}{a_1}

\newcommand{\begriff}[1]{\textbf{#1}}

\newcommand{\sect}
{
  \setcounter{equation}{0}
  \setcounter{figure}{0}
  \section
}

\newcommand{\enum}[1]{\textnormal{(\textbf{#1})}}

\theoremstyle{plain}
\newtheorem{definition}{Definition}[section]
\newtheorem{theorem}[definition]{Theorem}
\newtheorem{lemma}[definition]{Lemma}
\newtheorem{corollary}[definition]{Corollary}
\newtheorem{assumption}[definition]{Assumption}
\newtheorem{remark}[definition]{Remark}

\theoremstyle{definition}

\allowdisplaybreaks

\begin{document}

\title[Resolvent estimates for complex Ornstein-Uhlenbeck systems]{Exponentially weighted resolvent estimates for\\complex Ornstein-Uhlenbeck systems}

\maketitle

\vspace*{0.5cm}
\noindent
\hspace*{4.2cm}
%
\textbf{Denny Otten}\footnote[1]{e-mail: \textcolor{blue}{dotten@math.uni-bielefeld.de}, phone: \textcolor{blue}{+49 (0)521 106 4784}, \\
                                 fax: \textcolor{blue}{+49 (0)521 106 6498}, homepage: \url{http://www.math.uni-bielefeld.de/~dotten/}, \\
                                 supported by CRC 701 'Spectral Structures and Topological Methods in Mathematics'.} \\
\hspace*{4.2cm}
Department of Mathematics \\
\hspace*{4.2cm}
Bielefeld University \\
\hspace*{4.2cm}
33501 Bielefeld \\
\hspace*{4.2cm}
Germany

\vspace*{1.0cm}
\hspace*{4.2cm}
Date: \today
\normalparindent=12pt

\begin{abstract}
  In this paper we study differential operators of the form
  \begin{align*}
    \left[\L_{\infty}v\right](x) = A\triangle v(x) + \left\langle Sx,\nabla v(x)\right\rangle - Bv(x),\,x\in\R^d,\,d\geqslant 2,
  \end{align*}
  for matrices $A,B\in\C^{N,N}$, where the eigenvalues of $A$ have positive real parts. The sum 
  $A\triangle v(x)+\left\langle Sx,\nabla v(x)\right\rangle$ is known as the Ornstein-Uhlenbeck operator with an unbounded drift 
  term defined by a skew-symmetric matrix $S\in\R^{d,d}$. Differential operators such as $\L_{\infty}$ arise as linearizations at rotating 
  waves in time-dependent reaction diffusion systems. The results of this paper serve as foundation for proving exponential decay of 
  such waves. Under the assumption that $A$ and $B$ can be diagonalized simultaneously we construct a heat kernel matrix $H(x,\xi,t)$ 
  of $\L_{\infty}$ that solves the evolution equation $v_t=\L_{\infty}v$. In the following we study the Ornstein-Uhlenbeck semigroup
  \begin{align*}
    \left[T(t)v\right](x) = \int_{\R^d}H(x,\xi,t)v(\xi)d\xi,\,x\in\R^d,\,t>0,
  \end{align*}
  in exponentially weighted function spaces. This is used to derive resolvent estimates for $\L_{\infty}$ in exponentially weighted $L^p$-spaces 
  $L^p_{\theta}(\R^d,\C^N)$, $1\leqslant p<\infty$, as well as in exponentially weighted $\Cb$-spaces $\Cbt(\R^d,\C^N)$.
\end{abstract}




\sect{Introduction}
\label{sec:Introduction}

In this paper we study differential operators of the form
\begin{align*}
  \left[\L_{\infty}v\right](x) := A\triangle v(x) + \left\langle Sx,\nabla v(x)\right\rangle - Bv(x),\,x\in\R^d,\,d\geqslant 2,
\end{align*}
for simultaneously diagonalizable matrices $A,B\in\C^{N,N}$ with $\Re\sigma(A)>0$ and a skew-symmetric matrix $S\in\R^{d,d}$.

Introducing the complex Ornstein-Uhlenbeck operator, \cite{UhlenbeckOrnstein1930},
\begin{align*}
  \left[\L_0 v\right](x) := A\triangle v(x) + \left\langle Sx,\nabla v(x)\right\rangle,\,x\in\R^d,
\end{align*}
with diffusion term and drift term given by
\begin{align*}
  A\triangle v(x):=A\sum_{i=1}^{d}\frac{\partial^2}{\partial x_i^2}v(x)\quad\text{and}\quad\left\langle Sx,\nabla v(x)\right\rangle:=\sum_{i=1}^{d}(Sx)_i\frac{\partial}{\partial x_i}v(x),
\end{align*}
we observe that the operator $\L_{\infty}=\L_0-B$ is a constant coefficient perturbation of $\L_0$. Our interest is in skew-symmetric matrices $S=-S^T$, in which case 
the drift term is a rotational term containing angular derivatives
\begin{align*}
  \left\langle Sx,\nabla v(x)\right\rangle=\sum_{i=1}^{d-1}\sum_{j=i+1}^{d}S_{ij}\left(x_j\frac{\partial}{\partial x_i}-x_i\frac{\partial}{\partial x_j}\right)v(x).
\end{align*}
Such problems arise when investigating exponential decay of rotating waves in reaction diffusion systems, see \cite{Otten2014} and \cite{BeynLorenz2008}. 
The operator $\L_{\infty}$ appears as a far-field linearization at the solution of the nonlinear problem $\L_0 v=f(v)$ for the Ornstein-Uhlenbeck operator. 
For details see \cite{Otten2014}.

The aim of this paper is to prove unique solvability of the resolvent equation for the realization of $\L_{\infty}$ in different types of exponentially weighted 
function spaces. To be more precise, we investigate the resolvent equation in exponentially weighted $L^p$-spaces $L^p_{\theta}(\R^d,\C^N)$ for $1\leqslant p\leqslant\infty$ and 
in exponentially weighted spaces of bounded continuous functions $\Cbt(\R^d,\C^N)$. The class of underlying weight functions $\theta\in C(\R^d,\R)$ is taken 
from \cite{ZelikMielke2009} and allows both, exponentially decreasing and exponentially increasing weight functions. 

We first determine a complex-valued heat kernel matrix of $\L_{\infty}$, that enables us to introduce the corresponding Ornstein-Uhlenbeck semigroup $(T(t))_{t\geqslant 0}$. 
We then show that the semigroup is strongly continuous on the above function spaces (or possibly on certain subspaces), which justifies to define the infinitesimal generator. 
The generator can be considered as a type of realization of $\L_{\infty}$ on the underlying function space. We apply general results from semigroup theory to deduce unique 
solvability of the resolvent equation for the generator and derive a-priori estimates for the solution of the resolvent equation. For a more detailed outline we refer to 
Section \ref{sec:AssumptionsAndMainResult}. 

In the following we comment on some well-known results concerning the scalar Ornstein-Uhlenbeck operator
\begin{align}
  \label{equ:GeneralOrnsteinUhlenbeckOperator}
  \left[\L_{\mathrm{OU}}(t)v\right](x):=\frac{1}{2}\trace\left(Q D^2 v(x)\right)+\left\langle Sx,\nabla v(x)\right\rangle,\,x\in\R^d
\end{align}
in real-valued function spaces with covariance matrix $Q\in\R^{d,d}$, $Q=Q^T$, $Q>0$ and drift matrix $0\neq S\in\R^{d,d}$. The associated semigroup $(T_{\mathrm{OU}}(t))_{t\geqslant 0}$ 
is given by
\begin{align*}
  \left[T_{\mathrm{OU}}(t)v\right](x) = (4\pi)^{-\frac{d}{2}}\left(\det Q_t\right)^{-\frac{1}{2}}\int_{\R^d}e^{-\frac{1}{4}\left\langle Q_t^{-1}\xi,\xi\right\rangle}v(e^{tS}x-\xi)d\xi,\,t>0
\end{align*}
with covariance operators
\begin{align*}
  Q_t:=\int_{0}^{t}e^{\tau S}Qe^{\tau S^T}d\tau.
\end{align*}
Here we use the terminology from \cite{MauceriNoselli2009}. We give a detailed overview of results in unweighted function spaces, in which case a wide range of 
literature exists. A collection of most of these results including proofs can be found in \cite{LorenziBertoldi2007}. We also briefly discuss similarities and 
differences of our results to those in the literature.

\textbf{The space $L^p(\R^d,\R)$.} The Ornstein-Uhlenbeck semigroup $(T_{\mathrm{OU}}(t))_{t\geqslant 0}$ is a semigroup in $L^p(\R^d,\R)$ for every $1\leqslant p\leqslant\infty$, 
that is even strongly continuous on $L^p(\R^d,\R)$ but only for $1\leqslant p<\infty$, \cite[Lemma 3.1]{Metafune2001}. Our arguments to prove strong continuity in exponentially 
weighted spaces follow those from \cite{Metafune2001}. Therein, the author uses the explicit representation of the semigroup and shows strong continuity directly. This 
implies the existence of the infinitesimal generator $A_p$ by results from abstract semigroup theory, see \cite[Section II.1]{EngelNagel2000}. One problem that occurs and which 
is caused by the unbounded coefficients in the drift term, is to find an explicit representation for the domain of the infinitesimal generator $A_p$, which is the 
maximal realization of $\L_{\mathrm{OU}}$ in $L^p(\R^d,\R)$ for $1\leqslant p<\infty$. The maximal domain is given by 
\begin{align*}
  \D^p_{\mathrm{max}}(\L_{\mathrm{OU}})=\{v\in W^{2,p}(\R^d,\R)\mid \left\langle Sx,\nabla v(x)\right\rangle\in L^p(\R^d,\R)\} 
\end{align*}
for every $1<p<\infty$. This can either be shown directly, \cite[Theorem 1]{MetafunePallaraVespri2005}, or with the aid of the Dore-Venni theorem, 
\cite[Theorem 2.4]{PruessRhandiSchnaubelt2006}, but we also refer \cite[Theorem 4.1]{MetafunePruessRhandiSchnaubelt2002} and 
to \cite[Theorem 1]{BramantiCupiniLanconelliPriola2010}, \cite[Theorem 1]{BramantiCupiniLanconelliPriola2013} for global $L^p$-estimates concerning degenerate Ornstein-Uhlenbeck 
operators. In case $p=1$ one has the weaker result that $\D^1(\L_{\mathrm{OU}})$ is the closure of $C_{c}^{\infty}(\R^d,\R)$ with respect to the graph norm of $\L_{\mathrm{OU}}$
$\left\|\cdot\right\|_{\L_{\mathrm{OU}}}:=\left\|\cdot\right\|_{L^1}+\left\|\L_{\mathrm{OU}}\cdot\right\|_{L^1}$, i.e. 
$\D^1(\L_{\mathrm{OU}})=\overline{C_{\mathrm{c}}^{\infty}}^{\left\|\cdot\right\|_{\L_{\mathrm{OU}}}}$, see \cite[Proposition 3.2]{Metafune2001} and also \cite{LunardiMetafune2004}.
An argument, different from ours, which avoids solving the identification problem is presented in \cite[Theorem 2.2]{MetafunePallaraVespri2005}. Therein, the authors 
show that the operator $\L_0$ is closed on a suitable domain and deduce the existence of a strongly continuous semigroup. Their arguments are based on results from 
\cite{GilbargTrudinger2010} for scalar real-valued differential operators which is not directly applicable in the complex setting. This justifies why we follow the semigroup 
approach from \cite{Metafune2001}. 

It is also shown in \cite[Theorem 4.4, 4.7, 4.11 and 4.12]{Metafune2001} that the spectrum of the infinitesimal generator $A_p$ is given by 
\begin{align*}
  \sigma(A_p)=\left\{z\in\C\mid \Re z\leqslant -\frac{\trace(S)}{p}\right\}
\end{align*} 
for every $1\leqslant p<\infty$, if $\sigma(S)\subset\C_{+}$, $\sigma(S)\subset\C_{-}$ or $S$ symmetric and $Q$ and $S$ commute. Therefore, 
the semigroup $(T_{\mathrm{OU}}(t))_{t\geqslant 0}$ is not analytic on $L^p(\R^d,\R)$ for $1\leqslant p<\infty$, if $S\neq 0$, and the parabolic equation $v_t=\L_{\mathrm{OU}} v$ 
does not satisfy the standard parabolic regularity properties on $L^p(\R^d,\R)$. The results for the $L^p$-case can also be found in \cite[Section 9.4]{LorenziBertoldi2007}. 
Investigating the maximal domain, the spectrum and the analyticity in exponentially weighted spaces will be part of our future work.

\textbf{The space $L^p(\R^d,\R,\mu)$.} Under the additional assumption $\sigma(S)\subset\C_{-}$, which is very interesting from the point of view of diffusion processes, 
the Ornstein-Uhlenbeck semigroup $(T_{\mathrm{OU}}(t))_{t\geqslant 0}$ with invariant probability measure
\begin{align*}
  \mu(x)=\left(4\pi\right)^{-\frac{d}{2}}\left(\det Q_{\infty}\right)^{-\frac{1}{2}}
         e^{-\frac{1}{4}\left\langle Q_{\infty}^{-1}x,x\right\rangle}
\end{align*} 
is a semigroup of positive contractions on $L^p(\R^d,\R,\mu)$ for every $1\leqslant p\leqslant\infty$ and a $C_0$-semigroup for every $1\leqslant p<\infty$. The maximal domain is 
given by 
\begin{align*}
  \D^p_{\mathrm{max},\mu}(\L_{\mathrm{OU}}) = W^{2,p}(\R^d,\R,\mu) 
                                 = \{v\in L^p(\R^d,\R,\mu)\mid D_iv, D_jD_iv \in L^p(\R^d,\R,\mu),\,i,j=1,\ldots,d\}
\end{align*}
for every $1<p<\infty$, see \cite[Theorem 3.4]{MetafunePruessRhandiSchnaubelt2002} and \cite[Theorem 4.1]{Lunardi1997} for $p=2$. For $p=1$ we have the weaker result 
$\D^1_{\mu}(\L_{\mathrm{OU}})=\overline{C_{\mathrm{c}}^{\infty}}^{\left\|\cdot\right\|_{\L_{\mathrm{OU}}}}$. 
A major difference to the usual $L^p$-cases is that $(T_{\mathrm{OU}}(t))_{t\geqslant 0}$ is compact and analytic on $L^p(\R^d,\R,\mu)$ for every $1<p<\infty$, see \cite[Proposition 2.4]{MetafunePallaraPriola2002} and \cite[Section 2.2]{DaPratoLunardi1995} for $p=2$. The analyticity is caused by the fact that the probability measure is Gaussian. For our applications 
it is important that the weight function is non Gaussian and the corresponding semigroup is not analytic in exponentially weighted spaces. In \cite[Theorem 3.4]{MetafunePallaraPriola2002}, 
it is shown for $1<p<\infty$ that the spectrum of the infinitesimal generator $A_{p,\mu}$ of the Ornstein-Uhlenbeck semigroup $(T_{\mathrm{OU}}(t))_{t\geqslant 0}$ in $L^p(\R^d,\R,\mu)$ 
is a discrete set, independent of $p$ and given by 
\begin{align*}
  \sigma(A_{p,\mu})=\left\{\lambda=\sum_{i=1}^{r}n_i\lambda_i\mid n_i\in\N_{0},\,i=1,\ldots,r\right\},
\end{align*}
where $\lambda_1,\ldots,\lambda_r$ denote the distinct eigenvalues of $S$. This is in strong contrast to the $L^p$-case. The eigenvalues are semisimple if and only if $S$ is 
diagonalizable over $\C$. Moreover, the eigenfunctions of $A_{p,\mu}$ are polynomials of degree at most $\frac{\Re\lambda}{s(S)}$, \cite[Proposition 3.2]{MetafunePallaraPriola2002}, 
where $s(S):=\max_{\lambda\in\sigma(S)}\Re\lambda$ denotes the spectral bound of $S$. In case $p=1$ the situation changes drastically and the spectrum is given by 
$\sigma(A_{1,\mu})=\C_{-}\cup i\R$, \cite[Theorem 4.1]{MetafunePallaraPriola2002}. Results concerning the $L^p_{\mu}$-case can also be found in \cite[Section 9.3]{LorenziBertoldi2007}.

\textbf{The space $C_{\mathrm{b}}(\R^d,\R)$.} The Ornstein-Uhlenbeck semigroup $(T_{\mathrm{OU}}(t))_{t\geqslant 0}$ is a semigroup on $\Cb(\R^d,\R)$. To guarantee strong continuity of 
$(T_{\mathrm{OU}}(t))_{t\geqslant 0}$ one usually considers the semigroup on the closed subspace $\Cub(\R^d,\R)$ if the operator has constant or smooth bounded coefficients. But in case of 
the Ornstein-Uhlenbeck operator this space leads only to a weakly continuous semigroup, since the additional term $\left\langle Sx,\nabla v(x)\right\rangle$ has smooth but unbounded 
coefficients, and hence, the space $\Cub(\R^d,\R)$ is too large in order to guarantee strong continuity, see \cite[Section 6]{Cerrai1994}. One can show that 
$(T_{\mathrm{OU}}(t))_{t\geqslant 0}$ is a $C_0$-semigroup on the much smaller subspace $\Crub(\R^d,\R)$, see \cite[Lemma 3.2]{DaPratoLunardi1995} and \cite[Section I.6]{DaPratoZabczyk2002}. 
In Section \ref{sec:TheOrnsteinUhlenbeckSemigroupInCrub} we show that this result extends to the complex-valued Ornstein-Uhlenbeck semigroup associated to $\L_{\infty}$. 
The domain is characterized by 
\begin{align*}
  \D(\L_{\mathrm{OU}})=\bigg\{v\in\Crub(\R^d,\R)\cap \bigg(\bigcap_{p\geqslant 1}W^{2,p}_{\mathrm{loc}}(\R^d,\R)\bigg)\bigg| \L_{\mathrm{OU}}v\in\Crub(\R^d,\R)\bigg\},
\end{align*}
see \cite[Proposition 3.5]{DaPratoLunardi1995}. By \cite[Proposition 6.1, Theorem 6.2, Corollary 6.3]{Metafune2001} the spectrum of the (weak) infinitesimal generator 
$A_{\mathrm{b}}$ of the Ornstein-Uhlenbeck semigroup $(T_{\mathrm{OU}}(t))_{t\geqslant 0}$ in $\Cub(\R^d,\R)$ is given by
\begin{align*}
  \sigma(A_{\mathrm{b}})=\{\lambda\in\C\mid \Re\lambda\leqslant 0\},
\end{align*}
if $\sigma(S)\cap\C_{+}\neq\emptyset$, $\sigma(S)\subset\C_{-}$ or $\sigma(S)\cap i\R=\emptyset$. Moreover, it is proved in \cite[Lemma 3.3]{DaPratoLunardi1995} that 
$(T_{\mathrm{OU}}(t))_{t\geqslant 0}$ is not analytic on $\Crub(\R^d,\R)$ and hence it is neither analytic on $\Cb(\R^d,\R)$ nor on $\Cub(\R^d,\R)$. Results concerning the $\Cb$-theory 
can also be found in \cite[Section 9.2]{LorenziBertoldi2007}. We refer to \cite{DaPratoLunardi1995} for additional information on the Ornstein-Uhlenbeck operator in spaces 
of H\"older-continuous functions $\Cb^{\gamma}(\R^d,\R)$.

In this paper we analyze the perturbed operator $\L_{\infty}$ in exponentially weighted $L^p$- and $\Cb$-spaces namely $L^p_{\theta}(\R^d,\C^N)$ and $\Cbt(\R^d,\C^N)$, 
respectively. Obviously, the operator $\L_{\infty}$ generalizes the Ornstein-Uhlenbeck operator $\L_{\mathrm{OU}}$, since for $N=1$, $A=1$ and $B=0$ we have $\L_{\infty}=\L_0=\L_{\mathrm{OU}}$ 
with $Q=I_d$. The whole theory in this paper extends also to $Q\neq I_d$, in which case the diffusion term slightly changes into $\frac{1}{2}A\trace(Q D^2v(x))$. But for 
simplicity we restrict to the case $Q=I_d$.

\textbf{The spaces $L^p_{\theta}(\R^d,\R)$, $\Cbt(\R^d,\R)$, $\Cbt^{\gamma}(\R^d,\R)$.} There are only few papers in which the Ornstein-Uhlenbeck operator is analyzed in 
exponentially weighted spaces, see \cite{Addona2013} for $\Cbt(\R^d,\R)$ and $\Cbt^{\gamma}(\R^d,\R)$ and see \cite{HarboureTorreaViviani2000}, \cite{SobajimaYokota2013} for 
$L^p_{\theta}(\R^d,\R)$. In \cite{Addona2013} the author studies nonautonomous Ornstein-Uhlenbeck operators in weighted spaces of bounded continuous functions 
$\Cbt(\R^d,\R)$ and in weighted H\"older spaces $\Cb^{\gamma}(\R^d,\R)$. Therein, the author uses the weight functions $\theta(x)=\left(1+|x|^{2m}\right)^{-1}$, $m\in\N$, 
and $\theta(x)=e^{-(1-|x|^2)\rho}$, $\rho\in(0,\frac{1}{2}]$, and proves sharp uniform estimates for the spatial derivatives of the associated evolution operator,
that are used to prove optimal Schauder estimates for solutions of nonhomogeneous parabolic Cauchy problems. We emphasize that our resolvent estimates require stronger condition 
than those for the solvability of Cauchy problems. Moreover, the weight functions used in \cite{Addona2013} (and also those from \cite{HarboureTorreaViviani2000} and \cite{SobajimaYokota2013}) 
are all of Gaussian type, in which case the associated semigroup is known to be analytic. We recall that our class includes non-Gaussian weight functions. A further 
difference to \cite{Addona2013} is that we study systems of complex-valued Ornstein-Uhlenbeck operators. Such coupled operators appear in different fields of application 
when investigating rotating waves (as well as their interactions) in reaction diffusion equations, see \cite{BeynLorenz2008} and \cite{Otten2014}. 

Some numerical examples that exhibit such rotating wave solutions are the cubic-quintic complex Ginzburg-Landau equation, the $\lambda$-$\omega$-system and Barkley's model 
in \cite[Section 2.1, 10.3 and 10.6]{Otten2014}. We refer to \cite{DelmonteLorenzi2011} for real-valued weakly coupled systems and to \cite{Pascucci2005} for a general survey 
of applications concerning the Ornstein-Uhlenbeck operator. Scalar complex-valued Ornstein-Uhlenbeck processes also appear in applications when analyzing the so-called 
Chandler Wobble, see \cite{AratoBaranIspany1999} and the references therein. In \cite{AratoBaranIspany1999} the authors consider complex matrices $S\in\C^{2,2}$ that lead 
to a coupling in the drift term, which is in contrast to our case having a coupling in the diffusion term. In this paper we consider skew-symmetric matrices $S\in\R^{d,d}$. 
Scalar real-valued skew-symmetric Ornstein-Uhlenbeck processes have been analyzed for instance in \cite{MauceriNoselli2009}. Finally, we stress that our semigroup approach 
that follows \cite{Metafune2001}, is crucial to obtain uniqueness of the resolvent equation. In the scalar real-valued case one usually applies a maximum principle to obtain 
uniqueness, which seems not to be available in complex-valued case. 

We emphasize that the results from Section \ref{sec:AHeatKernelMatrixForTheComplexOrnsteinUhlenbeckOperator} and \ref{sec:SomePropertiesOfTheHeatKernel} 
are based on the PhD thesis \cite{Otten2014}. Section \ref{sec:TheOrnsteinUhlenbeckSemigroupInLp} is an extension of the results from \cite{Otten2014}, 
where only the unweighted case was considered, and Section \ref{sec:TheOrnsteinUhlenbeckSemigroupInCrub} is completely new.

\sect{Assumptions and outline of results}
\label{sec:AssumptionsAndMainResult}

Consider the differential operator
\begin{align}
  \label{equ:Linfty}
  \left[\L_{\infty}v\right](x) := A\triangle v(x) + \left\langle Sx,\nabla v(x)\right\rangle - Bv(x),\,x\in\R^d,
\end{align}
for some matrices $A,B\in\C^{N,N}$ and $S\in\R^{d,d}$. We define a heat kernel matrix of $\L_{\infty}$ in the sense of \cite[Section 1.2]{CalinChangFurutaniIwasaki2011}:

\begin{definition}\label{def:HeatKernel}
  A \begriff{heat kernel} (or a \begriff{fundamental solution}) \begriff{of $\L_{\infty}$} given by \eqref{equ:Linfty} is a function
  \begin{align*}
    H:\R^d\times\R^d\times]0,\infty[\rightarrow\C^{N,N},\,(x,\xi,t)\mapsto H(x,\xi,t)
  \end{align*}
  such that the following properties are satisfied
  \begin{align*}
    &H\in C^{2,2,1}(\R^d\times\R^d\times\R_+^*,\C^{N,N}),                                          \tag{H1}\label{equ:H1} \\
    &\frac{\partial}{\partial t}H(x,\xi,t)=\L_{\infty}H(x,\xi,t) &&\forall\,x,\xi\in\R^d,\,t>0,      \tag{H2}\label{equ:H2} \\
    &\lim_{t\downarrow 0}H(x,\xi,t)=\delta_{x}(\xi)I_N             &&\forall\,x,\xi\in\R^d.          \tag{H3}\label{equ:H3}
  \end{align*}
  Note that the differential operator $\L_{\infty}$ in \eqref{equ:H2} acts on $x$. Moreover, the convergence in \eqref{equ:H3} is meant in the sense of distributions 
  and $\delta_{x}(\xi)=\delta(x-\xi)$ denotes the Dirac delta function. If $N>1$ then $H$ is also called a \begriff{heat kernel matrix} (or \begriff{matrix fundamental solution}) 
  \begriff{of $\L_{\infty}$}.
\end{definition}

In order to determine a heat kernel matrix for the differential operator $\L_{\infty}$, we require the following
\begin{assumption} 
  \label{ass:Assumption1}
  Let $A,B\in\C^{N,N}$ and $S\in\R^{d,d}$ be such that
  \begin{align}
    &\text{$A$ and $B$ are simultaneously diagonalizable (with transformation matrix $Y\in\C^{N,N}$)}         \tag{A1}\label{cond:A8B} \\
    &\Re\sigma(A)>0,                                                           \tag{A2}\label{cond:A2} \\
    &\text{$S$ is skew-symmetric}.                                             \tag{A3}\label{cond:A5} 
  \end{align}
\end{assumption}

Condition \eqref{cond:A8B} is a system condition and ensures that all results for scalar equations can be extended to system cases. It is motivated by the fact that transforming 
a scalar complex-valued equation into a $2$-dimensional real-valued system, the corresponding (real) matrices $A$ and $B$ are always simultaneously diagonalizable 
(over $\C$). The positivity assumption \eqref{cond:A2} guarantees that the diffusion part $A\triangle$ is an elliptic operator and requires that all eigenvalues 
$\lambda$ of $A$ are contained in the open right half-plane $\C_{+}:=\{\lambda\in\C\mid\Re\lambda>0\}$, where $\sigma(A)$ denotes the spectrum of $A$. Assumption 
\eqref{cond:A5} implies that the drift term contains only angular derivatives, which is crucial when investigating rotating patterns in reaction diffusion systems.

For a matrix $C\in\K^{N,N}$ we denote by $\sigma(C)$ the \begriff{spectrum of $C$}, by $\rho(C):=\max_{\lambda\in\sigma(C)}\left|\lambda\right|$ the \begriff{spectral radius of $C$} 
and by $s(C):=\max_{\lambda\in\sigma(C)}\Re\lambda$ the \begriff{spectral abscissa} (or \begriff{spectral bound}) \begriff{of $C$}. Using this notation, we define the constants
\begin{equation}
  \begin{aligned}
  \amin  :=& \left(\rho\left(A^{-1}\right)\right)^{-1}, &&\azero := -s(-A), \\
  \amax  :=& \rho(A),                                   &&\bzero := -s(-B).
  \end{aligned}
  \label{equ:aminamaxazerobzero}
\end{equation}
These quantities appear in the heat kernel estimates below. Moreover, for an invertible matrix $C\in\K^{N,N}$ we denote by $\kappa:=\cond(C):=|C||C^{-1}|$ the 
\begriff{condition number of $C$}.

Our main tool for investigating the semigroup in exponentially weighted function spaces is the choice of the weight function. The first part of the 
definition comes originally from \cite[Def. 3.1]{ZelikMielke2009}:
\begin{definition}\label{def:WeightFunctionOfExponentialGrowthRate} 
  \enum{1} A function $\theta\in C(\R^d,\R)$ is called a \begriff{weight function of exponential growth rate $\eta\geqslant 0$} provided that 
  \begin{align}
    &\theta(x)>0\;\forall\,x\in\R^d,                                                                        \tag{W1}\label{equ:WeightFunctionProp1} \\
    &\exists\,C_{\theta}>0:\;\theta(x+y)\leqslant C_{\theta}\theta(x)e^{\eta|y|}\;\forall\,x,y\in\R^d.      \tag{W2}\label{equ:WeightFunctionProp2}
  \end{align}
  \enum{2} A weight function $\theta\in C(\R^d,\R)$ of exponential growth rate $\eta\geqslant 0$ is called \begriff{radial} provided that 
  \begin{align}
    \exists\,\phi:[0,\infty[\rightarrow\R:\;\theta(x)=\phi\left(\left|x\right|\right)\;\forall\,x\in\R^d.   \tag{W3}\label{equ:WeightFunctionProp3}
  \end{align}
\end{definition}

Note that skew-symmetry of $S$ from \eqref{cond:A5} and condition \eqref{equ:WeightFunctionProp3} imply the important relation $\theta(e^{tS}x)=\theta(x)$ for any $t\in\R$ and $x\in\R^d$. 
For the proof of strong continuity the following two properties are essential:
\begin{align}
  &\lim_{|\psi|\to 0}\sup_{x\in\R^d}\left|\frac{\theta(x+\psi)-\theta(x)}{\theta(x)}\right|=0, \tag{W4}\label{equ:WeightFunctionProp4} \\
  &\lim_{t\to 0}\sup_{x\in\R^d}\left|\frac{\theta(e^{tS}x)-\theta(x)}{\theta(x)}\right|=0.     \tag{W5}\label{equ:WeightFunctionProp5}
\end{align}
Obviously, \eqref{equ:WeightFunctionProp3} implies \eqref{equ:WeightFunctionProp5}, since $S$ is skew-symmetric due to \eqref{cond:A5}. Moreover, if we assume 
\eqref{equ:WeightFunctionProp1} and \eqref{equ:WeightFunctionProp2} then condition \eqref{equ:WeightFunctionProp4} follows directly if $C_{\theta}=1$. 
A convenient sufficient condition to guarantee \eqref{equ:WeightFunctionProp4} in case $\theta\in C^1(\R^d,\R)$ is
\begin{align}
  \exists\,C>0:\;\left|\nabla\theta(x)\right|\leqslant C\theta(x)\;\forall\,x\in\R^d. \tag{W6}\label{equ:WeightFunctionProp6}
\end{align}
A further property that is necessary to derive pointwise estimates in exponentially weighted spaces of bounded continuous functions is
\begin{align}
  \exists\,\nu\in\R\;\exists\,\tilde{C}_{\theta}>0:\;\theta(x)\geqslant \tilde{C}_{\theta}e^{\nu|x|}\;\forall\,x\in\R^d.     \tag{W7}\label{equ:WeightFunctionProp7}
\end{align}

Standard examples are 
\begin{align*}
  \theta_1(x)=\exp\left(-\mu|x|\right) \quad\text{and}\quad \theta_2(x)=\cosh\left(\mu|x|\right)
\end{align*}
as well as their smooth analogs 
\begin{align*}
  \theta_3(x)=\exp\left(-\mu\sqrt{\left|x\right|^2+1}\right)\quad\text{and}\quad\theta_4(x)=\cosh\left(\mu\sqrt{\left|x\right|^2+1}\right)
\end{align*}
for $\mu\in\R$ and $x\in\R^d$. Obviously, all these functions are radial weight functions of exponential growth rate $\eta=|\mu|$ with $C_{\theta}=1$. 
In particular, they satisfy \eqref{equ:WeightFunctionProp4} and \eqref{equ:WeightFunctionProp5}. The $C^1$-weight functions $\theta_3$ and $\theta_4$ 
even satisfy \eqref{equ:WeightFunctionProp6}. 

The examples show that the conditions \eqref{equ:WeightFunctionProp1}--\eqref{equ:WeightFunctionProp2} allow both increasing and decreasing weight functions. 
In some cases we explicitly need lower exponential bounds of the weight function as in \eqref{equ:WeightFunctionProp7}, see Section \ref{sec:TheOrnsteinUhlenbeckSemigroupInCrub}.

We now introduce the exponentially weighted Lebesgue and Sobolev spaces via
\begin{align*}
  L_{\theta}^{p}(\R^d,\C^N)   :=& \{v\in L^1_{\mathrm{loc}}(\R^d,\C^N)\mid \left\|\theta v\right\|_{L^p(\R^d,\C^N)}<\infty\}, \\
  W_{\theta}^{k,p}(\R^d,\C^N) :=& \{v\in L^p_{\theta}(\R^d,\C^N)\mid D^{\beta}v\in L^p_{\theta}(\R^d,\C^N)\;\forall\,\left|\beta\right|\leqslant k\}
\end{align*}
with norms
\begin{align*}
  \left\|v\right\|_{L^p_{\theta}(\R^d,\C^N)} :=& \left\|\theta v\right\|_{L^p(\R^d,\C^N)} := \left(\int_{\R^d}\left|\theta(x)v(x)\right|^p dx\right)^{\frac{1}{p}},\\
  \left\|v\right\|_{W^{k,p}_{\theta}(\R^d,\C^N)} :=& \bigg(\sum_{0\leqslant |\beta|\leqslant k}\left\|D^{\beta}v\right\|_{L^p_{\theta}(\R^d,\C^N)}^p\bigg)^{\frac{1}{p}}
\end{align*}
for every $1\leqslant p<\infty$, $k\in\N_0$ and multi-index $\beta\in\N_0^d$. In the proofs we sometimes abbreviate $\left\|\cdot\right\|_{L^p_{\theta}(\R^d,\C^N)}$ 
by $\left\|\cdot\right\|_{L^p_{\theta}}$. Similarly, we define the exponentially weighted space of bounded continuous functions via
\begin{align*}
  \Cbt(\R^d,\C^N)     :=& \left\{v\in C(\R^d,\C^N)\mid \left\|\theta v\right\|_{\Cb(\R^d,\C^N)}<\infty\right\}, \\
  \Cbt^k(\R^d,\C^N)   :=& \left\{v\in\Cbt(\R^d,\C^N)\mid D^{\beta}v\in\Cbt(\R^d,\C^N)\;\forall\,|\beta|\leqslant k\right\}
\end{align*}
with norms
\begin{align*}
  \left\|v\right\|_{\Cbt(\R^d,\C^N)}   :=& \left\|\theta v\right\|_{\Cb(\R^d,\C^N)} := \sup_{x\in\R^d}\left|\theta(x)v(x)\right|, \\
  \left\|v\right\|_{\Cbt^k(\R^d,\C^N)} :=& \max_{0\leqslant |\beta|\leqslant k}\left\|D^{\beta}v\right\|_{\Cbt(\R^d,\C^N)}
\end{align*}
for every $k\in\N_0$ and multi-index $\beta\in\N_0^d$. In the unweighted case with $\theta\equiv 1$ we omit the subindex $\theta$ at these function spaces and their 
corresponding norms. For the $\Cb$-semigroup theory, we additionally define the subspaces
\begin{align*}
  \Cub(\R^d,\C^N)  :=& \left\{v\in\Cb(\R^d,\C^N)\mid v\text{ in uniformly continuous on $\R^d$}\right\}, \\
  \Crub(\R^d,\C^N) :=& \,\big\{v\in\Cub(\R^d,\C^N)\mid \lim_{t\to 0}\sup_{x\in\R^d}\left|v(e^{-tS}x)-v(x)\right|=0\big\}.
\end{align*}
Note that $\Cub(\R^d,\C^N)$ and $\Crub(\R^d,\C^N)$, suggested in \cite{DaPratoLunardi1995}, are closed subspaces of $\Cb(\R^d,\C^N)$.

In Section \ref{sec:AHeatKernelMatrixForTheComplexOrnsteinUhlenbeckOperator} we investigate the heat kernel of $\L_{\infty}$. Assuming \eqref{cond:A8B}--\eqref{cond:A5} 
we show in Theorem \ref{thm:HeatKernel} that 
\begin{align}
  H(x,\xi,t)=(4\pi t A)^{-\frac{d}{2}}\exp\left(-Bt-(4tA)^{-1}\left|e^{tS}x-\xi\right|^2\right),\;x,\xi\in\R^d,\,t>0
\end{align}
is a heat kernel matrix of $\L_{\infty}$. 

In Section \ref{sec:SomePropertiesOfTheHeatKernel} we prove several heat kernel estimates which are essential not only for the semigroup theory in the next 
two sections, but also to solve the identification problem for $\L_{\infty}$ and to characterize its maximal domain, \cite{Otten2014}. The bounds of these 
estimates are sharp and depend on the quantities from \eqref{equ:aminamaxazerobzero}.

In Section \ref{sec:TheOrnsteinUhlenbeckSemigroupInLp} we study the Ornstein-Uhlenbeck semigroup in exponentially weighted complex-valued $L^p$-spaces 
$(L^p_{\theta}(\R^d,\C^N),\left\|\cdot\right\|_{L^p_{\theta}})$ for $1\leqslant p\leqslant\infty$. Via the heat kernel matrix of $\L_{\infty}$ 
the family of mappings $T(t):L^p_{\theta}\rightarrow L^p_{\theta}$, $t\geqslant 0$, is defined by
\begin{align}
  \left[T(t)v\right](x):= \begin{cases}
                              \int_{\R^d}H(x,\xi,t)v(\xi)d\xi &\text{, }t>0 \\
                              v(x) &\text{, }t=0
                            \end{cases}\quad ,x\in\R^d.
  \label{equ:OrnsteinUhlenbeckSemigroup}
\end{align}
We show in Theorem \ref{thm:OrnsteinUhlenbeckLpSemigroup} under the assumptions \eqref{cond:A8B}--\eqref{cond:A5} that $\left(T(t)\right)_{t\geqslant 0}$ 
generates a semigroup on $L^p_{\theta}(\R^d,\C^N)$ for every $1\leqslant p\leqslant\infty$ and for every weight functions satisfying 
\eqref{equ:WeightFunctionProp1}--\eqref{equ:WeightFunctionProp3}. The semigroup is even strongly continuous on $L^p_{\theta}(\R^d,\C^N)$ but only for 
$1\leqslant p<\infty$, provided that the weight function additionally satisfies \eqref{equ:WeightFunctionProp4}. This result is proved in Theorem 
\ref{thm:OrnsteinUhlenbeckLpStrongContinuity}. Introducing the infinitesimal generator $(A_{p,\theta},\D(A_{p,\theta}))$, which can be considered as the 
maximal realization of $\L_{\infty}$ in $L^p_{\theta}(\R^d,\C^N)$, we apply results from abstract semigroup theory from \cite{EngelNagel2000} and deduce in Corollary 
\ref{cor:OrnsteinUhlenbeckLpSolvabilityUniqueness} that the resolvent equation $(\lambda I-A_{p,\theta})v=g$ has a unique solution $v_{\star}\in\D(A_{p,\theta})$. 
In Theorem \ref{thm:OrnsteinUhlenbeckAPrioriEstimatesInLpConstantCoefficients} we finally prove exponentially weighted resolvent estimates. In particular, 
the estimates imply 
\begin{align*}
  \D(A_{p,\theta})\subseteq W^{1,p}_{\theta}(\R^d,\C^N)\quad\text{for}\quad 1\leqslant p<\infty.
\end{align*}

In Section \ref{sec:TheOrnsteinUhlenbeckSemigroupInCrub} we extend the results from Section \ref{sec:TheOrnsteinUhlenbeckSemigroupInLp} to the space of bounded 
continuous functions. For this purpose we consider the family of mappings $(T(t))_{t\geqslant 0}$ from \eqref{equ:OrnsteinUhlenbeckSemigroup} on the (complex-valued) 
Banach space $(\Cb(\R^d,\C^N),\left\|\cdot\right\|_{\Cb})$. Assuming \eqref{cond:A8B}--\eqref{cond:A5} we prove in Theorem \ref{thm:OrnsteinUhlenbeckLpSemigroup} 
that $\left(T(t)\right)_{t\geqslant 0}$ generates a semigroup on $\Cb(\R^d,\C^N)$ as well as on its closed subspaces $\Cub(\R^d,\C^N)$ and $\Crub(\R^d,\C^N)$. 
In Theorem \ref{thm:OrnsteinUhlenbeckCrubStrongContinuity} we show that $(T(t))_{t\geqslant 0}$ is strongly continuous but only on the subspace $\Crub(\R^d,\C^N)$. 
Indeed, the semigroup $\left(T(t)\right)_{t\geqslant 0}$ is discontinuous on $\Cb(\R^d,\C^N)$, weakly continuous on $\Cub(\R^d,\C^N)$ and strongly continuous on 
$\Crub(\R^d,\C^N)$. For the scalar real-valued case this is well-known from \cite{Cerrai1994} and \cite{DaPratoLunardi1995}, where the Ornstein-Uhlenbeck semigroup 
is studied in the unweighted space of bounded continuous functions $\Cb(\R^d,\R)$. Again, introducing the infinitesimal generator $(A_{\mathrm{b}},\D(A_{\mathrm{b}}))$, 
we deduce in Corollary \ref{cor:OrnsteinUhlenbeckCrubSolvabilityUniqueness} that the resolvent equation $(\lambda I-A_{\mathrm{b}})v=g$ admits a unique solution 
$v_{\star}\in\D(A_{\mathrm{b}})$. In Theorem \ref{thm:OrnsteinUhlenbeckAPrioriEstimatesInCrub} we finally prove exponentially weighted resolvent estimates in 
$\Cbt(\R^d,\C^N)$ and similar to the $L^p$-case the estimates imply 
\begin{align*}
  \D(A_{\mathrm{b}})\subseteq\Crub(\R^d,\C^N)\cap \Cub^1(\R^d,\C^N).
\end{align*}
It remains an open problem whether in addition to the resolvent estimates the semigroup theory extends to $\Cbt(\R^d,\C^N)$.

%
%
\sect{A heat kernel matrix for the complex Ornstein-Uhlenbeck operator}
\label{sec:AHeatKernelMatrixForTheComplexOrnsteinUhlenbeckOperator}

The following theorem provides a heat kernel matrix of $\L_{\infty}$ from \eqref{equ:Linfty}, which enables us to introduce the corresponding 
semigroup in Section \ref{sec:TheOrnsteinUhlenbeckSemigroupInLp}. It is motivated by the formal derivation of this kernel in the scalar real-valued 
case from \cite{Aarao2007}, \cite{Beals1999} and \cite[Section 13.2]{CalinChangFurutaniIwasaki2011}. The extension to the complex-valued system case 
is based on \cite[Theorem 4.2-4.4]{Otten2014}.

\begin{theorem}[Heat kernel for the Ornstein-Uhlenbeck operator]\label{thm:HeatKernel}
  Let the assumptions \eqref{cond:A8B}--\eqref{cond:A5} be satisfied, then the function 
  $H:\R^d\times\R^d\times]0,\infty[\rightarrow\C^{N,N}$ defined by
  \begin{align}
    \label{equ:HeatKernel}
    H(x,\xi,t)=(4\pi t A)^{-\frac{d}{2}}\exp\left(-Bt-(4tA)^{-1}\left|e^{tS}x-\xi\right|^2\right)
  \end{align}
  is a heat kernel of $\L_{\infty}$ given by
  \begin{align}
    \label{equ:Linfty2}
    \left[\L_{\infty}v\right](x):=A\triangle v(x)+\left\langle Sx,\nabla v(x)\right\rangle-Bv(x).
  \end{align}
\end{theorem}

\begin{proof}
  It is sufficient to prove the result for the complex scalar case, where we write $\alpha$ and $\delta$ instead of $A$ and $B$, respectively. For diagonal matrices $A,B\in\C^{N,N}$ 
  the components of $\L_{\infty}$ are decoupled and we obtain a heat kernel of the $k$-th component of $\L_{\infty}$ from the scalar case. In the general case, where $A$ and $B$ are 
  simultaneously diagonalizable with transformation matrix $Y\in\C^{N,N}$, we introduce the diagonalized operator $\tilde{\L}_{\infty}:=Y^{-1}\L_{\infty}Y$, for which we have a 
  heat kernel $\tilde{H}(x,\xi,t)$ from the diagonal case. We then finally deduce, that $H(x,\xi,t):=Y\tilde{H}(x,\xi,t)Y^{-1}$ is a heat kernel for $\L_{\infty}$.

  In the scalar case we must show that
  \begin{align}
    \label{equ:ScalarHeatKernel}
    H(x,\xi,t)=\left(4\pi\alpha t\right)^{-\frac{d}{2}}\exp\left(-\delta t-\left(4\alpha t\right)^{-1}\left|e^{tS}x-\xi\right|^2\right)
  \end{align}
  is a heat kernel of $\L_{\infty}$ given by
  \begin{align}
    \label{equ:LinftyScalar}
    \left[\L_{\infty}v\right](x):=\alpha\triangle v(x)+\left\langle Sx,\nabla v(x)\right\rangle-\delta v(x).
  \end{align}
  Before we verify that the heat kernel from \eqref{equ:ScalarHeatKernel} satisfies the properties \eqref{equ:H1}--\eqref{equ:H3} 
  we discuss a formal derivation of this kernel. To compute the heat kernel \eqref{equ:ScalarHeatKernel} of \eqref{equ:LinftyScalar} 
  we generalize the approach from \cite{Aarao2007} and \cite{Beals1999} to the complex case and use the complexified ansatz
  \begin{align}
    \label{equ:ansatz}
    H(x,\xi,t)=
    \varphi(t)\cdot\exp\left(-\frac{1}{2}\left\langle M(t)\left(\begin{array}{c}x\\\xi\end{array}\right), 
    \left(\begin{array}{c}x\\\xi\end{array}\right)\right\rangle\right),
  \end{align}
  where
  \begin{align*}
    \varphi:\R_+^*\rightarrow\C\text{, }t\mapsto\varphi(t)\quad\quad\text{and}\quad\quad
    M:\R_+^*\rightarrow\C^{2d,2d}\text{, }t\mapsto M(t)
  \end{align*}
  have to be determined and $\left\langle u,v\right\rangle:=\overline{u}^Tv$ denotes the Euclidean inner product on $\C^{2d}$.
  Note that it is sufficient to determine the symmetric part of the complex-valued matrix $M$ which we denote by $N$, i.e.
  \begin{align*}
    &N:\R_+^*\rightarrow\C^{2d\times 2d}\text{, }t\mapsto N(t):=\frac{1}{2}\left(M(t)+M^T(t)\right)
                                                               =\left(\begin{array}{cc}A(t)&B(t)\\C(t)&D(t)\end{array}\right)\text{,} \\
    &A,B,C,D:\R_+^*\rightarrow\C^{d\times d}\text{, }t\mapsto A(t),B(t),C(t),D(t).
  \end{align*}
  We emphasize that $N$ is a symmetric but in general not a Hermitian matrix. In particular, $A$ and $D$ are symmetric and $B^T=C$. For $x,\xi\in\R^d$ 
  we obtain from \eqref{equ:ansatz} and the definition of $N$
  \begin{align*}
            H(x,\xi,t)
    &= \varphi(t)\cdot\exp\left(-\frac{1}{2}\left\langle N(t)\left(\begin{array}{c}x\\\xi\end{array}\right), 
             \left(\begin{array}{c}x\\\xi\end{array}\right)\right\rangle\right).
  \end{align*}
  Since the heat kernel must satisfy \eqref{equ:H2} we introduce the extended matrices
  \begin{align*}
    \tilde{P}=\left(\begin{array}{cc}I_d &0\\0&0\end{array}\right)\text{, }
    \tilde{S}=\left(\begin{array}{cc}S &0\\0&0\end{array}\right)\in\R^{2d,2d}
  \end{align*}
  and obtain from the general Leibniz rule, the chain rule and the symmetry of $N$
  \pagebreak
  \begin{align*}
       0
    =& H(x,\xi,t)\left[\frac{\varphi_t(t)}{\varphi(t)}+\alpha\,\mathrm{tr}\left(\,\overline{A}(t)\right)+\delta\right. \\
     & \left.+\left\langle\left(-\frac{1}{2}N_t(t)-\overline{\alpha}N(t)\tilde{P}N(t)
       +\frac{1}{2}\tilde{S}^TN(t)+\frac{1}{2}N(t)\tilde{S}\right)\left(\begin{array}{c}x\\\xi\end{array}\right),
       \left(\begin{array}{c}x\\\xi\end{array}\right)\right\rangle\right].
  \end{align*}
  Thus, the kernel satisfies \eqref{equ:H2} if the following differential equations hold
  \begin{align}
    \label{equ:phiequ}
    \varphi_t(t)&=-\left(\alpha\,\mathrm{tr}\left(\overline{A}(t)\right)+\delta\right)\varphi(t) &&\text{, }t>0, \\
    \label{equ:Nequ}
    N_t(t)&=-2\overline{\alpha}N(t)\tilde{P}N(t)+\tilde{S}^TN(t)+N(t)\tilde{S} &&\text{, }t>0.
  \end{align}
  Since \eqref{equ:phiequ} depends on the solution of \eqref{equ:Nequ}, we will first solve the matrix-Riccati equation \eqref{equ:Nequ}, see 
  \cite[Section 3.1]{Abou-KandilFreilingIonescuJank2003}. It is obvious that the solutions of \eqref{equ:phiequ} and \eqref{equ:Nequ} are not 
  unique but one can select appropriate initial values, see \cite{Aarao2007} and \cite{Beals1999}.
  Using similar techniques and by some computations one can deduce that, \cite{Otten2014},
  \begin{align}
    N(t) = \frac{1}{2\overline{\alpha}t}\left(\begin{array}{cc}I_d &-\exp(tS^T)\\-\exp(tS) &I_d\end{array}\right). \label{equ:NSol}
  \end{align}
  Here, one has to keep in mind that \eqref{equ:H3} should be satisfied. 
  Thus, $\mathrm{tr}\left(\overline{A}(t)\right)=\frac{d}{2\alpha t}$ and \eqref{equ:phiequ} can be written as
  \begin{align*}
      \varphi_t(t)
    = -\left(\alpha\,\trace\left(\overline{A}(t)\right)+\delta\right)\varphi(t)
    = -\left(\frac{d}{2t}+\delta\right)\varphi(t).
  \end{align*}
  Hence, the general solution of \eqref{equ:phiequ} is given by
  \begin{align}
      \varphi(t) 
    = C\exp\left(-\int\left(\frac{d}{2t}+\delta\right)dt\right)
    = C\exp\left(-\frac{d}{2}\ln(t)-\delta t\right)
    = C t^{-\frac{d}{2}}e^{-\delta t},
    \label{equ:phiSol}
  \end{align}
  where $C\in\C$. Below we choose $C\in\C$ such that the normalization condition 
  \begin{align}
    \lim_{t\downarrow 0}\int_{\R^d}H(x,\xi,t)d\xi=1\quad\forall\,x\in\R^d
    \label{equ:normalization}
  \end{align}
  holds. First note that from
  \begin{align*}
      \left\langle \frac{1}{2\overline{\alpha}t}\left(\begin{array}{cc}I_d &-\exp(tS^T)\\-\exp(tS) &I_d\end{array}\right)\left(\begin{array}{c}x\\\xi\end{array}\right),
      \left(\begin{array}{c}x\\\xi\end{array}\right)\right\rangle
    = \frac{1}{2\alpha t}\left|e^{tS}x-\xi\right|^2
  \end{align*}
  we obtain
  \begin{align*}
       H(x,\xi,t)
    =& Ct^{-\frac{d}{2}}e^{-\delta t-\frac{1}{4\alpha t}\left|e^{tS}x-\xi\right|^2}.
  \end{align*}
  Now, integrating over $\R^d$ w.r.t. $\xi$, we obtain from the transformation theorem and assumption \eqref{cond:A2}
  \begin{align*}
     & \int_{\R^d}H(x,\xi,t)d\xi
    =  Ct^{-\frac{d}{2}}e^{-\delta t}\int_{\R^d}e^{-\frac{1}{4\alpha t}\left|e^{tS}x-\xi\right|^2}d\xi \\
    =&  Ct^{-\frac{d}{2}}e^{-\delta t}\int_{\R^d}e^{-\frac{1}{4\alpha t}\left|x-\psi\right|^2}d\psi
    =  Ct^{-\frac{d}{2}}e^{-\delta t}\prod_{j=1}^{d}\int_{-\infty}^{\infty}e^{-\frac{1}{4\alpha t}x_j^2}dx_j \\
    =&  Ct^{-\frac{d}{2}}e^{-\delta t}\left(4\pi\alpha t\right)^{\frac{d}{2}}
    =  C\left(4\pi\alpha\right)^{\frac{d}{2}}e^{-\delta t}\overset{t\to 0}{\rightarrow} C\left(4\pi\alpha\right)^{\frac{d}{2}}\overset{!}{=}1.
  \end{align*}
  Hence, we choose $C=\left(4\pi\alpha\right)^{-\frac{d}{2}}$ such that \eqref{equ:normalization} is satisfied. Here $\alpha^{-\frac{d}{2}}$ denotes the principal 
  root (main branch) of $\alpha^{-d}$. Finally, we obtain the heat kernel \eqref{equ:ScalarHeatKernel} from \eqref{equ:NSol} and \eqref{equ:phiSol}. The 
  properties \eqref{equ:H1} and \eqref{equ:H2} follow directly from the construction of the heat kernel. It remains to verify property \eqref{equ:H3}. For this 
  we use the integral
  \begin{align}
    \label{equ:IntegrationFormula}
    \int_{0}^{\infty}r^{n-1}e^{-z r^2}dr=\frac{z^{-\frac{n}{2}}}{2\Gamma\left(\frac{n}{2}\right)},
  \end{align}
  which holds for $n\in\C$ with $\Re n>0$ and $z\in\C$ with $\Re z>0$. Using the transformation theorem (with transformations for $d$-dimensional 
  polar coordinates and $\Phi(\xi)=2^{-1}t^{-\frac{1}{2}}\left(e^{tS}x-\xi\right)$) and formula \eqref{equ:IntegrationFormula} (with $n=d$ and $z=\alpha^{-1}$) we obtain, 
  similarly to the proof of \cite[Proposition 3.4.1]{CalinChangFurutaniIwasaki2011}, for every $\phi\in C_{\mathrm{c}}^{\infty}(\R^d,\C)$
  \begin{align*}
     & \lim_{t\downarrow 0}H(x,\xi,t)(\phi)
    =  \lim_{t\downarrow 0}\int_{\R^d}H(x,\xi,t)\phi(\xi)d\xi \\
    =& \lim_{t\downarrow 0}\int_{\R^d}\left(4\pi\alpha t\right)^{-\frac{d}{2}}\exp\left(-\delta t-\left(4\alpha t\right)^{-1}\left|e^{tS}x-\xi\right|^2\right)\phi(\xi)d\xi \\
    =& \lim_{t\downarrow 0}\left(4\pi\alpha t\right)^{-\frac{d}{2}}\left(4t\right)^{\frac{d}{2}}\int_{\R^d}\exp\left(-\delta t-\alpha^{-1}\left|\psi\right|^2\right)\phi(e^{tS}x-2t^{\frac{1}{2}}\psi)d\psi \\
    =& \left(\pi\alpha\right)^{-\frac{d}{2}}\int_{\R^d}\exp\left(-\alpha^{-1}\left|\psi\right|^2\right)d\psi \phi(x) \\
    =& \left(\pi\alpha\right)^{-\frac{d}{2}} 2\pi^{\frac{d}{2}}\Gamma\left(\frac{d}{2}\right)\int_{0}^{\infty}r^{d-1}e^{-\alpha^{-1}r^2}dr \phi(x) \\
    =& \left(\pi\alpha\right)^{-\frac{d}{2}} 2\pi^{\frac{d}{2}}\Gamma\left(\frac{d}{2}\right) \frac{\alpha^{\frac{d}{2}}}{2\Gamma\left(\frac{d}{2}\right)} \phi(x)
    =  \phi(x)
    = \delta_{x}(\xi)(\phi).
  \end{align*}
  Note that $\Re z=\Re\left(\alpha^{-1}\right)=\frac{\Re\overline{\alpha}}{\left|\alpha\right|^2}=\frac{\Re\alpha}{\left|\alpha\right|^2}>0$ is true 
  by assumption \eqref{cond:A2}.
\end{proof}

We conclude this section with several remarks concerning generalizations and extensions.

\textbf{Simultaneous diagonalization of $A$ and $B$.}
We stress that condition \eqref{cond:A8B} is crucial in Theorem \ref{thm:HeatKernel}. For arbitrary matrices $A,B\in\C^{N,N}$, where only the matrix $A$ is diagonalizable 
and satisfies \eqref{cond:A2}, the heat kernel of \eqref{equ:Linfty2} is in general not given by \eqref{equ:HeatKernel}. In this case we can expect at most a series 
representation for the heat kernel. This seems to be an open problem. 

\textbf{Ellipticity assumption.}
Assumption \eqref{cond:A2} in Theorem \ref{thm:HeatKernel} states that $\L_{\infty}$ is an elliptic differential operator. Using the weaker assumption $\Re\sigma(A)\geqslant 0$, 
which includes coupled parabolic-hyperbolic differential operators, no heat kernel representation seems to be known.

\textbf{Generalized heat kernel ansatz.}
For the computation of heat kernels for more general differential operators, Beals used in \cite[(2)]{Beals1999} the generalized ansatz
\begin{align*}
  H(x,\xi,t)=\varphi(t)\exp\left(-Q_t(x,\xi)\right),\,t>0,\,x,\xi\in\R^d,
\end{align*}
where $Q_t$ is a quadratic form of $2d$ variables, instead of the ansatz from \eqref{equ:ansatz}. This formula is motivated by the Trotter product formula and 
the Feynman-Kac formula, \cite[Section 2.8]{Mao2008}. Such a general ansatz was also used in \cite[(13.2.14)]{CalinChangFurutaniIwasaki2011} for the construction 
of heat kernels for degenerate elliptic operators.

\textbf{Generalized Ornstein-Uhlenbeck operator.}
Let the assumptions \eqref{cond:A8B}, \eqref{cond:A2}, $Q\in\R^{d,d}$, $Q>0$, $Q=Q^T$ and $S\in\R^{d,d}$ be satisfied and consider 
the generalized $N$-dimensional complex-valued Ornstein-Uhlenbeck operator
\begin{align*}
  \left[\L_{\mathrm{OU}}v\right](x)=&A\,\trace\left(Q D^2v(x)\right)+\left\langle Sx,\nabla v(x)\right\rangle-B v(x) \\
                                   =&A\sum_{i=1}^{d}\sum_{j=1}^{d}Q_{ij}D_{i}D_{j}v(x)+\sum_{i=1}^{d}\sum_{j=1}^{d}S_{ij}x_jD_iv(x)-B v(x),\,x\in\R^d.
\end{align*}
Then one can show that
\begin{align*}
  H(x,\xi,t)=\left(4\pi A\right)^{-\frac{d}{2}}\left(\det Q_t\right)^{-\frac{1}{2}}\exp\left(-Bt-\left(4 A\right)^{-1}\left\langle Q_t^{-1}(e^{tS}x-\psi),(e^{tS}x-\psi)\right\rangle\right)
\end{align*}
with
\begin{align*}
  Q_t = \int_{0}^{t}\exp\left(\tau S\right)Q\exp\left(\tau S^T\right) d\tau
\end{align*}
is a heat kernel of $\L_{\mathrm{OU}}$. This is true, even if \eqref{cond:A5} is not satisfied.

\textbf{Heat kernel via Fourier-Bessel method.}
The Fourier-Bessel method used in \cite[Section 3.2]{BeynLorenz2008} provides a further possibility to determine a heat kernel for $\L_{\infty}$ on $\R^2$. 
Therein one computes a Green's function of $\L_{\infty}$ and finds that Green's function equals the time integral over the heat kernel, see Remark \ref{rem:GreensFunctionLptheta} 
below and \cite[Section 1.6]{Otten2014}. An advantage of this method is that it can be extended to circular disks (bounded domains) with Dirichlet, Neumann and Robin boundary conditions.

%
%
\sect{Properties of the Ornstein-Uhlenbeck kernel}
\label{sec:SomePropertiesOfTheHeatKernel}

In this section we collect some basic properties of the heat kernel from Theorem \ref{thm:HeatKernel}. These properties are essential for the analysis of the semigroup, 
for the solvability of identification problems and for the evidence that the solution of the resolvent equation has exponential decay in space. 

The heat kernel satisfies the following Chapman-Kolmogorov formula, which plays an important role for the semigroup properties, \cite[Proposition C.3.2]{LorenziBertoldi2007}.

\begin{lemma}[Chapman-Kolmogorov formula]\label{lem:ChapmanKolmogorovFormula}
  Let the assumptions \eqref{cond:A8B}--\eqref{cond:A5} be satisfied. Then the following equality is satisfied
  \begin{align*}
    \int_{\R^d}H(x,\tilde{\xi},t_1)H(\tilde{\xi},\xi,t_2)d\tilde{\xi} = H(x,\xi,t_1+t_2)\quad\forall\,x,\xi\in\R^d,\,\forall\,t_1,t_2>0.
  \end{align*}
\end{lemma}

\begin{remark}
  For the proof we need the following integral
  \begin{align}
    \label{equ:integralformula}
    \begin{split}
      &\int_{-\infty}^{\infty}\exp\left(-c_1\left(a-\psi\right)^2-c_2\left(\psi-b\right)^2\right)d\psi \\
      &\quad\quad\quad\quad\quad\quad\quad = \left(\frac{\pi}{c_1+c_2}\right)^{\frac{1}{2}}\exp\left(-\frac{c_1 c_2}{c_1 +c_2}\left(a-b\right)^2\right)
    \end{split}
  \end{align}
  for $a,b,c_1,c_2\in\C$ with $\Re c_1>0$, $\Re c_2>0$.
\end{remark}

\begin{proof}
  First let us prove the assertion for the diagonalized kernel
  \begin{align*}
    \tilde{H}(x,\xi,t)=(4\pi t \Lambda_A)^{-\frac{d}{2}}\exp\left(-\Lambda_Bt-(4t\Lambda_A)^{-1}\left|e^{tS}x-\xi\right|^2\right),
  \end{align*}
  where $A=Y\Lambda_A Y^{-1}$ and $B=Y\Lambda_B Y^{-1}$. Condition \eqref{cond:A5} yields $\left|e^{tS}x\right|=\left|x\right|$, hence
  \begin{align} 
    \label{equ:SolKernIntProperty}
    \begin{split}
       & \int_{\R^d}\tilde{H}(x,\tilde{\xi},t_1)\tilde{H}(\tilde{\xi},\xi,t_2)d\tilde{\xi} \\
      =& \left(4\pi t_1 \Lambda_A\right)^{-\frac{d}{2}} \left(4\pi t_2 \Lambda_A\right)^{-\frac{d}{2}} \exp\left(-\Lambda_B(t_1+t_2)\right) \\
       & \cdot\int_{\R^d}\exp\left(-\left(4 t_1 \Lambda_A\right)^{-1}\left|e^{t_1S}x-\tilde{\xi}\right|^2-\left(4 t_2 \Lambda_A\right)^{-1}\left|\tilde{\xi}-e^{-t_2S}\xi\right|^2\right) d\tilde{\xi}.
    \end{split}
  \end{align}
  From \eqref{cond:A2} we deduce that $\Re\lambda_j^A>0$ and hence $\Re\left(\lambda_j^A\right)^{-1}=\Re\frac{\overline{\lambda_j^A}}{\left|\lambda_j^A\right|^2}>0$ for every $j=1,\ldots,N$. Using formula 
  \eqref{equ:integralformula} componentwise with $c_1=\left(4 t_1 \lambda_j^A\right)^{-1}$, $c_2=\left(4 t_2 \lambda_j^A\right)^{-1}$, $\psi=\tilde{\xi}_{i}$, $a=\left(e^{t_1S}x\right)_i$, $b=\left(e^{-t_2S}\xi\right)_i$,
  $i=1,\dots,d$ we obtain
  \begin{align*}
       &\int_{-\infty}^{\infty}\exp\left(-\left(4 t_1 \Lambda_A\right)^{-1}\left(\left(e^{t_1S}x\right)_i-\tilde{\xi}_i\right)^2-\left(4 t_2 \Lambda_A\right)^{-1}\left(\tilde{\xi}_i-\left(e^{-t_2S}\xi\right)_i\right)^2\right) d\tilde{\xi}_i\\
      =&\left(4\pi t_1 \Lambda_A\right)^{\frac{1}{2}} \left(4\pi t_2 \Lambda_A\right)^{\frac{1}{2}} \left(4\pi \left(t_1+t_2\right) \Lambda_A\right)^{-\frac{1}{2}} \\
       &\cdot\exp\left(-\left(4 \left(t_1+t_2\right) \Lambda_A\right)^{-1}\left(\left(e^{t_1S}x\right)_i-\left(e^{-t_2S}\xi\right)_i\right)^2\right).
  \end{align*}
  Using this integral and $\left|e^{tS}x\right|=\left|x\right|$ we can compute the latter integral in \eqref{equ:SolKernIntProperty}
  \begin{align*}
     &\int_{\R^d}\exp\left(-\left(4 t_1 \Lambda_A\right)^{-1}\left|e^{t_1S}x-\tilde{\xi}\right|^2-\left(4 t_2 \Lambda_A\right)^{-1}\left|\tilde{\xi}-e^{-t_2 S}\xi\right|^2\right) d\tilde{\xi} \\
    =&\int_{-\infty}^{\infty}\cdots\int_{-\infty}^{\infty}\exp\bigg(\sum_{i=1}^{d}\bigg[-\left(4 t_1 \Lambda_A\right)^{-1}\left((e^{t_1S}x)_i-\tilde{\xi}_i\right)^2 \\
     &\quad\quad\quad\quad                                                              -\left(4 t_2 \Lambda_A\right)^{-1}\left(\tilde{\xi}_i-(e^{-t_2S}\xi)_i\right)^2\bigg]\bigg)d\tilde{\xi}_{1}\cdots d\tilde{\xi}_{d} \\
    =&\int_{-\infty}^{\infty}\cdots\int_{-\infty}^{\infty}\prod_{i=1}^{d}\exp\bigg(-\left(4 t_1 \Lambda_A\right)^{-1}\left((e^{t_1S}x)_i-\tilde{\xi}_i\right)^2 \\
     &\quad\quad\quad\quad                                                         -\left(4 t_2 \Lambda_A\right)^{-1}\left(\tilde{\xi}_i-(e^{-t_2S}\xi)_i\right)^2\bigg)d\tilde{\xi}_{1}\cdots d\tilde{\xi}_{d} \\
    =&\prod_{i=1}^{d}\int_{-\infty}^{\infty}\exp\left(-\left(4 t_1 \Lambda_A\right)^{-1}\left((e^{t_1S}x)_i-\tilde{\xi}_i\right)^2-\left(4 t_2 \Lambda_A\right)^{-1}\left(\tilde{\xi}_i-(e^{-t_2S}\xi)_i\right)^2\right)d\tilde{\xi}_{i} \\
    =&\left(4\pi t_1 \Lambda_A\right)^{\frac{d}{2}} \left(4\pi t_2 \Lambda_A\right)^{\frac{d}{2}} \left(4\pi \left(t_1+t_2\right) \Lambda_A\right)^{-\frac{d}{2}} \\
     &\quad\quad\quad\quad\cdot\exp\left(-\left(4 \left(t_1+t_2\right) \Lambda_A\right)^{-1}\sum_{i=1}^{d}\left(\left(e^{t_1S}x\right)_i-\left(e^{-t_2S}\xi\right)_i\right)^2\right) \\
    =&\left(4\pi t_1 \Lambda_A\right)^{\frac{d}{2}} \left(4\pi t_2 \Lambda_A\right)^{\frac{d}{2}} \left(4\pi \left(t_1+t_2\right) \Lambda_A\right)^{-\frac{d}{2}} \\
     &\quad\quad\quad\quad\cdot\exp\left(-\left(4 \left(t_1+t_2\right) \Lambda_A\right)^{-1}\left|e^{\left(t_1+t_2\right)S}x-\xi\right|^2\right).
  \end{align*}
  Using this equality in \eqref{equ:SolKernIntProperty} we obtain
  \begin{align*}
     \int_{\R^d}\tilde{H}(x,\tilde{\xi},t_1)\tilde{H}(\tilde{\xi},\xi,t_2)d\tilde{\xi}
    =\tilde{H}(x,\xi,t_1+t_2)\quad\forall\,x,\xi\in\R^d,\,\forall\,t_1,t_2>0.
  \end{align*}
  We now consider the general case: Since $H(x,\xi,t)=Y\tilde{H}(x,\xi,t)Y^{-1}$ we obtain
  \begin{align*}
     & \int_{\R^d}H(x,\tilde{\xi},t_1)H(\tilde{\xi},\xi,t_2)d\tilde{\xi}
    = Y\int_{\R^d}\tilde{H}(x,\tilde{\xi},t_1)\tilde{H}(\tilde{\xi},\xi,t_2)d\tilde{\xi}Y^{-1} \\
    =& Y\tilde{H}(x,\xi,t_1+t_2)Y^{-1} = H(x,\xi,t_1+t_2)\quad\forall\,x,\xi\in\R^d,\,\forall\,t_1,t_2>0.
  \end{align*}
\end{proof}

The first two partial derivatives of $H$ with respect to $x$, i.e. $D_i:=\frac{\partial}{\partial x_i}$, are
\begin{align*}
      D_i H(x,\xi,t) =& -(2 t A)^{-1}\left\langle e^{tS}x-\xi,e^{tS}e_i\right\rangle H(x,\xi,t), \\
  D_j D_i H(x,\xi,t) =& \left(-\left(2 t A\right)^{-1}\delta_{ij}+\left(2 t A\right)^{-2}\left\langle e^{tS}x-\xi,e^{tS}e_i\right\rangle\left\langle e^{tS}x-\xi,e^{tS}e_j\right\rangle\right) \\
                      & \cdot H(x,\xi,t)
\end{align*}
for $i,j=1,\ldots,d$, where we used \eqref{cond:A5} once more. Let us define the radial kernels
\begin{align}
       \tilde{K}(\psi,t) :=& \left(4\pi t \Lambda_A\right)^{-\frac{d}{2}}\exp\left(-\Lambda_B t-\left(4 t \Lambda_A\right)^{-1}\left|\psi\right|^2\right), \label{equ:Ktilde}\\
               K(\psi,t) :=& H(x,e^{tS}x-\psi,t) = Y\tilde{K}(\psi,t)Y^{-1}\label{equ:K}\\
                          =& \left(4\pi t A\right)^{-\frac{d}{2}}\exp\left(-B t-\left(4 t A\right)^{-1}\left|\psi\right|^2\right), \nonumber\\
     \tilde{K}^i(\psi,t) :=& -\left(2 t \Lambda_A\right)^{-1}\left\langle\psi,e^{tS}e_i\right\rangle \tilde{K}(\psi,t), \label{equ:Kitilde}\\
             K^i(\psi,t) :=& \left[D_i H(x,\xi,t)\right]_{\xi=e^{tS}x-\psi} = Y\tilde{K}^i(\psi,t)Y^{-1} \label{equ:Ki}\\
                          =& -\left(2 t A\right)^{-1}\left\langle\psi,e^{tS}e_i\right\rangle K(\psi,t), \nonumber\\
  \tilde{K}^{ji}(\psi,t) :=& \left(\left(2 t \Lambda_A\right)^{-2}\left\langle\psi,e^{tS}e_i\right\rangle\left\langle\psi,e^{tS}e_j\right\rangle-\left(2 t\Lambda_A\right)^{-1}\delta_{ij}\right)\tilde{K}(\psi,t), \label{equ:Kjitilde}\\
          K^{ji}(\psi,t) :=& \left[D_j D_i H(x,\xi,t)\right]_{\xi=e^{tS}x-\psi} = Y\tilde{K}^{ji}(\psi,t)Y^{-1}\label{equ:Kji}\\
                          =& \left(\left(2 t A\right)^{-2}\left\langle\psi,e^{tS}e_i\right\rangle\left\langle\psi,e^{tS}e_j\right\rangle-\left(2 t A\right)^{-1}\delta_{ij}\right)K(\psi,t). \nonumber
\end{align}

To prove boundedness of the associated semigroup in exponentially weighted function spaces, we need upper bounds of the exponentially weighted integrals of 
the kernels $K$, $K^i$ and $K^{ji}$.

\begin{lemma}\label{lem:PropK}
  Let the assumptions \eqref{cond:A8B}--\eqref{cond:A5} be satisfied, $p,\eta\in\R$ and let $K$, $K^i$, $K^{ji}$ be given by 
  \eqref{equ:K}, \eqref{equ:Ki} and \eqref{equ:Kji} for every $i,j=1,\ldots,d$, then
  \begin{align*}
    &\textrm{\enum{1} }\int_{\R^d}e^{\eta p|\psi|}\left|K(\psi,t)\right| d\psi\leqslant C_1(t)                       &&,\,t>0, \\
    &\textrm{\enum{2} }\int_{\R^d}e^{\eta p|\psi|}\left|K^i(\psi,t)\right| d\psi\leqslant C_2(t)                     &&,\,t>0, \\
    &\textrm{\enum{3} }\int_{\R^d}e^{\eta p|\psi|}\left|K^{ji}(\psi,t)\right| d\psi\leqslant C_3(t)                  &&,\,t>0,
  \end{align*}
  where $\left|\cdot\right|$ denotes the spectral norm and the functions are given by
  \begin{align*}
    C_1(t) =& \kappa \aone^{\frac{d}{2}}e^{-\bzero t}\bigg[{}_1F_1\left(\frac{d}{2};\frac{1}{2};\nu t\right)
              +2\frac{\Gamma\left(\frac{d+1}{2}\right)}{\Gamma\left(\frac{d}{2}\right)}\left(\nu t\right)^{\frac{1}{2}}{}_1F_1\left(\frac{d+1}{2};\frac{3}{2};\nu t\right)\bigg], \\
    C_2(t) =& \kappa \aone^{\frac{d+1}{2}}e^{-\bzero t}\left(t\amin\right)^{-\frac{1}{2}}\bigg[\frac{\Gamma\left(\frac{d+1}{2}\right)}{\Gamma\left(\frac{d}{2}\right)}
               {}_1F_1\left(\frac{d+1}{2};\frac{1}{2};\nu t\right) \\
            & +2\frac{\Gamma\left(\frac{d+2}{2}\right)}{\Gamma\left(\frac{d}{2}\right)}\left(\nu t\right)^{\frac{1}{2}}{}_1F_1\left(\frac{d+2}{2};\frac{3}{2};\nu t\right)\bigg], \\
    C_3(t) =& \kappa \aone^{\frac{d+2}{2}} e^{-\bzero t} \left(t\amin\right)^{-1}\bigg[
              \frac{\Gamma\left(\frac{d+2}{2}\right)}{\Gamma\left(\frac{d}{2}\right)}{}_1F_1\left(\frac{d+2}{2};\frac{1}{2};\nu t\right) \\
            & +2\frac{\Gamma\left(\frac{d+3}{2}\right)}{\Gamma\left(\frac{d}{2}\right)}\left(\nu t\right)^{\frac{1}{2}}{}_1F_1\left(\frac{d+3}{2};\frac{3}{2};\nu t\right)
              +\frac{\delta_{ij}}{2}\aone^{-1}{}_1F_1\left(\frac{d}{2};\frac{1}{2};\nu t\right) \\
            & +\delta_{ij}\aone^{-1}\frac{\Gamma\left(\frac{d+1}{2}\right)}{\Gamma\left(\frac{d}{2}\right)}\left(\nu t\right)^{\frac{1}{2}}
              {}_1F_1\left(\frac{d+1}{2};\frac{3}{2};\nu t\right)\bigg]
  \end{align*}
  with $\aone:=\frac{\amax^2}{\amin\azero}\geqslant 1$, $\nu:=\frac{\amax^2 \eta^2 p^2}{\azero}\geqslant 0$, $\kappa:=\cond(Y)$ with $Y$ from \eqref{cond:A8B} and 
  $\amax,\amin,\azero,\bzero$ defined in \eqref{equ:aminamaxazerobzero}. Note that $C_{1+\left|\beta\right|}(t)\sim t^{\frac{d-1}{2}} e^{-(\bzero-\nu)t}$ as $t\to\infty$ and 
  $C_{1+\left|\beta\right|}(t)\sim t^{-\frac{\left|\beta\right|}{2}}$ as $t\to 0$ for every $\left|\beta\right|=0,1,2$.
\end{lemma}

\begin{remark}\label{rem:1F1}
  The function ${}_1F_1(a;b;z)$ denotes the Kummer confluent hypergeometric function $M(a,b,z)$ and satisfies the formula
  \begin{align}
    \label{equ:1F1Formula1}
    \begin{split}
      \int_{0}^{\infty}s^n e^{-s^2+rs}ds =& \frac{1}{2}\Gamma\left(\frac{n+1}{2}\right){}_1F_1\left(\frac{n+1}{2};\frac{1}{2};\frac{r^2}{4}\right) \\
                                          & +\frac{r}{2}\Gamma\left(\frac{n}{2}+1\right){}_1F_1\left(\frac{n}{2}+1;\frac{3}{2};\frac{r^2}{4}\right)
    \end{split}
  \end{align}
  for $r\in\R$ with $r\geqslant 0$ and $n\in\C$ with $\Re n>-1$, that we need to prove Lemma \ref{lem:PropK}. 
  Moreover, in Lemma \ref{lem:ImproperInt} we will need the connection formula
  \begin{align}
      {}_1F_1\left(a;b;x\right) = e^x{}_1F_1\left(b-a;b;-x\right) \label{equ:1F1Formula2}
  \end{align}
  for $a,b,x\in\C$ with $b\neq 0,-1,-2,\ldots$ (see \cite{OlverLozierBoisvertClark2010} 13.2.39) and the integral
  \begin{align}
      \int_{0}^{\infty}t^{\alpha-1}e^{-ct}{}_1F_1\left(a;b;-t\right)dt = c^{-\alpha}\Gamma\left(\alpha\right){}_2F_1\left(a,\alpha;b;-\frac{1}{c}\right) \label{equ:1F1Formula3}
  \end{align}
  for $a,b,c,\alpha\in\C$ with $b\neq 0,-1,-2,\ldots$, $\Re\alpha>0$ and $\Re c>0$ (see \cite{OlverLozierBoisvertClark2010} 16.5.3) where ${}_2F_1\left(a_1,a_2;b_1;z\right)$ 
  denotes the generalized hypergeometric function. To verify the asymptotic behavior of the function ${}_1F_1\left(a,b,z\right)$ at infinity we 
  need the limiting form
  \begin{align}
    {}_1F_1\left(a,b,z\right) \sim \frac{\Gamma(b)}{\Gamma(a)}z^{a-b}e^{z}\text{, as }z\to\infty\text{, }\left|\arg z\right|<\frac{\pi}{2} \label{equ:1F1Formula4}
  \end{align}
  for $z\in\C$ and $a,b\in\C\backslash\{0,-1,-2,\ldots\}$ (see \cite{OlverLozierBoisvertClark2010} 13.2.4 and 13.2.23). Observe that ${}_1F_1\left(a;b;0\right)=1$ and ${}_2F_1\left(a_1,a_2;b_1;0\right)=1$ 
  which induce a simplification of the constants in Lemma \ref{lem:PropK} in case of $\eta=0$.
\end{remark}

\begin{proof}
  First note that due to \eqref{cond:A8B}, \eqref{equ:K}, \eqref{equ:Ki}, \eqref{equ:Kji} it hold 
  \begin{align}
    \label{equ:SpectralNormK}
    \left|K^{\beta}(\psi,t)\right|=\left|Y\tilde{K}^{\beta}(\psi,t)Y^{-1}\right|\leqslant |Y||Y^{-1}|\left|\tilde{K}^{\beta}(\psi,t)\right|
    =\kappa\max_{k=1,\ldots,N}\left|\tilde{K}^{\beta}_{kk}(\psi,t)\right|,
  \end{align}
  for every multi-index $\beta\in\N_0^d$ with $\left|\beta\right|\leqslant 2$. Note that $\tilde{K}^{\beta}(\psi,t)\in\C^{N,N}$ is diagonal.\\
  \enum{1}: Using \eqref{equ:Ktilde} a simple computation shows that
  \begin{align}
    \label{equ:SpectralNormKtilde}
    \max_{k=1,\ldots,N}\left|\tilde{K}_{kk}(\psi,t)\right|\leqslant\left(4\pi t\amin\right)^{-\frac{d}{2}}e^{-\bzero t-\frac{\azero}{4t\amax^2}\left|\psi\right|^2}
  \end{align}  
  for every $\psi\in\R^d$ and $t>0$. From \eqref{equ:SpectralNormK} with $\left|\beta\right|=0$, \eqref{equ:SpectralNormKtilde}, the transformation theorem (with transformations for 
  $d$-dimensional polar coordinates and $\Phi(r)=\left(\frac{\azero}{4t\amax^2}\right)^{\frac{1}{2}}r$) and formula \eqref{equ:1F1Formula1} (since \eqref{cond:A2} is 
  satisfied) we obtain
  \begin{align*}
              &\int_{\R^d}e^{\eta p|\psi|}\left|K(\psi,t)\right| d\psi \\
    \leqslant &\int_{\R^d}\kappa e^{\eta p|\psi|}\left(4\pi t\amin\right)^{-\frac{d}{2}}e^{-\bzero t-\frac{\azero}{4t\amax^2}|\psi|^2}d\psi \\
            = &\kappa\left(4\pi t\amin\right)^{-\frac{d}{2}}e^{-\bzero t}\frac{2\pi^{\frac{d}{2}}}{\Gamma\left(\frac{d}{2}\right)}\int_{0}^{\infty}r^{d-1}e^{-\frac{\azero}{4t\amax^2}r^2+\eta p r} dr \\
            = &\kappa\left(\frac{\amax^2}{\amin\azero}\right)^{\frac{d}{2}}e^{-\bzero t}\frac{2}{\Gamma\left(\frac{d}{2}\right)}
               \int_{0}^{\infty}s^{d-1}e^{-s^2+\left(\frac{4\amax^2 \eta^2 p^2 t}{\azero}\right)^{\frac{1}{2}}s}ds 
            =  C_1(t).
  \end{align*}

  \noindent
  \enum{2}: Using \eqref{equ:Kitilde} for every $i=1,\ldots,d$, $\psi\in\R^d$ and $t>0$ we obtain
  \begin{align}
    \label{equ:SpectralNormKitilde}
    \max_{k=1,\ldots,N}\left|\tilde{K}^i_{kk}(\psi,t)\right|\leqslant\left(2t\amin\right)^{-1}\left|\left\langle\psi,e^{tS}e_i\right\rangle\right|\max_{k=1,\ldots,N}\left|\tilde{K}_{kk}(\psi,t)\right|.
  \end{align}
  From \eqref{equ:SpectralNormK} with $\left|\beta\right|=1$, \eqref{equ:SpectralNormKitilde} with \eqref{equ:SpectralNormKtilde}, Cauchy-Schwarz inequality with assumption \eqref{cond:A5} 
  ($\left|\left\langle\psi,e^{tS}e_i\right\rangle\right|\leqslant|\psi||e^{tS}e_i|=|\psi|$), the transformation theorem (with transformations from (1)) and formula \eqref{equ:1F1Formula1} 
  we obtain
  \begin{align*}
              &\int_{\R^d}e^{\eta p|\psi|}\left|K^i(\psi,t)\right| d\psi \\
    \leqslant &\int_{\R^d}\kappa e^{\eta p|\psi|}\left(2t\amin\right)^{-1}\left|\left\langle\psi,e^{tS}e_i\right\rangle\right|\max_{k=1,\ldots,N}\left|\tilde{K}_{kk}(\psi,t)\right| d\psi \\
    \leqslant &\int_{\R^d}\kappa e^{\eta p|\psi|}(2t\amin)^{-1}\left|\psi\right| (4\pi t\amin)^{-\frac{d}{2}}e^{-\bzero t-\frac{\azero}{4t\amax^2}|\psi|^2}d\psi \\
            = &\kappa (2t\amin)^{-1}(4\pi t\amin)^{-\frac{d}{2}}e^{-\bzero t}\frac{2\pi^{\frac{d}{2}}}{\Gamma\left(\frac{d}{2}\right)}
               \int_{0}^{\infty}r^d e^{-\frac{\azero}{4t\amax^2}r^2+\eta p r} dr \\
            = &\kappa\left(\frac{\amax^2}{\amin\azero}\right)^{\frac{d+1}{2}}e^{-\bzero t}\frac{2}{\Gamma\left(\frac{d}{2}\right)}\left(t\amin\right)^{-\frac{1}{2}}
               \int_{0}^{\infty}s^{d}e^{-s^2+\left(\frac{4\amax^2 \eta^2 p^2 t}{\azero}\right)^{\frac{1}{2}}s}ds
             = C_2(t).
  \end{align*}

  \noindent
  \enum{3}: Using \eqref{equ:Kjitilde}, the triangle inequality and Cauchy-Schwarz inequality with assumption \eqref{cond:A5} (see (2)) yield for every $i,j=1,\ldots,d$, $\psi\in\R^d$ and $t>0$
  \begin{align}
    \label{equ:SpectralNormKjitilde}
    \max_{k=1,\ldots,N}\left|\tilde{K}^{ji}_{kk}(\psi,t)\right|\leqslant\left(\left(2t\amin\right)^{-2}\left|\psi\right|^2+\left(2t\amin\right)^{-1}\delta_{ij}\right) \max_{k=1,\ldots,N}\left|\tilde{K}_{kk}(\psi,t)\right|.
  \end{align}
  From \eqref{equ:SpectralNormK} with $\left|\beta\right|=2$, \eqref{equ:SpectralNormKjitilde} with \eqref{equ:SpectralNormKtilde}, the transformation theorem (with transformations from (1)) and 
  formula \eqref{equ:1F1Formula1} we obtain
  \begin{align*}
              &\int_{\R^d}e^{\eta p|\psi|}\left|K^{ji}(\psi,t)\right| d\psi \\
    \leqslant &\int_{\R^d}\kappa e^{\eta p|\psi|}\left(\left(2t\amin\right)^{-2}\left|\psi\right|^2+\left(2t\amin\right)^{-1}\delta_{ij}\right) \max_{k=1,\ldots,N}\left|\tilde{K}_{kk}(\psi,t)\right| d\psi \\
    \leqslant &\int_{\R^d}\kappa e^{\eta p|\psi|}\left(\left(2t\amin\right)^{-2}\left|\psi\right|^2+\left(2t\amin\right)^{-1}\delta_{ij}\right) \left(4\pi t\amin\right)^{-\frac{d}{2}}e^{-\bzero t-\frac{\azero}{4t\amax^2}\left|\psi\right|} d\psi \\
    = &\kappa(4\pi t\amin)^{-\frac{d}{2}} e^{-\bzero t}\frac{2\pi^{\frac{d}{2}}}{\Gamma\left(\frac{d}{2}\right)}
       \int_{0}^{\infty}\left(\left(2t\amin\right)^{-2}r^{2}+\left(2t\amin\right)^{-1}\delta_{ij}\right)r^{d-1}
       e^{-\frac{\azero}{4t\amax^2}r^2+\eta p r} dr \\
    = &\kappa\left(\frac{\amax^2}{\amin\azero}\right)^{\frac{d+2}{2}} e^{-\bzero t}\frac{2}{\Gamma\left(\frac{d}{2}\right)} \left(t\amin\right)^{-1}
       \int_{0}^{\infty}s^{d+1}e^{-s^2+\left(\frac{4\amax^2\eta^2 p^2 t}{\azero}\right)^{\frac{1}{2}} s} ds \\
      &+\delta_{ij}\kappa\left(\frac{\amax^2}{\amin\azero}\right)^{\frac{d}{2}} e^{-\bzero t}\frac{1}{\Gamma\left(\frac{d}{2}\right)} \left(t\amin\right)^{-1}
       \int_{0}^{\infty}s^{d-1}e^{-s^2+\left(\frac{4\amax^2\eta^2 p^2 t}{\azero}\right)^{\frac{1}{2}} s} ds = C_3(t).
  \end{align*}
\end{proof}

The following lemma is essential for the proof that the Schwartz space is a core for the infinitesimal generator of the Ornstein-Uhlenbeck semigroup, see 
\cite[Proposition 2.2 and 3.2]{Metafune2001} for the scalar real-valued case and \cite[Theorem 5.10]{Otten2014} for the complex-valued extension. 
Moreover, the first statement of this lemma is also needed to prove strong continuity of the semigroup.

\begin{lemma}\label{lem:PropK2}
  Let the assumptions \eqref{cond:A8B} and \eqref{cond:A2} be satisfied and let $K$ be given by \eqref{equ:K}, then for every $i,j=1,\ldots,d$ and $t>0$ we have
  \begin{align*}
    &\textrm{\enum{1} }\int_{\R^d}K(\psi,t)d\psi=e^{-Bt}, \\
    &\textrm{\enum{2} }\int_{\R^d}K(\psi,t)\psi_i d\psi=0, \\
    &\textrm{\enum{3} }\int_{\R^d}K(\psi,t)\psi_i \psi_j d\psi=\begin{cases}2te^{-Bt}A &,\,i=j\\0 &,\,i\neq j\end{cases}.
  \end{align*}
\end{lemma}

\begin{remark}
  Throughout this proof we will use $d$-dimensional polar coordinates\index{polar coordinates!d-dimensional}: Let $x\in\R^d$, $\Omega:=]0,\infty[\times[0,2\pi[\times[0,\pi]^d$ 
  and $(r,\phi,\theta_1,\ldots,\theta_{d-2})\in\Omega$, then we define
  \begin{align}
    x_1 =& \Phi_1(r,\phi,\theta_1,\ldots,\theta_{d-2}) := r \cos\phi \prod_{k=1}^{d-2}\sin\theta_k, \nonumber\\
    x_2 =& \Phi_2(r,\phi,\theta_1,\ldots,\theta_{d-2}) := r \sin\phi \prod_{k=1}^{d-2}\sin\theta_k, \label{equ:dDimSphericalCoordinates}\\
    x_i =& \Phi_i(r,\phi,\theta_1,\ldots,\theta_{d-2}) := r \cos\theta_{i-2} \prod_{k=i-1}^{d-2}\sin\theta_k,\,3\leqslant i\leqslant d. \nonumber
  \end{align}
  The transformation $\Phi:\Omega\rightarrow\R^d$ is a $C^{\infty}$-diffeomorphism, \cite[X.8.8 Lemma]{AmannEscher2001}, satisfying $\Phi(\overline{\Omega})=\R^d$ and
  \begin{align*}
    \mathrm{det}D\Phi(r,\phi,\theta_1,\ldots,\theta_{d-2})=(-1)^d r^{d-1}\prod_{k=1}^{d-2}\left(\sin\theta_k\right)^k.
  \end{align*}
\end{remark}

\begin{proof}
  First note that \eqref{cond:A2}, \eqref{cond:A8B} and componentwise integration yields for every $n>-1$
  \begin{align}
    \label{equ:AuxiliaryIntegral}
    \begin{split}
       & \int_{0}^{\infty}r^{n}e^{-\left(4tA\right)^{-1}r^2}dr
      =  \int_{0}^{\infty}r^{n}e^{-Y\left(4t\Lambda_A\right)^{-1}Y^{-1}r^2}dr \\
      =& Y\int_{0}^{\infty}r^{n}e^{-\left(4t\Lambda_A\right)^{-1}r^2}dr Y^{-1}
      =  \frac{\Gamma\left(\frac{n+1}{2}\right)}{2}Y(4t\Lambda_A)^{\frac{n+1}{2}}Y^{-1} \\
      =& \frac{\Gamma\left(\frac{n+1}{2}\right)}{2}(4tA)^{\frac{n+1}{2}}.
    \end{split}
  \end{align}
  \enum{1}: From \eqref{equ:K}, \eqref{equ:AuxiliaryIntegral} (with $n=d-1$), the transformation theorem (with $d$-dimensional polar coordinates) and \eqref{cond:A8B} 
  we directly obtain for $t>0$
  \begin{align*}
      \int_{\R^d}K(\psi,t)d\psi
    =& \left(4\pi t A\right)^{-\frac{d}{2}}e^{-Bt}\int_{\R^d}e^{-\left(4tA\right)^{-1}|\psi|^2}d\psi \\
    =& \left(4\pi t A\right)^{-\frac{d}{2}}e^{-Bt}\frac{2 \pi^{\frac{d}{2}}}{\Gamma\left(\frac{d}{2}\right)}\int_{0}^{\infty}r^{d-1}e^{-\left(4tA\right)^{-1}r^2}dr \\
    =& \left(4\pi t A\right)^{-\frac{d}{2}}e^{-Bt}\frac{2 \pi^{\frac{d}{2}}}{\Gamma\left(\frac{d}{2}\right)}\frac{\Gamma\left(\frac{d}{2}\right)}{2}(4tA)^{\frac{d}{2}}
    =  e^{-Bt}.
  \end{align*}
  \enum{2}: Now we must use $d$-dimensional polar coordinates. From the transformation theorem we obtain
  \begin{align*}
     & \int_{\R^d}e^{-\left(4tA\right)^{-1}|\psi|^2}\psi_i d\psi \\
    =& \int_{\Omega}e^{-(4tA)^{-1}r^2}\cdot\left\{\begin{array}{cl}r \cos\phi \prod_{k=1}^{d-2}\sin\theta_k &,\,i=1\\
                                                                   r \sin\phi \prod_{k=1}^{d-2}\sin\theta_k&,\,i=2\\
                                                                   r \cos\theta_{i-2} \prod_{k=i-1}^{d-2}\sin\theta_k&,\,3\leqslant i\leqslant d-2\end{array}\right\} \\
     & \quad\quad\cdot\left|\mathrm{det}D\Phi(r,\phi,\theta_1,\ldots,\theta_{d-2})\right|dr d\phi d\theta_1\cdots d\theta_{d-2} \\
    =& \int_{\Omega}e^{-(4tA)^{-1}r^2}\cdot\left\{\begin{array}{cl}r \cos\phi \prod_{k=1}^{d-2}\sin\theta_k &,\,i=1\\
                                                                   r \sin\phi \prod_{k=1}^{d-2}\sin\theta_k&,\,i=2\\
                                                                   r \cos\theta_{i-2} \prod_{k=i-1}^{d-2}\sin\theta_k&,\,3\leqslant i\leqslant d-2\end{array}\right\} \\
     & \quad\quad \cdot r^{d-1}\prod_{k=1}^{d-2}\left|\sin\theta_k\right|^k dr d\phi d\theta_1\cdots d\theta_{d-2} \\
    =& \left(\int_{0}^{\infty}r^{d}e^{-(4tA)^{-1}r^2}dr\right)\int_{0}^{2\pi}\int_{0}^{\pi}\cdots\int_{0}^{\pi} \\
     & \quad\quad\left\{\begin{array}{cl}\cos\phi \prod_{k=1}^{d-2}\sin\theta_k \prod_{k=1}^{d-2}\left|\sin\theta_k\right|^k&,\,i=1\\
                                                                   \sin\phi \prod_{k=1}^{d-2}\sin\theta_k \prod_{k=1}^{d-2}\left|\sin\theta_k\right|^k&,\,i=2\\
                                                                   \cos\theta_{i-2} \prod_{k=i-1}^{d-2}\sin\theta_k \prod_{k=1}^{d-2}\left|\sin\theta_k\right|^k&,\,3\leqslant i\leqslant d-2\end{array}\right\}
       d\phi d\theta_1\cdots d\theta_{d-2}.
  \end{align*}
  In case of $i=1$ and $i=2$ the $\phi$-integrals vanishes and in case of $3\leqslant i\leqslant d-2$ the $\theta_{i-2}$-integral vanishes, since 
  using for example 
  \begin{align*}
    \left(\sin a\right)^n = \frac{1}{2^n}\sum_{k=0}^{n}\left(\begin{array}{c}n\\k\end{array}\right)\cos\left((n-2k)\left(a-\frac{\pi}{2}\right)\right),\,n\in\N,
  \end{align*}
  we obtain
  \begin{align}
    \label{equ:CosinusSinusIntegral}
      \int_{0}^{\pi}\cos\theta_{i-2}\left|\sin\theta_{i-2}\right|^{i-2}d\theta_{i-2}
    = \int_{0}^{\pi}\cos\theta_{i-2}\left(\sin\theta_{i-2}\right)^{i-2}d\theta_{i-2}=0.
  \end{align}
  Hence, we have for every $i=1,\ldots,d$ and $t>0$
  \begin{align*}
      \int_{\R^d}K(\psi,t)\psi_i d\psi
    = \left(4\pi t A\right)^{-\frac{d}{2}}e^{-Bt}\int_{\R^d}e^{-\left(4tA\right)^{-1}|\psi|^2}\psi_i d\psi
    = 0.
  \end{align*}
  \enum{3}: Finally, let us use $d$-dimensional polar coordinates once more. Similar to (2) from the transformation theorem we obtain
  \begin{align*}
      \int_{\R^d}e^{-\left(4tA\right)^{-1}|\psi|^2}\psi_i \psi_j d\psi 
    =& \left(\int_{0}^{\infty}r^{d+1}e^{-(4tA)^{-1}r^2}dr\right)\int_{0}^{2\pi}\int_{0}^{\pi}\cdots\int_{0}^{\pi} \\
     & \quad\quad\left\{\begin{array}{cl}\cos\phi \prod_{k=1}^{d-2}\sin\theta_k &,\,i=1\\
                                         \sin\phi \prod_{k=1}^{d-2}\sin\theta_k &,\,i=2\\
                                         \cos\theta_{i-2} \prod_{k=i-1}^{d-2}\sin\theta_k &,\,3\leqslant i\leqslant d-2\end{array}\right\} \prod_{k=1}^{d-2}\left|\sin\theta_k\right|^k \\
     & \quad\quad\left\{\begin{array}{cl}\cos\phi \prod_{k=1}^{d-2}\sin\theta_k &,\,j=1\\
                                         \sin\phi \prod_{k=1}^{d-2}\sin\theta_k &,\,j=2\\
                                         \cos\theta_{j-2} \prod_{k=j-1}^{d-2}\sin\theta_k &,\,3\leqslant j\leqslant d-2\end{array}\right\}
       d\phi d\theta_1\cdots d\theta_{d-2} \\
    =& \begin{cases} \frac{\pi^{\frac{d}{2}}}{2}\left(4tA\right)^{\frac{d}{2}+1} &,\,i=j\\0&,\,i\neq j\end{cases}.
  \end{align*}
  Accept the last equality, we first deduce from \eqref{equ:AuxiliaryIntegral} with $n=d+1$
  \begin{align}
    \label{equ:rIntegral}
    \int_{0}^{\infty}r^{d+1}e^{-(4tA)^{-1}r^2}dr=\frac{\Gamma\left(\frac{d+2}{2}\right)}{2}(4tA)^{\frac{d}{2}+1}.
  \end{align}
  Moreover, for $\Re l>-1$, $a,b\in\N_0$ with $a\leqslant b$ it holds
  \begin{align}
    \label{equ:ProductFormulaSinusIntegral}
      \prod_{l=a}^{b}\int_{0}^{\pi}\left(\sin\theta\right)^l d\theta
    = \prod_{l=a}^{b}\pi^{\frac{1}{2}}\frac{\Gamma\left(\frac{l+1}{2}\right)}{\Gamma\left(\frac{l+2}{2}\right)}
    = \pi^{\frac{b-a+1}{2}}\frac{\Gamma\left(\frac{a+1}{2}\right)}{\Gamma\left(\frac{b+2}{2}\right)}.
  \end{align}
  Let is first consider the cases $i=j=1$ and $i=j=2$. Here we must use
  \begin{align*}
    \int_{0}^{2\pi}\left(\cos\phi\right)^2 d\phi = \pi,\quad \int_{0}^{2\pi}\left(\sin\phi\right)^2 d\phi = \pi
  \end{align*}
  and \eqref{equ:ProductFormulaSinusIntegral} with $a=3$ and $b=d$
  \begin{align*}
      \prod_{k=1}^{d-2}\int_{0}^{\pi}\left(\sin\theta_k\right)^2\left|\sin\theta_k\right|^{k}d\theta_k
    =& \prod_{k=1}^{d-2}\int_{0}^{\pi}\left(\sin\theta\right)^{k+2}d\theta
    = \prod_{l=3}^{d}\int_{0}^{\pi}\left(\sin\theta\right)^{l}d\theta \\
    =& \pi^{\frac{d}{2}-1}\frac{\Gamma\left(\frac{4}{2}\right)}{\Gamma\left(\frac{d+2}{2}\right)}
    = \frac{\pi^{\frac{d}{2}-1}}{\Gamma\left(\frac{d+2}{2}\right)}.
  \end{align*}
  Now, let us consider the case $3\leqslant i=j\leqslant d$. Here we can deduce from \eqref{equ:ProductFormulaSinusIntegral} (with
  $a=1$ and $b=i-3$, $a=b=i-2$, $a=b=i$ as well as $a=i+1$ and $b=d$)
  \begin{align*}
    & \prod_{k=1}^{i-3}\int_{0}^{\pi}\left|\sin\theta_k\right|^k d\theta_k
      = \prod_{k=1}^{i-3}\int_{0}^{\pi}\left(\sin\theta\right)^k d\theta
      = \pi^{\frac{i-3}{2}}\frac{\Gamma(1)}{\Gamma\left(\frac{i-1}{2}\right)}
      = \frac{\pi^{\frac{i-3}{2}}}{\Gamma\left(\frac{i-1}{2}\right)}, \\
    & \int_{0}^{\pi}\left(\cos\theta_{i-2}\right)^2\left|\sin\theta_{i-2}\right|^{i-2}d\theta_{i-2}
      = \int_{0}^{\pi}\left(1-\left(\sin\theta_{i-2}\right)^2\right)\left(\sin\theta_{i-2}\right)^{i-2}d\theta_{i-2} \\
    & \quad\quad  = \int_{0}^{\pi}\left(\sin\theta\right)^{i-2}d\theta - \int_{0}^{\pi}\left(\sin\theta\right)^{i}d\theta
      = \pi^{\frac{1}{2}}\left(\frac{\Gamma\left(\frac{i-1}{2}\right)}{\Gamma\left(\frac{i}{2}\right)}-\frac{\Gamma\left(\frac{i+1}{2}\right)}{\Gamma\left(\frac{i+2}{2}\right)}\right), \\
    & \prod_{k=i-1}^{d-2}\int_{0}^{\pi}\left(\sin\theta_k\right)^2\left|\sin\theta_k\right|^k d\theta_k
      = \prod_{k=i-1}^{d-2}\int_{0}^{\pi}\left(\sin\theta\right)^{k+2}d\theta \\
    & \quad\quad = \prod_{l=i+1}^{d}\int_{0}^{\pi}\left(\sin\theta\right)^{l}d\theta
      = \pi^{\frac{d-i}{2}}\frac{\Gamma\left(\frac{i+2}{2}\right)}{\Gamma\left(\frac{d+2}{2}\right)}, \\
    & \int_{0}^{2\pi}1d\phi=2\pi.
  \end{align*}
  Multiplying these four terms with \eqref{equ:rIntegral} and using $\Gamma(x+1)=x\Gamma(x)$ we obtain $\frac{\pi^{\frac{d}{2}}}{2}\left(4tA\right)^{\frac{d}{2}+1}$. 
  Next, we consider the cases $3\leqslant i<j\leqslant d$ and $3\leqslant j<i\leqslant d$. Let w.l.o.g. $i<j$, then the term from \eqref{equ:CosinusSinusIntegral} vanishes. 
  For all the other cases exactly one term vanishes, namely
  \begin{align*}
    &\int_{0}^{2\pi}\sin\phi \cos\phi d\phi=0, &&\text{if $(i=1,j=2)$ or $(i=2,j=1)$,} \\
    &\int_{0}^{2\pi}\cos\phi d\phi=0,          &&\text{if $(i=1,3\leqslant j\leqslant d)$ or $(3\leqslant i\leqslant d,j=1)$,} \\
    &\int_{0}^{2\pi}\sin\phi d\phi=0,          &&\text{if $(i=2,3\leqslant j\leqslant d)$ or $(3\leqslant i\leqslant d,j=2)$.}
  \end{align*}
\end{proof}

Assuming \eqref{cond:A2}, let us introduce the following functions
\begin{align*}
  C_4(t) =& C_{\theta} \kappa \aone^{\frac{d}{2}}e^{-\bzero t}\bigg[{}_1F_1\left(\frac{d}{2};\frac{1}{2};\nu t\right)
            +2\frac{\Gamma\left(\frac{d+1}{2}\right)}{\Gamma\left(\frac{d}{2}\right)}\left(\nu t\right)^{\frac{1}{2}}{}_1F_1\left(\frac{d+1}{2};\frac{3}{2};\nu t\right)\bigg]^{\frac{1}{p}}, \\
  C_5(t) =& C_{\theta} \kappa \aone^{\frac{d+1}{2}}e^{-\bzero t}\left(t\amin\right)^{-\frac{1}{2}}\bigg[\frac{\Gamma\left(\frac{d+1}{2}\right)}{\Gamma\left(\frac{d}{2}\right)}
            {}_1F_1\left(\frac{d+1}{2};\frac{1}{2};\nu t\right) \\
          & +2\frac{\Gamma\left(\frac{d+2}{2}\right)}{\Gamma\left(\frac{d}{2}\right)}\left(\nu t\right)^{\frac{1}{2}}{}_1F_1\left(\frac{d+2}{2};\frac{3}{2};\nu t\right)\bigg]^{\frac{1}{p}},\\
  C_6(t) =& C_{\theta} \kappa \aone^{\frac{d+2}{2}}e^{-\bzero t} \left(t\amin\right)^{-1}\bigg[
            \frac{\Gamma\left(\frac{d+2}{2}\right)}{\Gamma\left(\frac{d}{2}\right)}{}_1F_1\left(\frac{d+2}{2};\frac{1}{2};\nu t\right) \\
          & +2\frac{\Gamma\left(\frac{d+3}{2}\right)}{\Gamma\left(\frac{d}{2}\right)}\left(\nu t\right)^{\frac{1}{2}}{}_1F_1\left(\frac{d+3}{2};\frac{3}{2};\nu t\right)
            +\frac{\delta_{ij}}{2}\aone^{-1}{}_1F_1\left(\frac{d}{2};\frac{1}{2};\nu t\right) \\
          & +\delta_{ij}\aone^{-1}\frac{\Gamma\left(\frac{d+1}{2}\right)}{\Gamma\left(\frac{d}{2}\right)}\left(\nu t\right)^{\frac{1}{2}}
            {}_1F_1\left(\frac{d+1}{2};\frac{3}{2};\nu t\right)\bigg]^{\frac{1}{p}},
\end{align*}
with $\aone:=\frac{\amax^2}{\amin\azero}\geqslant 1$, $\nu:=\frac{\amax^2 \eta^2 p^2}{\azero}\geqslant 0$, $1\leqslant p\leqslant\infty$, $\eta\geqslant 0$, $\kappa:=\cond(Y)$ 
with $Y$ from \eqref{cond:A8B} and $\amax,\amin,\azero,\bzero$ defined in \eqref{equ:aminamaxazerobzero}. In case of $p=\infty$ the constants are given by $C_{4+\left|\beta\right|}(t)$ 
with $p=1$ for every $\left|\beta\right|=0,1,2$. Note that in case of $p=1$ we have $C_{4+\left|\beta\right|}(t)=C_{\theta}C_{1+\left|\beta\right|}(t)$.

In order to prove exponentially weighted resolvent estimates, we need the following lemma. The upper bound for $\eta^2$ is the maximal decay rate appearing in the 
resolvent estimates.

\begin{lemma}\label{lem:ImproperInt}
  Let the assumption \eqref{cond:A2} be satisfied and let $1\leqslant p<\infty$. Moreover, let $0<\vartheta<1$, $\omega\in\R$, $\lambda\in\C$ with $\Re\lambda>\omega$ and 
  $0\leqslant\eta^2\leqslant\vartheta\frac{\azero(\Re\lambda-\omega)}{\amax^2 p^2}$, then we have
  \begin{align*}
    &\textrm{\enum{1} }\int_{0}^{\infty}e^{-\Re\lambda t}C_4(t)dt\leqslant\frac{C_7}{\Re\lambda-\omega}, \\
    &\textrm{\enum{2} }\int_{0}^{\infty}e^{-\Re\lambda t}C_5(t)dt\leqslant\frac{C_8}{\left(\Re\lambda-\omega\right)^{\frac{1}{2}}},
  \end{align*}
  where the constants are given by
  \begin{align*}
    C_7 =& C_{\theta}\kappa \aone^{\frac{d}{2}}\left(\frac{1}{1-\vartheta}\right)^{\frac{1}{p}}
           \bigg({}_2F_1\left(-\frac{d-1}{2},1;\frac{1}{2};-\frac{\vartheta}{1-\vartheta}\right) \\
         & +\pi^{\frac{1}{2}}\frac{\Gamma\left(\frac{d+1}{2}\right)}{\Gamma\left(\frac{d}{2}\right)}\left(\frac{\vartheta}{1-\vartheta}\right)^{\frac{1}{2}}
           {}_2F_1\left(-\frac{d-2}{2},\frac{3}{2};\frac{3}{2};-\frac{\vartheta}{1-\vartheta}\right)\bigg)^{\frac{1}{p}}, \\
    C_8 =& C_{\theta}\kappa \aone^{\frac{d+1}{2}}\frac{\Gamma\left(\frac{1}{2}\right)}{\amin^{\frac{1}{2}}}
           \left(\frac{1}{1-\vartheta}\right)^{\frac{1}{2p}}
           \bigg(\frac{\Gamma\left(\frac{d+1}{2}\right)}{\Gamma\left(\frac{d}{2}\right)}{}_2F_1\left(-\frac{d}{2},\frac{1}{2};\frac{1}{2};-\frac{\vartheta}{1-\vartheta}\right) \\
         &  +2\frac{\Gamma\left(\frac{d+2}{2}\right)}{\Gamma\left(\frac{1}{2}\right)\Gamma\left(\frac{d}{2}\right)}\left(\frac{\vartheta}{1-\vartheta}\right)^{\frac{1}{2}}
           {}_2F_1\left(-\frac{d-1}{2},1;\frac{3}{2};-\frac{\vartheta}{1-\vartheta}\right)\bigg)^{\frac{1}{p}}
  \end{align*}
  with $\aone:=\frac{\amax^2}{\amin\azero}\geqslant 1$, $\kappa:=\cond(Y)$ with $Y$ from \eqref{cond:A8B} and $\amax,\amin,\azero,\bzero$ defined in \eqref{equ:aminamaxazerobzero}.
\end{lemma}

\begin{proof}
  \enum{1}: From $\czero:=\Re\lambda-\omega$, H{\"o}lder's inequality (with $q\in]1,\infty]$ such that $\frac{1}{p}+\frac{1}{q}=1$), the transformation theorem (with transformation 
  $\Phi(t)=\frac{\amax^2 \eta^2 p^2 t}{\azero}$), formula \eqref{equ:1F1Formula2} (with $a=\frac{d}{2}$, $b=\frac{1}{2}$, $x=s$ and $a=\frac{d+1}{2}$, $b=\frac{3}{2}$, $x=s$) 
  and formula \eqref{equ:1F1Formula3} (with $\alpha=1$, $c=\frac{\azero\czero-\amax^2 \eta^2 p^2}{\amax^2 \eta^2 p^2}$, $a=-\frac{d-1}{2}$, $b=\frac{1}{2}$ 
  and $\alpha=\frac{3}{2}$, $c=\frac{\azero\czero-\amax^2 \eta^2 p^2}{\amax^2 \eta^2 p^2}$, $a=-\frac{d-2}{2}$, $b=\frac{3}{2}$ -- note that because of \eqref{cond:A2}, $\czero>0$ and 
  $\eta^2<\frac{\azero\czero}{\amax^2 p^2}$ we have $\Re c>0$) we obtain
  \begin{align*}
             & \int_{0}^{\infty}e^{-\Re\lambda t}C_4(t)dt \\
            =& \int_{0}^{\infty}C_{\theta}\kappa\left(\frac{\amax^2}{\amin\azero}\right)^{\frac{d}{2}}e^{-\czero t}\bigg[{}_1F_1\left(\frac{d}{2};\frac{1}{2};\frac{\amax^2 \eta^2 p^2 t}{\azero}\right) \\
             & +2\frac{\Gamma\left(\frac{d+1}{2}\right)}{\Gamma\left(\frac{d}{2}\right)}\left(\frac{\amax^2 \eta^2 p^2 t}{\azero}\right)^{\frac{1}{2}}{}_1F_1\left(\frac{d+1}{2};\frac{3}{2};\frac{\amax^2 \eta^2 p^2 t}{\azero}\right)\bigg]^{\frac{1}{p}}dt\\
    \leqslant& C_{\theta}\kappa\left(\frac{\amax^2}{\amin\azero}\right)^{\frac{d}{2}}\bigg(\int_{0}^{\infty}e^{-\czero t}dt\bigg)^{\frac{1}{q}}
               \bigg(\int_{0}^{\infty}e^{-\czero t}{}_1F_1\left(\frac{d}{2};\frac{1}{2};\frac{\amax^2 \eta^2 p^2 t}{\azero}\right)dt \\
             & +2\frac{\Gamma\left(\frac{d+1}{2}\right)}{\Gamma\left(\frac{d}{2}\right)}\int_{0}^{\infty}\left(\frac{\amax^2 \eta^2 p^2 t}{\azero}\right)^{\frac{1}{2}}e^{-\czero t}
               {}_1F_1\left(\frac{d+1}{2};\frac{3}{2};\frac{\amax^2 \eta^2 p^2 t}{\azero}\right)dt\bigg)^{\frac{1}{p}} \\
            =& C_{\theta}\kappa \aone^{\frac{d}{2}}\left(\frac{1}{\czero}\right)^{\frac{1}{q}}\bigg(\left(\frac{\amax^2 \eta^2 p^2}{\azero}\right)^{-1}
               \int_{0}^{\infty}e^{-\frac{\azero\czero}{\amax^2 \eta^2 p^2}s}{}_1F_1\left(\frac{d}{2};\frac{1}{2};s\right)ds \\
             & +2\frac{\Gamma\left(\frac{d+1}{2}\right)}{\Gamma\left(\frac{d}{2}\right)}\left(\frac{\amax^2 \eta^2 p^2}{\azero}\right)^{-1}
               \int_{0}^{\infty}s^{\frac{1}{2}}e^{-\frac{\azero\czero}{\amax^2 \eta^2 p^2}s}{}_1F_1\left(\frac{d+1}{2};\frac{3}{2};s\right)ds\bigg)^{\frac{1}{p}} \\
            =& C_{\theta}\kappa \aone^{\frac{d}{2}}\left(\frac{1}{\czero}\right)\bigg(\left(\frac{\amax^2 \eta^2 p^2}{\azero\czero}\right)^{-1}
               \int_{0}^{\infty}e^{-\left(\frac{\azero\czero}{\amax^2 \eta^2 p^2}-1\right)s}{}_1F_1\left(-\frac{d-1}{2};\frac{1}{2};-s\right)ds \\
             & +2\frac{\Gamma\left(\frac{d+1}{2}\right)}{\Gamma\left(\frac{d}{2}\right)}\left(\frac{\amax^2 \eta^2 p^2}{\azero\czero}\right)^{-1}
               \int_{0}^{\infty}s^{\frac{1}{2}}e^{-\left(\frac{\azero\czero}{\amax^2 \eta^2 p^2}-1\right)s}{}_1F_1\left(-\frac{d-2}{2};\frac{3}{2};-s\right)ds\bigg)^{\frac{1}{p}} \\
            =& C_{\theta}\kappa\aone^{\frac{d}{2}}\left(\frac{1}{\czero}\right)\left(\frac{\azero\czero}{\azero\czero-\amax^2\eta^2 p^2}\right)^{\frac{1}{p}}
               \bigg({}_2F_1\left(-\frac{d-1}{2},1;\frac{1}{2};-\frac{\amax^2 \eta^2 p^2}{\azero\czero-\amax^2 \eta^2 p^2}\right) \\
             & +\pi^{\frac{1}{2}}\frac{\Gamma\left(\frac{d+1}{2}\right)}{\Gamma\left(\frac{d}{2}\right)}\left(\frac{\amax^2 \eta^2 p^2}{\azero\czero-\amax^2 \eta^2 p^2}\right)^{\frac{1}{2}}
               {}_2F_1\left(-\frac{d-2}{2},\frac{3}{2};\frac{3}{2};-\frac{\amax^2 \eta^2 p^2}{\azero\czero-\amax^2 \eta^2 p^2}\right)\bigg)^{\frac{1}{p}}.
  \end{align*}
  Finally, to obtain $C_7$ we must use that ${}_2F_1$ is strictly monotonically decreasing in $]-\infty,0]$ as well as the inequalities
  \begin{align}
    \label{equ:ThetaInequalities}
    \frac{\azero\czero}{\azero\czero-\amax^2\eta^2 p^2}\leqslant\frac{1}{1-\vartheta}\quad\text{and}\quad\frac{\amax^2 \eta^2 p^2}{\azero\czero-\amax^2 \eta^2 p^2}\leqslant\frac{\vartheta}{1-\vartheta}.
  \end{align}
  \enum{2}: From $\czero:=\Re\lambda-\omega$, H{\"o}lder's inequality (with $q\in]1,\infty]$ such that $\frac{1}{p}+\frac{1}{q}=1$), the transformation theorem (with transformation 
  $\Phi(t)=\frac{\amax^2 \eta^2 p^2 t}{\azero}$), formula \eqref{equ:1F1Formula2} (with $a=\frac{d+1}{2}$, $b=\frac{1}{2}$, $x=s$ and $a=\frac{d+2}{2}$, $b=\frac{3}{2}$, $x=s$) 
  and formula \eqref{equ:1F1Formula3} (with $\alpha=\frac{1}{2}$, $c=\frac{\azero\czero-\amax^2 \eta^2 p^2}{\amax^2 \eta^2 p^2}$, $a=-\frac{d}{2}$, $b=\frac{1}{2}$ 
  and $\alpha=1$, $c=\frac{\azero\czero-\amax^2 \eta^2 p^2}{\amax^2 \eta^2 p^2}$, $a=-\frac{d-1}{2}$, $b=\frac{3}{2}$ -- note that because of \eqref{cond:A2}, $\czero>0$ and 
  $\eta^2<\frac{\azero\czero}{\amax^2 p^2}$ we have $\Re c>0$) we obtain
  \begin{align*}
             & \int_{0}^{\infty}e^{-\Re\lambda t}C_5(t)dt \\
            =& \int_{0}^{\infty}C_{\theta}\kappa\left(\frac{\amax^2}{\amin\azero}\right)^{\frac{d+1}{2}}e^{-\czero t}\left(t\amin\right)^{-\frac{1}{2}}\bigg[\frac{\Gamma\left(\frac{d+1}{2}\right)}{\Gamma\left(\frac{d}{2}\right)}
               {}_1F_1\left(\frac{d+1}{2};\frac{1}{2};\frac{\amax^2 \eta^2 p^2 t}{\azero}\right) \\
             & +2\frac{\Gamma\left(\frac{d+2}{2}\right)}{\Gamma\left(\frac{d}{2}\right)}\left(\frac{\amax^2 \eta^2 p^2 t}{\azero}\right)^{\frac{1}{2}}
               {}_1F_1\left(\frac{d+2}{2};\frac{3}{2};\frac{\amax^2 \eta^2 p^2 t}{\azero}\right)\bigg]^{\frac{1}{p}}dt \\
    \leqslant& C_{\theta}\kappa\left(\frac{\amax^2}{\amin\azero}\right)^{\frac{d+1}{2}}\amin^{-\frac{1}{2}}\bigg(\int_{0}^{\infty}t^{-\frac{1}{2}}e^{-\czero t}dt\bigg)^{\frac{1}{q}} \\
             & \cdot\bigg(\frac{\Gamma\left(\frac{d+1}{2}\right)}{\Gamma\left(\frac{d}{2}\right)}\int_{0}^{\infty}t^{-\frac{1}{2}}e^{-\czero t}{}_1F_1\left(\frac{d+1}{2};\frac{1}{2};\frac{\amax^2 \eta^2 p^2 t}{\azero}\right)dt \\
             & +2\frac{\Gamma\left(\frac{d+2}{2}\right)}{\Gamma\left(\frac{d}{2}\right)}\left(\frac{\amax^2 \eta^2 p^2}{\azero}\right)^{\frac{1}{2}}
               \int_{0}^{\infty}e^{-\czero t}{}_1F_1\left(\frac{d+2}{2};\frac{3}{2};\frac{\amax^2 \eta^2 p^2 t}{\azero}\right)dt\bigg)^{\frac{1}{p}} \\
            =& C_{\theta}\kappa \aone^{\frac{d+1}{2}}\left(\left(\frac{1}{\czero}\right)^{\frac{1}{2}}\Gamma\left(\frac{1}{2}\right)\right)^{\frac{1}{q}}\amin^{-\frac{1}{2}} \\
             & \cdot\bigg(\frac{\Gamma\left(\frac{d+1}{2}\right)}{\Gamma\left(\frac{d}{2}\right)}\left(\frac{\amax^2 \eta^2 p^2}{\azero}\right)^{-\frac{1}{2}}
               \int_{0}^{\infty}s^{-\frac{1}{2}}e^{-\frac{\azero\czero}{\amax^2 \eta^2 p^2}s}{}_1F_1\left(\frac{d+1}{2};\frac{1}{2};s\right)ds \\
             & +2\frac{\Gamma\left(\frac{d+2}{2}\right)}{\Gamma\left(\frac{d}{2}\right)}\left(\frac{\amax^2 \eta^2 p^2}{\azero}\right)^{-\frac{1}{2}}
               \int_{0}^{\infty}e^{-\frac{\azero\czero}{\amax^2 \eta^2 p^2}s}{}_1F_1\left(\frac{d+2}{2};\frac{3}{2};s\right)ds\bigg)^{\frac{1}{p}} \\
            =& C_{\theta}\kappa \aone^{\frac{d+1}{2}}\left(\frac{1}{\czero}\right)^{\frac{1}{2}}\frac{\Gamma\left(\frac{1}{2}\right)}{\amin^{\frac{1}{2}}} \\
             & \cdot\bigg(\frac{\Gamma\left(\frac{d+1}{2}\right)}{\Gamma\left(\frac{1}{2}\right)\Gamma\left(\frac{d}{2}\right)}\left(\frac{\amax^2 \eta^2 p^2}{\azero\czero}\right)^{-\frac{1}{2}}
               \int_{0}^{\infty}s^{-\frac{1}{2}}e^{-\left(\frac{\azero\czero}{\amax^2 \eta^2 p^2}-1\right)s}{}_1F_1\left(-\frac{d}{2};\frac{1}{2};-s\right)ds \\
             & +2\frac{\Gamma\left(\frac{d+2}{2}\right)}{\Gamma\left(\frac{1}{2}\right)\Gamma\left(\frac{d}{2}\right)}\left(\frac{\amax^2 \eta^2 p^2}{\azero\czero}\right)^{-\frac{1}{2}}
               \int_{0}^{\infty}e^{-\left(\frac{\azero\czero}{\amax^2 \eta^2 p^2}-1\right)s}{}_1F_1\left(-\frac{d-1}{2};\frac{3}{2};-s\right)ds\bigg)^{\frac{1}{p}} \\
            =& C_{\theta}\kappa \aone^{\frac{d+1}{2}}\left(\frac{1}{\czero}\right)^{\frac{1}{2}}\frac{\Gamma\left(\frac{1}{2}\right)}{\amin^{\frac{1}{2}}}
               \left(\frac{\azero\czero}{\azero\czero-\amax^2 \eta^2 p^2}\right)^{\frac{1}{2p}} \\
             & \cdot\bigg(\frac{\Gamma\left(\frac{d+1}{2}\right)}{\Gamma\left(\frac{d}{2}\right)}{}_2F_1\left(-\frac{d}{2},\frac{1}{2};\frac{1}{2};-\frac{\amax^2 \eta^2 p^2}{\azero\czero-\amax^2 \eta^2 p^2}\right) \\
             & +2\frac{\Gamma\left(\frac{d+2}{2}\right)}{\Gamma\left(\frac{1}{2}\right)\Gamma\left(\frac{d}{2}\right)}\left(\frac{\amax^2 \eta^2 p^2}{\azero\czero-\amax^2 \eta^2 p^2}\right)^{\frac{1}{2}}
               {}_2F_1\left(-\frac{d-1}{2},1;\frac{3}{2};-\frac{\amax^2 \eta^2 p^2}{\azero\czero-\amax^2 \eta^2 p^2}\right)\bigg)^{\frac{1}{p}}.
  \end{align*}
  Finally, to obtain $C_8$ we use again that ${}_2F_1$ is strictly monotonically decreasing in $]-\infty,0]$ and the inequalities \eqref{equ:ThetaInequalities}.
\end{proof}

%
%
\sect{The Ornstein-Uhlenbeck semigroup in \texorpdfstring{$L^p_{\theta}(\R^d,\C^N)$}{L\_\{p,theta\}(Rd,CN)}}
\label{sec:TheOrnsteinUhlenbeckSemigroupInLp}

Recall the Ornstein-Uhlenbeck kernel
\begin{align*}
  H(x,\xi,t)=(4\pi t A)^{-\frac{d}{2}}\exp\left(-Bt-(4tA)^{-1}\left|e^{tS}x-\xi\right|^2\right)
\end{align*}
of the differential operator
\begin{align*}
  \left[\L_{\infty}v\right](x):=A\triangle v(x)+\left\langle Sx,\nabla v(x)\right\rangle-Bv(x)
\end{align*}
from Theorem \ref{thm:HeatKernel}. In the following we study the family of mappings
\begin{align}
  \left[T(t)v\right](x):= \begin{cases}
                              \int_{\R^d}H(x,\xi,t)v(\xi)d\xi &\text{, }t>0 \\
                              v(x) &\text{, }t=0
                            \end{cases}\quad ,x\in\R^d
  \label{equ:OrnsteinUhlenbeckSemigroup2}
\end{align}
on the (complex-valued) Banach space $(L^p_{\theta}(\R^d,\C^N),\left\|\cdot\right\|_{L^p_{\theta}})$ for $1\leqslant p\leqslant\infty$. 

The next two theorems show that the family of mappings $\left(T(t)\right)_{t\geqslant 0}$ generates a semigroup on $L^p_{\theta}(\R^d,\C^N)$ 
for every $1\leqslant p\leqslant\infty$. The proof of Theorem \ref{thm:OrnsteinUhlenbeckLpBoundedness} is taken from \cite[Theorem 5.3]{Otten2014} 
while the proof of Theorem \ref{thm:OrnsteinUhlenbeckLpSemigroup} extends \cite[Theorem 5.2]{Otten2014} to the weighted $L^p$-case. 
For later use, we keep track of the time dependence of constants in terms of Kummer functions.

\begin{theorem}[Boundedness on $L^p_{\theta}(\R^d,\C^N)$]\label{thm:OrnsteinUhlenbeckLpBoundedness}
  Let the assumptions \eqref{cond:A8B}--\eqref{cond:A5} be satisfied and let $1\leqslant p\leqslant\infty$. Then for every radial weight 
  function $\theta\in C(\R^d,\R)$ of exponential growth rate $\eta\geqslant 0$ and for every $v\in L^p_{\theta}(\R^d,\C^N)$ we have
  \begin{align}
    \left\|T(t)v\right\|_{L^p_{\theta}(\R^d,\C^N)}          &\leqslant C_4(t)\left\|v\right\|_{L^p_{\theta}(\R^d,\C^N)} &&,\,t\geqslant 0,          \label{equ:OrnsteinUhlenbeckLpBoundednessOfT}     \\
    \left\|D_i T(t)v\right\|_{L^p_{\theta}(\R^d,\C^N)}      &\leqslant C_5(t)\left\|v\right\|_{L^p_{\theta}(\R^d,\C^N)} &&,\,t>0,\,i=1,\ldots,d,    \label{equ:OrnsteinUhlenbeckLpBoundednessOfDiT}   \\
    \left\|D_j D_i T(t)v\right\|_{L^p_{\theta}(\R^d,\C^N)}  &\leqslant C_6(t)\left\|v\right\|_{L^p_{\theta}(\R^d,\C^N)} &&,\,t>0,\,i,j=1,\ldots,d,  \label{equ:OrnsteinUhlenbeckLpBoundednessOfDjDiT}
  \end{align}
  where the constants $C_{4+\left|\beta\right|}(t)=C_{4+\left|\beta\right|}(t;\bzero,p)$ for $\left|\beta\right|=0,1,2$ are given by
  \begin{align*}
    C_4(t) =& C_{\theta}\kappa\aone^{\frac{d}{2}} e^{-\bzero t} \bigg[{}_1F_1\left(\frac{d}{2};\frac{1}{2};\nu t\right)
              +2\frac{\Gamma\left(\frac{d+1}{2}\right)}{\Gamma\left(\frac{d}{2}\right)}\left(\nu t\right)^{\frac{1}{2}}{}_1F_1\left(\frac{d+1}{2};\frac{3}{2};\nu t\right)\bigg]^{\frac{1}{p}},\\
    C_5(t) =& C_{\theta}\kappa\aone^{\frac{d+1}{2}} e^{-\bzero t} \left(t\amin\right)^{-\frac{1}{2}}\bigg[\frac{\Gamma\left(\frac{d+1}{2}\right)}{\Gamma\left(\frac{d}{2}\right)}
               {}_1F_1\left(\frac{d+1}{2};\frac{1}{2};\nu t\right) \\
            & +2\frac{\Gamma\left(\frac{d+2}{2}\right)}{\Gamma\left(\frac{d}{2}\right)}\left(\nu t\right)^{\frac{1}{2}}{}_1F_1\left(\frac{d+2}{2};\frac{3}{2};\nu t\right)\bigg]^{\frac{1}{p}},\\
    C_6(t) =& C_{\theta}\kappa\aone^{\frac{d+2}{2}} e^{-\bzero t} \left(t\amin\right)^{-1}\bigg[
              \frac{\Gamma\left(\frac{d+2}{2}\right)}{\Gamma\left(\frac{d}{2}\right)}{}_1F_1\left(\frac{d+2}{2};\frac{1}{2};\nu t\right) \\
            & +2\frac{\Gamma\left(\frac{d+3}{2}\right)}{\Gamma\left(\frac{d}{2}\right)}\left(\nu t\right)^{\frac{1}{2}}{}_1F_1\left(\frac{d+3}{2};\frac{3}{2};\nu t\right)
              +\frac{\delta_{ij}}{2}\aone^{-1}{}_1F_1\left(\frac{d}{2};\frac{1}{2};\nu t\right) \\
            & +\delta_{ij}\aone^{-1}\frac{\Gamma\left(\frac{d+1}{2}\right)}{\Gamma\left(\frac{d}{2}\right)}\left(\nu t\right)^{\frac{1}{2}}
              {}_1F_1\left(\frac{d+1}{2};\frac{3}{2};\nu t\right)\bigg]^{\frac{1}{p}},
  \end{align*}
  with $\aone:=\frac{\amax^2}{\amin\azero}\geqslant 1$, $\nu:=\frac{\amax^2 \eta^2 p^2}{\azero}\geqslant 0$, $\kappa:=\cond(Y)$ with $Y$ from \eqref{cond:A8B} 
  and $\amax,\amin,\azero,\bzero$ defined in \eqref{equ:aminamaxazerobzero}, if $1\leqslant p<\infty$. In case $p=\infty$ they are given by $C_{4+\left|\beta\right|}(t)$ with $p=1$. 
  Note that $C_{4+\left|\beta\right|}(t)\sim t^{\frac{-p\left|\beta\right|+d+\left|\beta\right|-1}{2p}} e^{-\left(\bzero-\frac{\nu}{p}\right)t}$ as $t\to\infty$ 
  and $C_{4+\left|\beta\right|}(t)\sim t^{-\frac{\left|\beta\right|}{2}}$ as $t\to 0$ for every $\left|\beta\right|=0,1,2$.
\end{theorem}

\begin{proof}
  In the following $\beta\in\N_0^d$ denotes a $d$-dimensional multi-index with $|\beta|=0,1,2$ and we use the notation
  \begin{align*}
    D^{\beta}v = \begin{cases}v&{, }|\beta|=0\\ D_iv&{, }|\beta|=1\\ D_jD_iv&{, }|\beta|=2\end{cases},\,
    D^{\beta}H = \begin{cases}H&{, }|\beta|=0\\ D_iH&{, }|\beta|=1\\ D_jD_iH&{, }|\beta|=2\end{cases},\,
    K^{\beta}  = \begin{cases}K&{, }|\beta|=0\\  K^i&{, }|\beta|=1\\  K^{ji}&{, }|\beta|=2\end{cases}
  \end{align*}
  where $i,j=1,\ldots,d$. Moreover, the kernels $K$, $K^i$ and $K^{ji}$ are from \eqref{equ:K}, \eqref{equ:Ki} and \eqref{equ:Kji}. To show \eqref{equ:OrnsteinUhlenbeckLpBoundednessOfT}, 
  \eqref{equ:OrnsteinUhlenbeckLpBoundednessOfDiT} and \eqref{equ:OrnsteinUhlenbeckLpBoundednessOfDjDiT} for $v\in L^p_{\theta}(\R^d,\C^N)$ with $1\leqslant p<\infty$, we use 
  \eqref{equ:OrnsteinUhlenbeckSemigroup2}, the transformation theorem (with transformations $\Phi(\xi)=e^{tS}x-\xi$ in $\xi$, $\Phi(x)=e^{tS}x-\psi$ in $x$), \eqref{equ:K}, \eqref{equ:Ki}, 
  \eqref{equ:Kji}, the triangle inequality, H{\"o}lder's inequality (with $q\in]1,\infty]$ such that $\frac{1}{p}+\frac{1}{q}=1$), Fubini's theorem, 
  \eqref{equ:WeightFunctionProp1}--\eqref{equ:WeightFunctionProp3} and Lemma \ref{lem:PropK}(1),(2),(3) to obtain
  \begin{align*}
             & \left\|D^{\beta}T(t)v\right\|_{L^p_{\theta}}
            =  \left(\int_{\R^d}\theta^p(x)\left|D^{\beta}\left[T(t)v\right](x)\right|^p dx\right)^{\frac{1}{p}} \\
            =& \left(\int_{\R^d}\theta^p(x)\left|\int_{\R^d}\left[D^{\beta} H(x,\xi,t)\right]v(\xi)d\xi\right|^p dx\right)^{\frac{1}{p}} \\
            =& \left(\int_{\R^d}\theta^p(x)\left|\int_{\R^d}K^{\beta}(\psi,t)v(e^{tS}x-\psi)d\psi\right|^p dx\right)^{\frac{1}{p}} \\        
    \leqslant& \left(\int_{\R^d}\left(\int_{\R^d}\theta(x)\left|K^{\beta}(\psi,t)\right| \left|v(e^{tS}x-\psi)\right|d\psi\right)^p dx\right)^{\frac{1}{p}} \\
    \leqslant& \left(\int_{\R^d}\left(\int_{\R^d}\left|K^{\beta}(\psi,t)\right| d\psi\right)^{\frac{p}{q}}
               \int_{\R^d}\left|K^{\beta}(\psi,t)\right| \theta^p(x)\left|v(e^{tS}x-\psi)\right|^p d\psi dx\right)^{\frac{1}{p}} \\
            =& \left\|K^{\beta}(\cdot,t)\right\|_{L^1}^{\frac{1}{q}}
               \left(\int_{\R^d}\left|K^{\beta}(\psi,t)\right| \int_{\R^d}\theta^p(x)\left|v(e^{tS}x-\psi)\right|^p dx d\psi\right)^{\frac{1}{p}} \\
            =& \left\|K^{\beta}(\cdot,t)\right\|_{L^1}^{\frac{1}{q}}
               \left(\int_{\R^d}\left|K^{\beta}(\psi,t)\right| \int_{\R^d}\theta^p(e^{-tS}(y+\psi))\left|v(y)\right|^p dy d\psi\right)^{\frac{1}{p}} \\
    \leqslant& C_{\theta}\left\|K^{\beta}(\cdot,t)\right\|_{L^1}^{\frac{p-1}{p}}
               \left\|e^{\eta p|\cdot|}K^{\beta}(\cdot,t)\right\|_{L^1}^{\frac{1}{p}}
               \left\|v\right\|_{L^p_{\theta}}
    \leqslant  C_{4+|\beta|}(t)\left\|v\right\|_{L^p_{\theta}}
  \end{align*}
  for $t\geqslant 0$, if $|\beta|=0$, and for $t>0$, if $|\beta|=1$ or $|\beta|=2$. For $p=\infty$ the proof is very similar and can be found later in Theorem 
  \ref{thm:OrnsteinUhlenbeckCbBoundedness}.
\end{proof}

\begin{theorem}[Semigroup on $L^p_{\theta}(\R^d,\C^N)$]\label{thm:OrnsteinUhlenbeckLpSemigroup}
  Let the assumptions \eqref{cond:A8B}--\eqref{cond:A5} be satisfied and let $1\leqslant p\leqslant\infty$. Moreover, let $\theta\in C(\R^d,\R)$ be a radial 
  weight function of exponential growth rate $\eta\geqslant 0$. Then the operators $\left(T(t)\right)_{t\geqslant 0}$ given by \eqref{equ:OrnsteinUhlenbeckSemigroup2} generate 
  a semigroup on $L^p_{\theta}(\R^d,\C^N)$, i.e. $T(t):L^p_{\theta}(\R^d,\C^N)\rightarrow L^p_{\theta}(\R^d,\C^N)$ is linear and bounded for every 
  $t\geqslant 0$ and satisfies the semigroup properties
  \begin{align}
    &T(0) = I,                                        \label{equ:InitalProp} \\
    &T(t)T(s) = T(t+s),\,\forall\,s,t\geqslant 0. \label{equ:KozyklProp}
  \end{align}
\end{theorem}

\begin{proof}
  The boundedness of $T(t)$ in $L^p_{\theta}(\R^d,\C^N)$ for every $t\geqslant 0$ follows from \eqref{equ:OrnsteinUhlenbeckLpBoundednessOfT}. The linearity and property 
  \eqref{equ:InitalProp} follow obviously from the definition of $T(t)$ in \eqref{equ:OrnsteinUhlenbeckSemigroup2}. To prove property \eqref{equ:KozyklProp} we use 
  \eqref{equ:OrnsteinUhlenbeckSemigroup2} and apply Lemma \ref{lem:ChapmanKolmogorovFormula} and Fubini's theorem
  \begin{align*}
     \left[T(t)\left(T(s)v\right)\right](x)
    =&\int_{\R^d}H(x,\tilde{\xi},t)\int_{\R^d}H(\tilde{\xi},\xi,s)v(\xi)d\xi d\tilde{\xi} \\
    =&\int_{\R^d}\int_{\R^d}H(x,\tilde{\xi},t)H(\tilde{\xi},\xi,s)d\tilde{\xi}v(\xi)d\xi \\
    =&\int_{\R^d}H(x,\xi,t+s)v(\xi)d\xi = \left[T(t+s)v\right](x),\, x\in\R^d.
  \end{align*}
\end{proof}

The next theorem states that the semigroup $\left(T(t)\right)_{t\geqslant 0}$ is strongly continuous on $L^p_{\theta}(\R^d,\C^N)$ for every $1\leqslant p<\infty$, 
if we additionally require \eqref{equ:WeightFunctionProp4}. Recall that property \eqref{equ:WeightFunctionProp5} follows directly from \eqref{equ:WeightFunctionProp3} and 
strong continuity justifies to introduce the infinitesimal generator. If $\theta$ satisfies only \eqref{equ:WeightFunctionProp1}--\eqref{equ:WeightFunctionProp3} 
then $\left(T(t)\right)_{t\geqslant 0}$ is in general neither strongly nor weakly continuous on $L^p_{\theta}(\R^d,\C^N)$, see Remark \ref{rem:RemarkLpFunctionSpace} below. 
Only in the unweighted $L^p$-case is the semigroup $\left(T(t)\right)_{t\geqslant 0}$ strongly continuous on $L^p(\R^d,\C^N)$ for every $1\leqslant p<\infty$, 
\cite[Theorem 5.3]{Otten2014}. In this case the $L^p$-continuity of translations follows from \cite[Satz 2.14(1)]{Alt2006} and the $L^p$-continuity of rotations follows from 
a density argument. The following result is an extension of \cite[Proof of Theorem 5.3]{Otten2014}.

\begin{theorem}[Strong continuity on $L^p_{\theta}(\R^d,\C^N)$]\label{thm:OrnsteinUhlenbeckLpStrongContinuity}
  Let the assumptions \eqref{cond:A8B}--\eqref{cond:A5} be satisfied and let $1\leqslant p\leqslant\infty$. Moreover, let $\theta\in C(\R^d,\R)$ be a radial 
  weight function of exponential growth rate $\eta\geqslant 0$ satisfying \eqref{equ:WeightFunctionProp4}. Then $\left(T(t)\right)_{t\geqslant 0}$ is a $C_0$-semigroup 
 (or strongly continuous semigroup) on $L^p_{\theta}(\R^d,\C^N)$, i.e.
  \begin{align}
    \lim_{t\downarrow 0}\left\|T(t)v-v\right\|_{L^p_{\theta}(\R^d,\C^N)}=0\quad\forall\,v\in L^p_{\theta}(\R^d,\C^N). \label{equ:StrongCont}
  \end{align}
\end{theorem}

\begin{proof} 
  1. Let us define the (\begriff{$d$-dimensional}) \begriff{diffusion semigroup} (\begriff{Gaussian semigroup}, \begriff{heat semigroup})
  \begin{align}
    \label{equ:DiffusionSemigroup}
    \begin{split}
      \left[G(t,0)v\right](y) :=& \int_{\R^d}H(e^{-tS}y,\xi,t)v(\xi)d\xi \\
                               =& \int_{\R^d}\left(4\pi tA\right)^{-\frac{d}{2}}\exp\left(-Bt-\left(4tA\right)^{-1}\left|y-\xi\right|^2\right)v(\xi)d\xi,\,t>0,
    \end{split}
  \end{align}
  then we have $\left[T(t)v\right](x)=\left[G(t,0)v\right](e^{tS}x)$. Following \cite{DaPratoLunardi1995}, we consider the decomposition
  \begin{align*}
               \left\|T(t)v-v\right\|_{L^p_{\theta}}
    \leqslant &\left\|\left[G(t,0)v\right](e^{tS}\cdot)-v(e^{tS}\cdot)\right\|_{L^p_{\theta}}+\left\|v(e^{tS}\cdot)-v(\cdot)\right\|_{L^p_{\theta}} \\
           =: &\left\|v_1(\cdot,t)\right\|_{L^p_{\theta}} + \left\|v_2(\cdot,t)\right\|_{L^p_{\theta}}.
  \end{align*}

  \noindent
  2. First we compute the $v_1$-term: Applying the transformation theorem (with transformation $\Phi(x)=e^{tS}x$) and using \eqref{equ:WeightFunctionProp3} 
  and \eqref{cond:A5} we decompose $v_1$ as follows
  \begin{align*}
              & \left\|v_1(\cdot,t)\right\|_{L^p_{\theta}}
             =  \left\|\theta(\cdot)\left(\left[G(t,0)v\right](e^{tS}\cdot)-v(e^{tS}\cdot)\right)\right\|_{L^p} \\
             =& \left\|\theta(e^{-tS}\cdot)\left(\left[G(t,0)v\right](\cdot)-v(\cdot)\right)\right\|_{L^p}
             =  \left\|\left[G(t,0)v\right](\cdot)-v(\cdot)\right\|_{L^p_{\theta}} \\
     \leqslant& \left\|\int_{\R^d}H(e^{-tS}\cdot,\xi,t)\left(v(\xi)-v(\cdot)\right)d\xi\right\|_{L^p_{\theta}} 
                + \left\|\left(\int_{\R^d}H(e^{-tS}\cdot,\xi,t)d\xi-I_N\right)v(\cdot)\right\|_{L^p_{\theta}} \\
            =:& \left\|v_3(\cdot,t)\right\|_{L^p_{\theta}} + \left\|v_4(\cdot,t)\right\|_{L^p_{\theta}}.
  \end{align*}

  \noindent
  3. Let us consider the $v_4$-term: Applying the transformation theorem (with transformation $\Phi(\xi)=y-\xi$) and Lemma \ref{lem:PropK2}(1), we obtain
  \begin{align*}
             & \left\|v_4(\cdot,t)\right\|_{L^p_{\theta}} 
            =  \left(\int_{\R^d}\left|\theta(y)\left(\int_{\R^d}H(e^{-tS}y,\xi,t)d\xi - I_N\right)v(y)\right|^p dy\right)^{\frac{1}{p}} \\
            =& \left(\int_{\R^d}\left|\theta(y)\left(\int_{\R^d}K(\psi,t)d\psi - I_N\right)v(y)\right|^p dy\right)^{\frac{1}{p}} 
    \leqslant  \left|e^{-tB} - I_N\right| \left\|v\right\|_{L^p_{\theta}},\,t>0.
  \end{align*}
  Therefore, $\lim_{t\downarrow 0}\left\|v_4(\cdot,t)\right\|_{L^p_{\theta}}=0$.

  \noindent
  4. The $v_3$-term is much more delicate: First, we need the following integral for some constant $\delta_0>0$, compare proof of Lemma \ref{lem:PropK},
  \begin{align*}
             & \int_{\left|\psi\right|\geqslant\delta_0}e^{\eta p|\psi|}\left|K(\psi,t)\right| d\psi 
    \leqslant \int_{\left|\psi\right|\geqslant\delta_0}\kappa\left(4\pi t \amin\right)^{-\frac{d}{2}}
               e^{-\bzero t-\frac{\azero}{4t\amax^2}\left|\psi\right|^2+\eta p|\psi|}d\psi \\
            =& \kappa\left(4\pi t \amin\right)^{-\frac{d}{2}} e^{-\bzero t} \frac{2\pi^{\frac{d}{2}}}{\Gamma\left(\frac{d}{2}\right)}
               \int_{\delta_0}^{\infty}r^{d-1}e^{-\frac{\azero}{4t\amax^2}r^2-\eta p r}dr \\
            =& \kappa\left(\frac{\amax^2}{\amin\azero}\right)^{\frac{d}{2}}e^{-\bzero t}\frac{2}{\Gamma\left(\frac{d}{2}\right)}
               \int_{\left(\frac{\azero}{4t\amax^2}\right)^{\frac{1}{2}}\delta_0}^{\infty}s^{d-1}e^{-s^2+\left(\frac{4\amax^2 \eta^2 p^2 t}{\azero}\right)^{\frac{1}{2}}s}ds 
           =: \tilde{C}(t,\eta p,\delta_0),
  \end{align*}
  where we used the transformation theorem (with transformations for $d$-dimensional polar coordinates and 
  $\Phi(r)=\left(\frac{\azero}{4t\amax^2}\right)^{\frac{1}{2}}r$). Note, that $\tilde{C}(t,\eta p,\delta_0)\to 0$ as $t\downarrow 0$ for every fixed $\delta_0>0$ and $\eta, p\in\R$. 
  Moreover, we need $L^p_{\theta}$-continuity (at $0$) w.r.t. translations which states that
  \begin{align}
    \label{equ:LpthetaContinuityTranslation}
    \forall\,\varepsilon_0>0\;\exists\,\delta_0>0\;\forall\,\psi\in\R^d\text{ with }|\psi|\leqslant\delta_0:\;
    \left\|v(\cdot-\psi)-v(\cdot)\right\|_{L^p_{\theta}}\leqslant\varepsilon_0.
  \end{align}
  This was proved in \cite[Satz 2.14(1)]{Alt2006} for the unweighted $L^p$-case, i.e. $\theta\equiv 1$ and $C_{\theta}=1$. In the weighted case we use assumption \eqref{equ:WeightFunctionProp4} 
  which yields
  \begin{align}
    \label{equ:WeightFunctionConvergence}
    \forall\,\varepsilon_1>0\;\exists\,\delta_1>0\;\forall\,\psi\in\R^d\text{ with }|\psi|\leqslant\delta_1:\;
    \sup_{x\in\R^d}\left|\frac{\theta(x+\psi)-\theta(x)}{\theta(x)}\right|\leqslant\varepsilon_1.
  \end{align}
  Now, let $\varepsilon>0$ and choose $\delta:=\min\{\delta_0,\delta_1\}$ with $\delta_0$ from \eqref{equ:LpthetaContinuityTranslation} (for $\theta\equiv 1$ and 
  $\varepsilon_0=\frac{\varepsilon}{2}$) and $\delta_1$ from \eqref{equ:WeightFunctionConvergence} (with $\varepsilon_1:=\frac{\varepsilon}{2\left\|w\right\|_{L^p}}$ and 
  $w:=\theta v\in L^p(\R^d,C^N)$) then we obtain the $L^p_{\theta}$-continuity from \eqref{equ:LpthetaContinuityTranslation}
  \begin{align*}
              &\left\|v(\cdot-\psi)-v(\cdot)\right\|_{L^p_{\theta}} 
            =  \left\|\theta(\cdot)v(\cdot-\psi)-\theta(\cdot)v(\cdot)\right\|_{L^p} \\
    \leqslant &\left\|(\theta(\cdot)-\theta(\cdot-\psi))v(\cdot-\psi)\right\|_{L^p}+\left\|\theta(\cdot-\psi)v(\cdot-\psi)-\theta(\cdot)v(\cdot)\right\|_{L^p} \\
    \leqslant &\sup_{x\in\R^d}\left|\frac{\theta(x+\psi)-\theta(x)}{\theta(x)}\right|\left\|w\right\|_{L^p}+\left\|w(\cdot-\psi)-w(\cdot)\right\|_{L^p}
    \leqslant \varepsilon\quad\forall\,\psi\in\R^d,\,|\psi|\leqslant\delta.
  \end{align*}
  Using the transformation theorem (with transformations $\Phi(\xi)=y-\xi$ and $\Phi(y)=y-\psi$), the triangle inequality, H{\"o}lder's inequality (with $q\in]1,\infty]$ such 
  that $\frac{1}{p}+\frac{1}{q}=1$), Fubini's theorem, \eqref{equ:LpthetaContinuityTranslation} and Lemma \ref{lem:PropK}(1) we obtain
  \begin{align*}
             & \left\|v_3(\cdot,t)\right\|_{L^p_{\theta}} 
            = \left(\int_{\R^d}\left|\theta(y)\int_{\R^d}H(e^{-tS}y,\xi,t)\left(v(\xi)-v(y)\right)d\xi\right|^p dy\right)^{\frac{1}{p}} \\
            =& \left(\int_{\R^d}\left|\theta(y)\int_{\R^d}K(\psi,t)\left(v(y-\psi)-v(y)\right)d\psi\right|^p dy\right)^{\frac{1}{p}} \\
    \leqslant& \left(\int_{\R^d}\left(\theta(y)\int_{\R^d}\left|K(\psi,t)\right|\left|v(y-\psi)-v(y)\right|d\psi\right)^p dy\right)^{\frac{1}{p}} \\
    \leqslant& \left\|K(\cdot,t)\right\|_{L^1}^{\frac{1}{q}}\left(\int_{\R^d}\theta^p(y)\int_{\R^d}\left|K(\psi,t)\right| \left|v(y-\psi)-v(y)\right|^p d\psi dy\right)^{\frac{1}{p}} \\
            =& \left\|K(\cdot,t)\right\|_{L^1}^{\frac{1}{q}}\left(\int_{\R^d}\left|K(\psi,t)\right|\int_{\R^d}\theta^p(y)\left|v(y-\psi)-v(y)\right|^p dy d\psi\right)^{\frac{1}{p}} \\       
            =& \left\|K(\cdot,t)\right\|_{L^1}^{\frac{1}{q}}
               \left(\int_{\R^d}\left|K(\psi,t)\right|\left\|v(\cdot-\psi)-v(\cdot)\right\|_{L^p_{\theta}}^p d\psi\right)^{\frac{1}{p}} \\
            =& \left\|K(\cdot,t)\right\|_{L^1}^{\frac{1}{q}}
               \bigg(\int_{\left|\psi\right|\leqslant\delta_0}\left|K(\psi,t)\right|\left\|v(\cdot-\psi)-v(\cdot)\right\|_{L^p_{\theta}}^p d\psi \\
             & +\int_{\left|\psi\right|\geqslant\delta_0}\left|K(\psi,t)\right|\left\|v(\cdot-\psi)-v(\cdot)\right\|_{L^p_{\theta}}^p d\psi\bigg)^{\frac{1}{p}} \\
    \leqslant& \left\|K(\cdot,t)\right\|_{L^1}^{\frac{1}{q}}
               \bigg(\varepsilon_0^p \int_{\left|\psi\right|\leqslant\delta_0}\left|K(\psi,t)\right| d\psi \\
             & +2^{p-1} \left\|v\right\|_{L^p_{\theta}}^p \int_{\left|\psi\right|\geqslant\delta_0}\left(C_{\theta}^p e^{\eta p|\psi|}+1\right)\left|K(\psi,t)\right|d\psi\bigg)^{\frac{1}{p}} \\
    \leqslant& \left\|K(\cdot,t)\right\|_{L^1}^{\frac{p-1}{p}}
               \bigg(\varepsilon_0^p \int_{\R^d}\left|K(\psi,t)\right| d\psi
               +2^{p-1} \left\|v\right\|_{L^p_{\theta}}^p \int_{\left|\psi\right|\geqslant\delta_0}\left(C_{\theta}^p e^{\eta p|\psi|}+1\right)\left|K(\psi,t)\right| d\psi\bigg)^{\frac{1}{p}} \\
    \leqslant& C_1^{\frac{p-1}{p}}(t) \left(\varepsilon_0^p C_1(t)+2^{p-1} \left\|v\right\|_{L^p_{\theta}}^p \left(C_{\theta}^p \tilde{C}(t,\eta p,\delta_0)+\tilde{C}(t,0,\delta_0) \right)\right)^{\frac{1}{p}},
  \end{align*}
  where the constant $C_1(t)=C_1(t,0)$ is from Lemma \ref{lem:PropK} with $\eta p=0$. Hence, $\lim_{t\downarrow 0}\left\|v_3(\cdot,t)\right\|_{L^p_{\theta}}
  \leqslant\varepsilon_0 C_1(0)=\varepsilon_0 M^{\frac{d}{2}}$. Now, choose $\varepsilon_0>0$ arbitrary small.

  \noindent
  5. Finally, let us consider the $v_2$-term: Here we need the $L^p_{\theta}$-continuity (at $0$) w.r.t. rotations which states
  \begin{align}
    \label{equ:LpthetaContinuityRotation}
    \forall\,\varepsilon_0>0\;\exists\,t_0>0\;\forall\,0\leqslant t\leqslant t_0:\;\left\|v(e^{tS}\cdot)-v(\cdot)\right\|_{L^p_{\theta}}\leqslant\varepsilon_0.
  \end{align}
  First we verify \eqref{equ:LpthetaContinuityRotation} for the unweighted $L^p$-case with $\theta\equiv 1$ and $C_{\theta}=1$: Let $\varepsilon_0>0$. Since 
  $C_{\mathrm{c}}^{\infty}(\R^d,\C^N)$ is dense in $L^p(\R^d,\C^N)$ w.r.t. $\left\|\cdot\right\|_{L^p}$ for every $1\leqslant p<\infty$, see \cite[Satz 2.14(3)]{Alt2006}, 
  we can choose $\varphi_{\varepsilon_0}\in C_{\mathrm{c}}^{\infty}(\R^d,\C^N)$ such that $\left\|v-\varphi_{\varepsilon_0}\right\|_{L^p}\leqslant\frac{\varepsilon_0}{3}$. 
  Since $\varphi_{\varepsilon_0}\in C_{\mathrm{c}}^{\infty}(\R^d,\C^N)$, $\varphi_{\varepsilon_0}$ is uniformly continuous on $\supp(\varphi_{\varepsilon_0})$, i.e.
  \begin{align*}
    \forall\,\varepsilon_1>0\;\exists\,\delta_1=\delta_1(\varepsilon_1)>0\;&\forall\,x,x_0\in\supp(\varphi_{\varepsilon_0}) \\
    &\quad\text{ with }\left|x-x_0\right|\leqslant\delta_1:\;\left|\varphi_{\varepsilon_0}(x)-\varphi_{\varepsilon_0}(x_0)\right|\leqslant\varepsilon_1.
  \end{align*}
  Choosing $x_0:=e^{tS}x$ we have
  \begin{align*}
    \exists\,t_0=t_0(\delta_1)>0\;\forall\,0\leqslant t\leqslant t_0:\;\left|e^{tS}x-x\right|\leqslant\delta_1.
  \end{align*}
  Thus, choosing $\varepsilon_1:=\varepsilon_0\left(3\left|\supp(\varphi_{\varepsilon_0})\right|^{\frac{1}{p}}\right)^{-1}$ and combining this facts yields
  \begin{align*}
      \left\|\varphi_{\varepsilon_0}(e^{tS}\cdot)-\varphi_{\varepsilon_0}(\cdot)\right\|_{L^p}
    = \left(\int_{\supp(\varphi_{\varepsilon_0})}\left|\varphi_{\varepsilon_0}(e^{tS}x)-\varphi_{\varepsilon_0}(x)\right|^p\right)^{\frac{1}{p}}
    \leqslant\varepsilon_0\quad\forall\,0\leqslant t\leqslant t_0(\varepsilon_0).
  \end{align*}
  Now, we prove \eqref{equ:LpthetaContinuityRotation} for the general weighted $L^p$-case: Let $\varepsilon>0$ and choose $t_0$ as in \eqref{equ:LpthetaContinuityRotation} 
  (for $\theta\equiv 1$, $C_{\theta}=1$, $\varepsilon_0=\varepsilon$ and $w:=\theta v\in L^p(\R^d,\C^N)$ instead of $v$) then using \eqref{equ:WeightFunctionProp3} we obtain 
  the $L^p_{\theta}$-continuity from \eqref{equ:LpthetaContinuityRotation}
  \begin{align*}
               \left\|v_2(\cdot,t)\right\|_{L^p_{\theta}} 
            =& \left\|v(e^{tS}\cdot)-v(\cdot)\right\|_{L^p_{\theta}}
            =  \left\|\theta(\cdot)v(e^{tS}\cdot)-\theta(\cdot)v(\cdot)\right\|_{L^p} \\
    \leqslant& \left\|(\theta(\cdot)-\theta(e^{tS}\cdot))v(e^{tS}\cdot)\right\|_{L^p}+\left\|\theta(e^{tS}\cdot)v(e^{tS}\cdot)-\theta(\cdot)v(\cdot)\right\|_{L^p} \\
            =& \left\|w(e^{tS}\cdot)-w(\cdot)\right\|_{L^p} \leqslant\varepsilon\quad\forall\,0\leqslant t\leqslant t_0.
  \end{align*}
  Here, we used that \eqref{cond:A5} and \eqref{equ:WeightFunctionProp3} imply $\theta(e^{tS}x)=\theta(x)$ for every $x\in\R^d$ and $t\in\R$. Note that one can also use 
  \eqref{equ:WeightFunctionProp5} which is automatically satisfied by \eqref{equ:WeightFunctionProp3}. We end up with 
  $\lim_{t\downarrow 0}\left\|v_2(\cdot,t)\right\|_{L^p_{\theta}}\leqslant\varepsilon$. Now, choose $\varepsilon>0$ arbitrary small.
\end{proof}

\begin{remark}\label{rem:RemarkLpFunctionSpace}
  Condition \eqref{equ:WeightFunctionProp3} in this paper was assumed only to simplify the proofs of the heat kernel estimates in Section \ref{sec:SomePropertiesOfTheHeatKernel}. 
  The strong continuity in Theorem \ref{thm:OrnsteinUhlenbeckLpStrongContinuity} does not necessarily need the assumptions \eqref{equ:WeightFunctionProp3}, \eqref{equ:WeightFunctionProp4} 
  and \eqref{equ:WeightFunctionProp5}. If we omit these assumptions, then the semigroup $(T(t))_{t\geqslant 0}$ is still strongly continuous but (in general) only on a closed 
  subspace of $L^p_{\theta}(\R^d,\C^N)$, namely
  \begin{align*}
    X^p_{\theta}:=\left\{v\in L^p_{\theta}(\R^d,\C^N)\mid\text{ $v$ satisfies \eqref{equ:LpthetaContinuityTranslation} and \eqref{equ:LpthetaContinuityRotation}}\right\}, 1\leqslant p\leqslant\infty.
  \end{align*}
  Note that \eqref{equ:WeightFunctionProp4} was only assumed to deduce \eqref{equ:LpthetaContinuityTranslation} directly and therefore, to guarantee that $L^p_{\theta}=X^p_{\theta}$ 
  for any $1\leqslant p<\infty$. If \eqref{equ:WeightFunctionProp4} is not assumed it is possible to show that $T(t)$ maps $X^p_{\theta}$ into itself by extending the 
  proof of Theorem \ref{thm:OrnsteinUhlenbeckLpSemigroup}. In contrast to Theorem \ref{thm:OrnsteinUhlenbeckLpStrongContinuity} this even leads to strong continuity for $p=\infty$. 
  Note that in the unweighted case it holds $X^p=L^p(\R^d,\C^N)$ for any $1\leqslant p<\infty$ but not for $p=\infty$. Summarizing these facts, if we impose additionally 
  assumption \eqref{equ:WeightFunctionProp4} then $(T(t))_{t\geqslant 0}$ is strongly continuous on the whole $L^p_{\theta}(\R^d,\C^N)$ for any $1\leqslant p<\infty$, otherwise 
  $(T(t))_{t\geqslant 0}$ is only strongly continuous on the closed subspace $X^p_{\theta}$ but even for any $1\leqslant p\leqslant\infty$.
\end{remark}

Now, the \begriff{infinitesimal generator} $A_{p,\theta}:\D(A_{p,\theta})\subseteq L^p_{\theta}(\R^d,\C^N)\rightarrow L^p_{\theta}(\R^d,\C^N)$ of the Ornstein-Uhlenbeck 
semigroup $\left(T(t)\right)_{t\geqslant 0}$ in $L^p_{\theta}(\R^d,\C^N)$ for $1\leqslant p<\infty$, short $\left(A_{p,\theta},\D(A_{p,\theta})\right)$, is 
defined by, \cite[II.1.2 Definition]{EngelNagel2000},
\begin{align*}
  A_{p,\theta}v := \lim_{t\downarrow 0}\frac{T(t)v-v}{t},\; 1\leqslant p<\infty
\end{align*}
for every $v\in\D(A_{p,\theta})$, where the \begriff{domain} (or \begriff{maximal domain}) of $A_{p,\theta}$ is given by the subspace
\begin{align*}
  \D(A_{p,\theta}):=&\left\{v\in L^p_{\theta}(\R^d,\C^N)\mid \lim_{t\downarrow 0}\frac{T(t)v-v}{t}\text{ exists in $L^p_{\theta}(\R^d,\C^N)$}\right\} \\
          =&\left\{v\in L^p_{\theta}(\R^d,\C^N)\mid A_{p,\theta}v\in L^p_{\theta}(\R^d,\C^N)\right\}.
\end{align*}
Since $\left(A_{p,\theta},\D(A_{p,\theta})\right)$ is a closed operator on the Banach space $L^p_{\theta}(\R^d,\C^N)$ for every $1\leqslant p<\infty$, we have the following notion
\begin{align}
     \sigma(A_{p,\theta}) :=& \left\{\lambda\in\C\mid \lambda I-A_{p,\theta}\text{ is not bijective}\right\} &&\text{\begriff{spectrum of $A_{p,\theta}$},} \nonumber\\
       \rho(A_{p,\theta}) :=& \C\backslash\sigma(A_{p,\theta})                                               &&\text{\begriff{resolvent set of $A_{p,\theta}$},} 
       \label{equ:SpectrumResolventsetResolventLp}\\
  R(\lambda,A_{p,\theta}) :=& \left(\lambda I-A_{p,\theta}\right)^{-1}\text{, for }\lambda\in\rho(A_{p,\theta})       &&\text{\begriff{resolvent of $A_{p,\theta}$}.}\nonumber
\end{align}
In particular, $(\D(A_{p,\theta}),\left\|\cdot\right\|_{A_{p,\theta}})$ is a Banach space w.r.t. the \begriff{graph norm of $A_{p,\theta}$}
\begin{align*}
  \left\|v\right\|_{A_{p,\theta}}:=\left\|A_{p,\theta} v\right\|_{L^p_{\theta}(\R^d,\C^N)}+\left\|v\right\|_{L^p_{\theta}(\R^d,\C^N)},\quad v\in\D(A_{p,\theta}),
\end{align*}
see \cite[B.1 Definition]{EngelNagel2000}. 

In the next Corollary \ref{cor:OrnsteinUhlenbeckLpSolvabilityUniqueness} we show that the resolvent equation $(\lambda I-A_{p,\theta})v=g$ has a unique 
solution in $\D(A_{p,\theta})$. This follows immediately from a general result about semigroup theory, \cite[II.1.10 Theorem]{EngelNagel2000}. The application 
of this theorem requires an estimate of the form
\begin{align}
  \exists\,\omega_{p,\theta}\in\R\;\wedge\;\exists\,M_{p,\theta}\geqslant 1:\;\left\|T(t)\right\|_{\mathcal{L}(L^p_{\theta},L^p_{\theta})}\leqslant M_{p,\theta} e^{\omega_{p,\theta} t}\;
  \forall\,t\geqslant 0.
  \label{equ:OrnsteinUhlenbeckBoundednessOfTinLp}
\end{align}
We point out that such an estimate holds for every strongly continuous semigroup, \cite[I.5.5 Proposition]{EngelNagel2000}. For the unweighted $L^p$-case this follows directly 
from the estimate \eqref{equ:OrnsteinUhlenbeckLpBoundednessOfT} in Theorem \ref{thm:OrnsteinUhlenbeckLpBoundedness} (with $\theta\equiv 1$, $C_{\theta}=1$ and $\eta=\nu=0$) and 
leads to
\begin{align*}
  \left\|T(t)\right\|_{\L(L^p,L^p)}\leqslant M_p e^{\omega_{p} t}\;\forall\,t\geqslant 0\quad\text{with}\quad M_p:=\kappa\aone^{\frac{d}{2}},\quad\omega_p:=-\bzero,
\end{align*}
where $\aone:=\frac{\amax^2}{\amin\azero}\geqslant 1$, $\kappa:=\cond(Y)$ with $Y$ from \eqref{cond:A8B} and $\bzero:=s(-B)$. The weighted $L^p$-case is much more involved since 
we must estimate the terms of the Kummer confluent hypergeometric function ${}_1F_1$ that are contained in the constant $C_4(t)$. These terms must be estimated pointwise for every 
$t\geqslant 0$ by some exponential expression of the form $e^{ct}$: From Remark \ref{rem:1F1} we know that ${}_1F_1(a,b,0)=1$ for $a,b\in\C$ with $b\notin\{0,-1,-2,\ldots\}$ and
\begin{align*}
  t^c {}_1F_1(a,b,t) \sim \frac{\Gamma(b)}{\Gamma(a)}t^{a-b+c}e^{t}\text{, as }t\to\infty,\,t\geqslant 0
\end{align*}
for $a,b\in\C\backslash\{0,-1,-2,\ldots\}$, $c\in\C$. This implies for $a,b\in\C\backslash\{0,-1,-2,\ldots\}$, $c\in\C$
\begin{align*}
  \forall\,\varepsilon>0\;\exists\,\tilde{C}=\tilde{C}(a,b,c,\varepsilon)>0:\;\left|t^c\,{}_1F_1(a,b,t)\right|\leqslant\tilde{C}e^{(1+\varepsilon)t}\;\forall\,t\geqslant 0.
\end{align*}
Therefore, we obtain
\begin{align*}
  \forall\,\varepsilon>0\;\exists\,C_{*}=C_{*}(d,A,\eta,p,\varepsilon)>0:\;
  C_4(t) \leqslant C_{\theta}\kappa\aone^{\frac{d}{2}}C_{*}e^{\left(-\bzero+(1+\varepsilon)\frac{\nu}{p}\right)t} \;\forall\,t\geqslant 0,
\end{align*}
and we find that \eqref{equ:OrnsteinUhlenbeckBoundednessOfTinLp} is satisfied with constants
\begin{align}
  M_{p,\theta}=C_{\theta}\kappa\aone^{\frac{d}{2}}C_{*},\quad \omega_{p,\theta}=-\bzero+(1+\varepsilon)\frac{\nu}{p}=-\bzero+(1+\varepsilon)\frac{\amax^2 \eta^2 p}{\azero},\quad\varepsilon>0.
  \label{equ:MthetaOmegatheta}
\end{align}

\begin{corollary}[Solvability and uniqueness in $L^p_{\theta}(\R^d,\C^N)$]\label{cor:OrnsteinUhlenbeckLpSolvabilityUniqueness}
  Let the assumptions \eqref{cond:A8B}--\eqref{cond:A5} be satisfied and let $1\leqslant p<\infty$. Moreover, let $\theta\in C(\R^d,\R)$ be a radial weight function 
  of exponential growth rate $\eta\geqslant 0$ satisfying \eqref{equ:WeightFunctionProp4} and let $\lambda\in\C$ with $\Re\lambda>\omega_{p,\theta}$ for some $\varepsilon>0$ 
  (see \eqref{equ:MthetaOmegatheta}).
  Then for every $g\in L^p_{\theta}(\R^d,\C^N)$ the resolvent equation
  \begin{align*}
    \left(\lambda I-A_{p,\theta}\right)v = g
  \end{align*}
  admits a unique solution $v_{\star}\in\D(A_{p,\theta})$, which is given by the integral expression
  \begin{align}
    v_{\star} = R(\lambda,A_{p,\theta})g = \int_{0}^{\infty}e^{-\lambda t}T(t)g dt 
                            = \int_{0}^{\infty}e^{-\lambda t}\int_{\R^d}H(\cdot,\xi,t)g(\xi)d\xi dt.
    \label{equ:IntegralExpressionLp}
  \end{align}
  Moreover, the following resolvent estimate holds
  \begin{align*}
    \left\|v_{\star}\right\|_{L^p_{\theta}(\R^d,\C^N)}\leqslant \frac{M_{p,\theta}}{\Re\lambda-\omega_{p,\theta}}\left\|g\right\|_{L^p_{\theta}(\R^d,\C^N)}.
  \end{align*}
\end{corollary}

Corollary \ref{cor:OrnsteinUhlenbeckLpSolvabilityUniqueness} states that the complete right half--plane $\Re\lambda>\omega_{p,\theta}$ belongs to the resolvent set $\rho(A_{p,\theta})$. 
Therefore, the $L^p_{\theta}$-spectrum $\sigma(A_{p,\theta})$ is contained in the left half--plane $\Re\lambda\leqslant\omega_{p,\theta}$. The spectral bound $s(A_{p,\theta})$ of 
$A_{p,\theta}$, \cite[II.1.12 Definition]{EngelNagel2000}, defined by
\begin{align*}
  -\infty\leqslant s(A_{p,\theta}):=\sup_{\lambda\in\sigma(A_{p,\theta})}\Re\lambda\leqslant\omega_{p,\theta}< +\infty
\end{align*}
can be considered as the smallest value $\omega\in\R$ such that the spectrum is contained in the half--plane $\Re\lambda\leqslant\omega$. This value is an important characteristic 
for linear operators. We point out that a large growth rate $\eta\geqslant 0$ yields a large right-shift of the upper bound $\omega_{p,\theta}$, compare \eqref{equ:MthetaOmegatheta}. 
The following remark about the Green's function of $A_{p,\theta}$ concerns the case where $0\in\rho(A_{p,\theta})$. For more information about the $L^p$-spectrum of $\L_{\infty}$ 
we refer to the more general result from \cite[Theorem 7.9 and Theorem 9.4]{Otten2014}. A result about the essential spectrum for the special case $d=p=2$ can be found in 
\cite[Section 8.2]{BeynLorenz2008}. The essential $L^p$-spectrum of the drift term for a general scalar real-valued Ornstein-Uhlenbeck operator was treated in 
\cite[Theorem 2.6]{Metafune2001}.

\begin{remark} \label{rem:GreensFunctionLptheta}
  Let the assumptions \eqref{cond:A8B}--\eqref{cond:A5} and $\Re\sigma(B)>0$ be satisfied. Moreover, choose $\varepsilon>0$ and the growth rate $\eta\geqslant 0$ 
  small such that $\omega_{p,\theta}<0$, then we deduce from Corollary \ref{cor:OrnsteinUhlenbeckLpSolvabilityUniqueness} with $\lambda=0>\omega_{p,\theta}$ that
  \begin{align*}
    \forall\,g\in L^p_{\theta}(\R^d,\C^N)\;\exists\,v_{\star}\in\D(A_{p,\theta}):\;A_{p,\theta}v_{\star}=g\text{ in $L^p_{\theta}(\R^d,\C^N)$}
  \end{align*}
  and the solution is given by, compare \eqref{equ:IntegralExpressionLp},
  \begin{align*}
    v_{\star}(x)=-\int_{0}^{\infty}\int_{\R^d}H(x,\xi,t)g(\xi)d\xi dt.
  \end{align*}
  If Fubini's theorem applies here, we obtain
  \begin{align*}
    v_{\star}(x) = -\left[R(0,A_{p,\theta})g\right](x) = \int_{\R^d}G(x,\xi)g(\xi)d\xi,
  \end{align*}
  where
  \begin{align*}
    G(x,\xi) := -\int_{0}^{\infty}H(x,\xi,t) dt
  \end{align*}
  denotes the Green's function of $A_{p,\theta}$. We have not investigated in detail in which cases this argument can be rigorous.
\end{remark}

We next prove exponentially weighted $L^p$-resolvent estimates for the solution $v_{\star}$ of the resolvent equation $\left(\lambda I-A_{p,\theta}\right)v=g$. 
One consequence of these estimates is the inclusion $\D(A_{p,\theta})\subseteq W^{1,p}_{\theta}(\R^d,\C^N)$ for any $1\leqslant p<\infty$. The proof uses only 
the integral expression \eqref{equ:IntegralExpressionLp} and extends \cite[Theorem 6.8]{Otten2014}.

\begin{theorem}[Exponentially weighted resolvent estimates in $L^p_{\theta}(\R^d,\C^N)$]\label{thm:OrnsteinUhlenbeckAPrioriEstimatesInLpConstantCoefficients}
  Let the assumptions \eqref{cond:A8B}--\eqref{cond:A5} be satisfied, $1\leqslant p<\infty$ and let $\theta_1\in C(\R^d,\R)$ 
  be a radial weight function of exponential growth rate $\eta_1\geqslant 0$ satisfying \eqref{equ:WeightFunctionProp4}. Moreover, let $0<\vartheta<1$, 
  $\lambda\in\C$ with $\Re\lambda>\omega_{p,\theta_1}$ for some $\varepsilon>0$ (see \eqref{equ:MthetaOmegatheta}) and let $\theta_2\in C(\R^d,\R)$ be a further radial weight function 
  of exponential growth rate $\eta_2\geqslant 0$ with 
  \begin{align*}
    0\leqslant\eta_2^2\leqslant\vartheta\frac{\azero(\Re\lambda-\omega_{p,\theta_1})}{\amax^2 p^2},
  \end{align*} 
  which satisfies the relation
  \begin{align*}
    \exists\,C>0:\;\theta_1(x)\leqslant C\theta_2(x)\;\forall\,x\in\R^d.
  \end{align*}
  Let $g\in L^p_{\theta_1}(\R^d,\C^N)$ and $v_{\star}\in\D(A_{p,\theta_1})$ be the unique solution of $(\lambda I-A_{p,\theta_1})v = g$ in $L^p_{\theta_1}(\R^d,\C^N)$ 
  according to Corollary \ref{cor:OrnsteinUhlenbeckLpSolvabilityUniqueness}. If in addition $g\in L^p_{\theta_2}(\R^d,\C^N)$ then we have $v_{\star}\in W^{1,p}_{\theta_2}(\R^d,\C^N)$ 
  and the following estimates hold
  \begin{align}
        \left\|v_{\star}\right\|_{L^p_{\theta_2}(\R^d,\C^N)} \leqslant& \frac{C_7}{\Re\lambda-\omega_{p,\theta_1}}\left\|g\right\|_{L^p_{\theta_2}(\R^d,\C^N)},                                             \label{equ:OrnsteinUhlenbeckExpDecStatVstar}\\
    \left\|D_i v_{\star}\right\|_{L^p_{\theta_2}(\R^d,\C^N)} \leqslant& \frac{C_8}{\left(\Re\lambda-\omega_{p,\theta_1}\right)^{\frac{1}{2}}}\left\|g\right\|_{L^p_{\theta_2}(\R^d,\C^N)},\,i=1,\ldots,d,   \label{equ:OrnsteinUhlenbeckExpDecStatDiVstar}
  \end{align}
  where the $\lambda$-independent constants $C_7$, $C_8$ are given by Lemma \ref{lem:ImproperInt} for $1\leqslant p<\infty$ and $C_{\theta}=C_{\theta_2}$.
\end{theorem}

\begin{proof}
  In the following we use the same notation as in Theorem \ref{thm:OrnsteinUhlenbeckLpBoundedness}. Let $v_{\star}\in\D(A_{p,\theta_1})$ be from Corollary 
  \ref{cor:OrnsteinUhlenbeckLpSolvabilityUniqueness} the unique solution of the resolvent equation. To show \eqref{equ:OrnsteinUhlenbeckExpDecStatVstar} and 
  \eqref{equ:OrnsteinUhlenbeckExpDecStatDiVstar} for $1\leqslant p<\infty$ we use the integral expression \eqref{equ:IntegralExpressionLp}, the transformation theorem (with transformation 
  $\Phi(\xi)=e^{tS}x-\xi$ in $\xi$ and $\Phi(x)=e^{tS}x-\psi$ in $x$), \eqref{equ:K} and \eqref{equ:Ki}, the triangle inequality, H{\"o}lder's inequality (with $q\in]1,\infty]$ 
  such that $\frac{1}{p}+\frac{1}{q}=1$), Fubini's theorem, \eqref{equ:WeightFunctionProp1}--\eqref{equ:WeightFunctionProp3} for $\theta_2$, Lemma \ref{lem:ImproperInt} 
  (with $\omega=\omega_{p,\theta_1}$, $C_{\theta}=C_{\theta_2}$, $\eta=\eta_2$) and obtain for every $\beta\in\N_0^d$ with $\left|\beta\right|\in\{0,1\}$
  \begin{align*}
             & \left\|D^{\beta}v_{\star}\right\|_{L^p_{\theta_2}}
            =  \left(\int_{\R^d}\theta_2^p(x)\left|D^{\beta}v_{\star}(x)\right|^p dx\right)^{\frac{1}{p}} \\
            =& \left(\int_{\R^d}\theta_2^p(x)\left|\int_{0}^{\infty}e^{-\lambda t}\int_{\R^d}\left[D^{\beta} H(x,\xi,t)\right]g(\xi)d\xi dt\right|^p dx\right)^{\frac{1}{p}} \\
            =& \left(\int_{\R^d}\theta_2^p(x)\left|\int_{0}^{\infty}e^{-\lambda t}\int_{\R^d}K^{\beta}(\psi,t)g(e^{tS}x-\psi)d\psi dt\right|^p dx\right)^{\frac{1}{p}} \\        
    \leqslant& \int_{0}^{\infty}e^{-\Re\lambda t}\left(\int_{\R^d}\theta_2^p(x)\left|\int_{\R^d}K^{\beta}(\psi,t)g(e^{tS}x-\psi)d\psi\right|^p dx\right)^{\frac{1}{p}} dt \\
    \leqslant& \int_{0}^{\infty}e^{-\Re\lambda t}\left(\int_{\R^d}\left(\int_{\R^d}\theta_2(x)\left|K^{\beta}(\psi,t)\right| \left|g(e^{tS}x-\psi)\right|d\psi\right)^p dx\right)^{\frac{1}{p}} dt \\
    \leqslant& \int_{0}^{\infty}e^{-\Re\lambda t}\left(\int_{\R^d} Z^{\frac{p}{q}}(t)
               \int_{\R^d}\left|K^{\beta}(\psi,t)\right| \left(\theta_2(x)\left|g(e^{tS}x-\psi)\right|\right)^p d\psi dx\right)^{\frac{1}{p}} dt \\
            =& \int_{0}^{\infty}e^{-\Re\lambda t}Z^{\frac{1}{q}}(t)
               \left(\int_{\R^d}\left|K^{\beta}(\psi,t)\right| \int_{\R^d}\left(\theta_2(x)\left|g(e^{tS}x-\psi)\right|\right)^p dx d\psi\right)^{\frac{1}{p}} dt \\
            =& \int_{0}^{\infty}e^{-\Re\lambda t}Z^{\frac{1}{q}}(t)
               \left(\int_{\R^d}\left|K^{\beta}(\psi,t)\right| \int_{\R^d}\left(\theta_2(e^{-tS}(y+\psi))\left|g(y)\right|\right)^p dy d\psi\right)^{\frac{1}{p}} dt \\
    \leqslant& \int_{0}^{\infty}e^{-\Re\lambda t}Z^{\frac{1}{q}}(t)
               \left(\int_{\R^d}C_{\theta_2}^p e^{\eta_2 p |\psi|}\left|K^{\beta}(\psi,t)\right| \int_{\R^d}\theta_2^p(y)\left|g(y)\right|^p dy d\psi\right)^{\frac{1}{p}} dt \\
            =& \int_{0}^{\infty}e^{-\Re\lambda t} C_{\theta_2}\left(\int_{\R^d}\left|K^{\beta}(\psi,t)\right| d\psi\right)^{\frac{p-1}{p}}
               \left(\int_{\R^d}e^{\eta_2 p |\psi|}\left|K^{\beta}(\psi,t)\right| d\psi\right)^{\frac{1}{p}} dt
               \left\|g\right\|_{L^p_{\theta_2}} \\
    \leqslant& \int_{0}^{\infty}e^{-\Re\lambda t}C_{4+|\beta|}(t)dt \left\|g\right\|_{L^p_{\theta_2}}
    \leqslant  \frac{C_{7+|\beta|}}{\left(\Re\lambda-\omega_{p,\theta_1}\right)^{1-\frac{\left|\beta\right|}{2}}} \left\|g\right\|_{L^p_{\theta_2}},
  \end{align*}
  where we used the abbreviation
  \begin{align*}
    Z(t):=\int_{\R^d}\left|K^{\beta}(\psi,t)\right| d\psi=\left\|K^{\beta}(\cdot,t)\right\|_{L^1}.
  \end{align*}
\end{proof}

\begin{remark}
  \enum{1} Let us briefly clarify the meaning of Theorem \ref{thm:OrnsteinUhlenbeckAPrioriEstimatesInLpConstantCoefficients} for a concrete choice of weight functions: 
  Consider the weight functions $\theta_1(x)=e^{-\varepsilon_1|x|}$, $\varepsilon_1>0$, and $\theta_2(x)=e^{\varepsilon_2|x|}$, $\varepsilon_2>0$, with growth rates 
  $\eta_1=\varepsilon_1$ and $\eta_2=\varepsilon_2$. Then $L^p_{\theta_1}(\R^d,\C^N)$ contains exponentially increasing functions and $L^p_{\theta_2}(\R^d,\C^N)$ exponentially 
  decreasing functions. If $g\in L^p_{\theta_1}(\R^d,\C^N)$ is exponentially increasing then Corollary \ref{cor:OrnsteinUhlenbeckLpSolvabilityUniqueness} yields a unique 
  exponentially increasing solution $v_{\star}\in\D(A_{p,\theta_1})$ of the resolvent equation $(\lambda I-A_{p,\theta_1})v_{\star}=g$. 
  Theorem \ref{thm:OrnsteinUhlenbeckAPrioriEstimatesInLpConstantCoefficients} then states that if the right hand side $g$ is even exponentially decreasing (in the $L^p$-sense) 
  then the unique exponentially increasing solution $v_{\star}$ and its derivatives up to order $1$ must also decay exponentially at the same rate $\eta_2$ as the inhomogeneity, 
  meaning that $v_{\star}\in W^{1,p}_{\theta_2}(\R^d,\C^N)$. 
  In the following we discuss two special cases of Theorem \ref{thm:OrnsteinUhlenbeckAPrioriEstimatesInLpConstantCoefficients} and their consequences.

  \noindent
  \enum{2} An application of Theorem \ref{thm:OrnsteinUhlenbeckAPrioriEstimatesInLpConstantCoefficients} with $\theta_1=\theta_2=:\theta$ shows that
  \begin{align*}
    \D(A_{p,\theta})\subseteq W^{1,p}_{\theta}(\R^d,\C^N),\text{ for every $1\leqslant p<\infty$}.
  \end{align*}
  This result is useful to solve the identification problem for $A_{p,\theta}$ and to characterize the maximal domain of $\L_{\infty}$ in $L^p_{\theta}(\R^d,\C^N)$. 
  For more details see the forthcoming papers \cite{Otten2015a} and \cite{Otten2015b}. For the unweighted case $\theta_1=\theta_2\equiv 1$ this yields
  \begin{align*}
    \D(A_{p})\subseteq W^{1,p}(\R^d,\C^N),\text{ for every $1\leqslant p<\infty$}.
  \end{align*}
  In general the procedure used in the proof of Theorem \ref{thm:OrnsteinUhlenbeckAPrioriEstimatesInLpConstantCoefficients} does not yield an estimate 
  of second-order derivatives, for instance of $\left\|D_jD_i v_{\star}\right\|_{L^p_{\theta}(\R^d,\C^N)}$. This is caused by the fact that the asymptotic behavior 
  at the origin $C_{4+|\beta|}(t)\sim t^{-\frac{\left|\beta\right|}{2}}$ as $t\to 0$, cf. Theorem \ref{thm:OrnsteinUhlenbeckLpBoundedness}, leads to the singularity 
  $t^{-1}$ at $t=0$ for $\left|\beta\right|=2$ in the integral from Lemma \ref{lem:ImproperInt}.

  \noindent
  \enum{3} The special case $\theta_1\equiv 1$ and $\theta_2=:\theta$ of Theorem \ref{thm:OrnsteinUhlenbeckAPrioriEstimatesInLpConstantCoefficients} is treated 
  in \cite[Theorem 6.8]{Otten2014}. The result states that unique $L^p$-solutions $v_{\star}$ with exponentially decreasing inhomogeneities $g$ 
  decay exponentially. This is crucial when investigating exponential decay of rotating patterns in reaction-diffusion equations.
\end{remark}

%
%
\sect{The Ornstein-Uhlenbeck semigroup in \texorpdfstring{$\Cbt(\R^d,\C^N)$}{C\_\{b,theta\}(Rd,CN)}}
\label{sec:TheOrnsteinUhlenbeckSemigroupInCrub}

We now extend the results from Section \ref{sec:TheOrnsteinUhlenbeckSemigroupInLp} and derive resolvent estimates in exponentially weighted spaces of bounded 
continuous functions. For this purpose recall the family of mappings $\left(T(t)\right)_{t\geqslant 0}$ from \eqref{equ:OrnsteinUhlenbeckSemigroup2} which is 
now defined on the (complex-valued) Banach space $(\Cb(\R^d,\C^N),\left\|\cdot\right\|_{\Cb})$. 

The next two theorems show that the family of mappings $\left(T(t)\right)_{t\geqslant 0}$ generates a semigroup on $\Cb(\R^d,\C^N)$ and on its closed 
subspaces $\Cub(\R^d,\C^N)$ and $\Crub(\R^d,\C^N)$. This is well-known from \cite[Lemma 3.2]{DaPratoLunardi1995} for the scalar real-valued case with $B=0$. In 
the following theorem we even prove boundedness of the operator $T(t)$ in exponentially weighted $\Cb$-spaces. This is crucial for the proof of exponentially 
weighted resolvent estimates below.

\begin{theorem}[Boundedness on $\Cbt(\R^d,\C^N)$]\label{thm:OrnsteinUhlenbeckCbBoundedness}
  Let the assumptions \eqref{cond:A8B}--\eqref{cond:A5} be satisfied. Then for every radial weight function $\theta\in C(\R^d,\R)$ of exponential growth rate 
  $\eta\geqslant 0$ and for every $v\in \Cbt(\R^d,\C^N)$ we have
  \begin{align}
    \left\|T(t)v\right\|_{\Cbt(\R^d,\C^N)}          &\leqslant C_4(t)\left\|v\right\|_{\Cbt(\R^d,\C^N)} &&,\,t\geqslant 0,          \label{equ:OrnsteinUhlenbeckCbBoundednessOfT}     \\
    \left\|D_i T(t)v\right\|_{\Cbt(\R^d,\C^N)}      &\leqslant C_5(t)\left\|v\right\|_{\Cbt(\R^d,\C^N)} &&,\,t>0,\,i=1,\ldots,d,    \label{equ:OrnsteinUhlenbeckCbBoundednessOfDiT}   \\
    \left\|D_j D_i T(t)v\right\|_{\Cbt(\R^d,\C^N)}  &\leqslant C_6(t)\left\|v\right\|_{\Cbt(\R^d,\C^N)} &&,\,t>0,\,i,j=1,\ldots,d,  \label{equ:OrnsteinUhlenbeckCbBoundednessOfDjDiT}
  \end{align}
  where the constants $C_{4+\left|\beta\right|}(t)=C_{4+\left|\beta\right|}(t;\bzero,1)$ for $\left|\beta\right|=0,1,2$ are from Theorem \ref{thm:OrnsteinUhlenbeckLpBoundedness} 
  with $p=1$. Note that $C_{4+\left|\beta\right|}(t)\sim t^{\frac{d-1}{2}} e^{-(\bzero-\nu)t}$ as $t\to\infty$ and $C_{4+\left|\beta\right|}(t)\sim t^{-\frac{\left|\beta\right|}{2}}$ 
  as $t\to 0$ for every $\left|\beta\right|=0,1,2$, where $\nu:=\frac{\amax^2 \eta^2}{\azero}\geqslant 0$ and $\amax,\azero,\bzero$ are defined in \eqref{equ:aminamaxazerobzero}.
\end{theorem}

\begin{proof}
  We use the same notation as in the proof of Theorem \ref{thm:OrnsteinUhlenbeckLpBoundedness}. To show \eqref{equ:OrnsteinUhlenbeckCbBoundednessOfT}, 
  \eqref{equ:OrnsteinUhlenbeckCbBoundednessOfDiT} and \eqref{equ:OrnsteinUhlenbeckCbBoundednessOfDjDiT} for $v\in\Cbt(\R^d,\C^N)$, we use 
  \eqref{equ:OrnsteinUhlenbeckSemigroup2}, the transformation theorem (with transformations $\Phi(\xi)=e^{tS}x-\xi$ in $\xi$, $\Phi(x)=e^{tS}x-\psi$ in $x$), 
  \eqref{equ:K}, \eqref{equ:Ki}, \eqref{equ:Kji}, the triangle inequality, \eqref{equ:WeightFunctionProp1}--\eqref{equ:WeightFunctionProp3} and Lemma \ref{lem:PropK} (1),(2),(3) 
  with $p=1$ to obtain
  \begin{align*}
             & \left\|D^{\beta}T(t)v\right\|_{\Cbt}
            =  \sup_{x\in\R^d}\theta(x)\left|D^{\beta}\left[T(t)v\right](x)\right| \\
            =& \sup_{x\in\R^d}\theta(x)\left|\int_{\R^d}\left[D^{\beta}H(x,\xi,t)\right]v(\xi)d\xi\right| \\
            =& \sup_{x\in\R^d}\theta(x)\left|\int_{\R^d}K^{\beta}(\psi,t)v(e^{tS}x-\psi)d\psi\right| \\
    \leqslant& \sup_{x\in\R^d}\theta(x)\int_{\R^d}\left|K^{\beta}(\psi,t)\right| \left|v(e^{tS}x-\psi)\right|d\psi \\
    \leqslant& \int_{\R^d}\left|K^{\beta}(\psi,t)\right| \sup_{x\in\R^d}\theta(x) \left|v(e^{tS}x-\psi)\right|d\psi \\
            =& \int_{\R^d}\left|K^{\beta}(\psi,t)\right| \sup_{y\in\R^d}\theta(e^{-tS}(y+\psi))\left|v(y)\right|d\psi \\ 
    \leqslant& C_{\theta} \left\|e^{\eta|\cdot|}K^{\beta}(\cdot,t)\right\|_{L^1} \left\|v\right\|_{\Cbt}
    \leqslant  C_{4+|\beta|}(t)\left\|v\right\|_{\Cbt}
  \end{align*}
  for $t\geqslant 0$, if $|\beta|=0$, and for $t>0$, if $|\beta|=1$ or $|\beta|=2$.
\end{proof}

\begin{theorem}[Semigroup on $\Cb(\R^d,\C^N)$, $\Cub(\R^d,\C^N)$ and $\Crub(\R^d,\C^N)$]\label{thm:OrnsteinUhlenbeckCrubSemigroup}
  Let the assumptions \eqref{cond:A8B}--\eqref{cond:A5} be satisfied. Moreover, let $\theta\in C(\R^d,\R)$ be a radial weight function of exponential 
  growth rate $\eta\geqslant 0$. Then the operators $\left(T(t)\right)_{t\geqslant 0}$ given by \eqref{equ:OrnsteinUhlenbeckSemigroup2} generate a semigroup on $\Cb(\R^d,\C^N)$, 
  $\Cub(\R^d,\C^N)$ and $\Crub(\R^d,\C^N)$.
\end{theorem}

\begin{proof}
  1. Similarly as in the proof of Theorem \ref{thm:OrnsteinUhlenbeckLpSemigroup} but with $\theta\equiv 1$ and \eqref{equ:OrnsteinUhlenbeckCbBoundednessOfT} instead of 
  \eqref{equ:OrnsteinUhlenbeckLpBoundednessOfT} we verify that $\left(T(t)\right)_{t\geqslant 0}$ is a semigroup on $\Cb(\R^d,\C^N)$. Obviously, the 
  semigroup properties for $\left(T(t)\right)_{t\geqslant 0}$ are also satisfied on the subspaces $\Cub(\R^d,\C^N)$ and $\Crub(\R^d,\C^N)$. Therefore, we 
  must only prove that for any $t\geqslant 0$ the mapping $T(t)$ maps $\Cub(\R^d,\C^N)$ into itself and $\Crub(\R^d,\C^N)$ into itself.

  \noindent
  2. First we show that $T(t)$ maps $\Cub(\R^d,\C^N)$ into $\Cub(\R^d,\C^N)$ for any $t\geqslant 0$: Let $v\in\Cub(\R^d,\C^N)$, i.e. $v$ satisfies
  \begin{align*}
    \forall\,\tilde{\varepsilon}>0\;\exists\,\tilde{\delta}>0\;\forall\,x,y\in\R^d\text{ with }|x-y|\leqslant\tilde{\delta}:\left|v(x)-v(y)\right|\leqslant\tilde{\varepsilon}.
  \end{align*}
  Let $\varepsilon>0$ be arbitrary, choose $\tilde{\varepsilon}:=\frac{\varepsilon}{C_1(t)}$, where $C_1(t)=C_1(t;\bzero,1)>0$ is from Lemma \ref{lem:PropK} for $\eta=0$, 
  and choose $\delta:=\tilde{\delta}>0$. Using the transformation theorem (with transformations $\Phi(\xi)=e^{tS}x-\xi$ and $\Phi(\xi)=e^{tS}y-\xi$), \eqref{equ:K} 
  and Lemma \ref{lem:PropK} (1), we obtain for every $x,y\in\R^d$ with $\left|x-y\right|\leqslant\delta$ 
  \begin{align*}
             & \left|\left[T(t)v\right](x)-\left[T(t)v\right](y)\right| \\
            =& \left|\int_{\R^d}H(x,\xi,t)v(\xi)d\xi - \int_{\R^d}H(y,\xi,t)v(\xi)d\xi\right| \\
            =& \left|\int_{\R^d}H(x,e^{tS}x-\psi,t)v(e^{tS}x-\psi)d\psi - \int_{\R^d}H(y,e^{tS}y-\psi,t)v(e^{tS}y-\psi)d\psi\right| \\
    \leqslant& \int_{\R^d}\left|K(\psi,t)\right| \left|v(e^{tS}x-\psi)- v(e^{tS}y-\psi)\right| d\psi \\
    \leqslant& \tilde{\varepsilon}\int_{\R^d}\left|K(\psi,t)\right| d\psi
    \leqslant  \tilde{\varepsilon}C_1(t) = \varepsilon.
  \end{align*}
  Therefore, $\left(T(t)\right)_{t\geqslant 0}$ is a semigroup on $\Cub(\R^d,\C^N)$.

  \noindent
  3. Finally, we show that $T(t)$ maps $\Crub(\R^d,\C^N)$ into $\Crub(\R^d,\C^N)$ for any $t\geqslant 0$: Let $v\in C_{rub}(\R^d,\C^N)$, i.e. 
  $\left\|v(e^{\tau S}\cdot)-v(\cdot)\right\|_{C_b(\R^d,\C^N)}\to 0$ as $\tau\to 0$. 
  Using the transformation theorem (with transformations $\Phi(\xi)=e^{tS}x-e^{-\tau S}\xi$ and $\Phi(\xi)=e^{tS}x-\xi$), \eqref{equ:K}, $K(e^{\tau S}\psi,t)=K(\psi,t)$ 
  and Lemma \ref{lem:PropK} (1) (with $p=1$ and $\eta=0$) we obtain
  \begin{align*}
             & \left|\left[T(t)v\right](e^{\tau S}x)-\left[T(t)v\right](x)\right| \\
            =& \left|\int_{\R^d}H(e^{\tau S}x,\xi,t)v(\xi)d\xi-\int_{\R^d}H(x,\xi,t)v(\xi)d\xi\right| \\
            =& \left|\int_{\R^d}H(e^{\tau S}x,e^{\tau S}(e^{tS}x-\psi),t)v(e^{\tau S}(e^{tS}x-\psi))d\psi\right. \\
             & \left.\quad\quad\quad-\int_{\R^d}H(x,e^{tS}x-\psi,t)v(e^{tS}x-\psi)d\psi\right| \\
    \leqslant& \int_{\R^d}\left|K(\psi,t)\right| \left|v(e^{\tau S}(e^{tS}x-\psi))-v(e^{tS}x-\psi)\right|d\psi \\
    \leqslant& \int_{\R^d}\left|K(\psi,t)\right| \sup_{x\in\R^d}\left|v(e^{\tau S}(e^{tS}x-\psi))-v(e^{tS}x-\psi)\right|d\psi \\
            =& \int_{\R^d}\left|K(\psi,t)\right| \sup_{y\in\R^d}\left|v(e^{\tau S}y)-v(y)\right|d\psi \\
            =& \left\|v(e^{\tau S}\cdot)-v(\cdot)\right\|_{C_b} \int_{\R^d}\left|K(\psi,t)\right| d\psi \\
    \leqslant& C_1(t) \left\|v(e^{\tau S}\cdot)-v(\cdot)\right\|_{C_b},\,x\in\R^d.
  \end{align*}
  Now, we take the suprema over all $x\in\R^d$ and the limit as $\tau\to 0$ to end up with
  \begin{align*}
    \lim_{\tau\to 0}\left\|\left[T(t)v\right](e^{\tau S}\cdot)-\left[T(t)v\right](\cdot)\right\|_{C_b}\leqslant C_1(t) \lim_{\tau\to 0}\left\|v(e^{\tau S}\cdot)-v(\cdot)\right\|_{C_b}=0.
  \end{align*}
\end{proof}

The next result shows that the Ornstein-Uhlenbeck semigroup $\left(T(t)\right)_{t\geqslant 0}$ is neither strongly continuous on $\Cb(\R^d,\C^N)$ nor on $\Cub(\R^d,\C^N)$. 
More precisely, $\left(T(t)\right)_{t\geqslant 0}$ generates only a weakly continuous semigroup on $\Cub(\R^d,\C^N)$. This observation is due to \cite[Section 6]{Cerrai1994} 
and \cite{DaPratoLunardi1995} for the scalar real-valued Ornstein-Uhlenbeck operator. It was shown in \cite[Lemma 3.2]{DaPratoLunardi1995} that $\Crub(\R^d,\R)$ is the 
largest subspace of $\Cub(\R^d,\R)$ such that the scalar real-valued Ornstein-Uhlenbeck operator is strongly continuous. The next theorem is an extension of 
\cite[Lemma 3.2]{DaPratoLunardi1995} to complex systems and states that the Ornstein-Uhlenbeck semigroup $\left(T(t)\right)_{t\geqslant 0}$ is strongly continuous on 
$\Crub(\R^d,\C^N)$. The proof follows the line of thought from Theorem \ref{thm:OrnsteinUhlenbeckLpStrongContinuity}.

\begin{theorem}[Strong continuity on $\Crub(\R^d,\C^N)$]\label{thm:OrnsteinUhlenbeckCrubStrongContinuity}
  Let the assumptions \eqref{cond:A8B}, \eqref{cond:A2} and \eqref{cond:A5} be satisfied. Then $\left(T(t)\right)_{t\geqslant 0}$ is a $C_0$-semigroup 
  (or strongly continuous semigroup) on $\Crub(\R^d,\C^N)$, i.e.
  \begin{align}
    \lim_{t\downarrow 0}\left\|T(t)v-v\right\|_{\Cb(\R^d,\C^N)}=0\;\forall\,v\in\Crub(\R^d,\C^N). \label{equ:CrubStrongCont}
  \end{align}
\end{theorem}

\begin{proof}
  The proof is similar to the unweighted version of Theorem \ref{thm:OrnsteinUhlenbeckLpStrongContinuity}.

  \noindent
  1.-3. The first three steps can be adopted from the proof of Theorem \ref{thm:OrnsteinUhlenbeckLpStrongContinuity} with $\theta\equiv 1$ and 
  $\left\|\cdot\right\|_{\Cb}$ instead of $\left\|\cdot\right\|_{L^p_{\theta}}$.
  
  \noindent 
  4. Let us consider the $v_3$-term: The only difference to the $L^p_{\theta}$-case is that we need $\Cb$-continuity (at $0$) w.r.t. translations, i.e.
  \begin{align}
    \label{equ:CbContinuityTranslation}
    \forall\,\varepsilon_0>0\;\exists\,\delta_0>0\;\forall\,\psi\in\R^d\text{ with }|\psi|\leqslant\delta_0:\;
    \left\|v(\cdot-\psi)-v(\cdot)\right\|_{\Cb}\leqslant\varepsilon_0,
  \end{align}
  instead of the $L^p_{\theta}$-continuity from \eqref{equ:LpthetaContinuityTranslation}. Since $v\in\Cub(\R^d,\C^N)$ we have
  \begin{align*}
    \forall\,\varepsilon_0>0\;\exists\,\delta_0>0\;\forall\,x,y\in\R^d\text{ with }|x-y|\leqslant\delta_0:\;
    \left|v(x)-v(y)\right|\leqslant\varepsilon_0.
  \end{align*}
  Choosing $x:=y-\psi$ and taking the supremum over $y\in\R^d$ we obtain exactly \eqref{equ:CbContinuityTranslation}.

  \noindent
  5. Finally, let us consider the $v_2$-term. Here
  \begin{align*}
    \lim_{t\downarrow 0}\left\|v_2(\cdot,t)\right\|_{\Cb}=\lim_{t\downarrow 0}\left\|v(e^{tS}\cdot)-v(\cdot)\right\|_{\Cb}=0
  \end{align*}
  follows directly from the definition of $\Crub(\R^d,\C^N)$, since $v\in\Crub(\R^d,\C^N)$.
\end{proof}

\begin{remark}
  \enum{1} As shown in Theorem \ref{thm:OrnsteinUhlenbeckCrubStrongContinuity} the semigroup $(T(t))_{t\geqslant 0}$ is neither strongly continuous 
           on $\Cb(\R^d,\C^N)$ nor on $\Cub(\R^d,\C^N)$. To guarantee strong continuity we must choose an appropriate subspace, that is $\Crub(\R^d,\C^N)$. 
           But in the special case $S=0$ the semigroup is strongly continuous in $\Cub(\R^d,\C^N)=\Crub(\R^d,\C^N)$ since $\Cub$-functions guarantee continuity 
           with respect to translations. For general matrices $S\in\R^{d,d}$ we must choose a suitable subspace of $\Cub(\R^d,\C^N)$ that implies continuity with 
           respect to rotations. This is $\Crub(\R^d,\C^N)$ introduced in \cite{DaPratoLunardi1995}.

  \noindent
  \enum{2} We point out that we analyzed the semigroup only in unweighted $\Cb$-spaces. This is in contrast to the $L^p$-theory from Section \ref{sec:TheOrnsteinUhlenbeckSemigroupInLp}. 
           One main problem that arises when investigating the semigroup in exponentially weighted $\Cb$-spaces is that we do not know how to define 
           appropriate subspaces $\Cubt(\R^d,\C^N)$ and $\Crubt(\R^d,\C^N)$ of $\Cbt(\R^d,\C^N)$ which guarantee strong continuity for all $S\in\R^{d,d}$. 
           This remains as an open problem.  Analogously to Remark \ref{rem:RemarkLpFunctionSpace} one may introduce $\Cbt$-continuity w.r.t. translations
           \begin{align}
             \label{equ:CbthetaContinuityTranslation}
             \forall\,\varepsilon_0>0\;\exists\,\delta_0>0\;\forall\,\psi\in\R^d\text{ with }|\psi|\leqslant\delta_0:\;
             \left\|v(\cdot-\psi)-v(\cdot)\right\|_{\Cbt}\leqslant\varepsilon_0
           \end{align}
           and $\Cbt$-continuity w.r.t. rotations
           \begin{align}
             \label{equ:CbthetaContinuityRotation}
             \forall\,\varepsilon_0>0\;\exists\,t_0>0\;\forall\,0\leqslant t\leqslant t_0:\;\left\|v(e^{tS}\cdot)-v(\cdot)\right\|_{\Cbt}\leqslant\varepsilon_0.
           \end{align}
           With this define the closed subspace
           \begin{align*}
             X_{\mathrm{b},\theta}:=\left\{v\in\Cbt(\R^d,\C^N)\mid\text{ $v$ satisfies \eqref{equ:CbthetaContinuityTranslation} and \eqref{equ:CbthetaContinuityRotation}}\right\}.
           \end{align*}
           Then, we expect $(T(t))_{t\geqslant 0}$ to be strongly continuous on $X_{\mathrm{b},\theta}$.
\end{remark}

For the infinitesimal generator of $(T(t))_{t\geqslant 0}$ in $\Crub(\R^d,\C^N)$, denoted by $\left(A_{\mathrm{b}},\D(A_{\mathrm{b}})\right)$, we define 
the spectrum $\sigma(A_{\mathrm{b}})$, the resolvent set $\rho(A_{\mathrm{b}})$ and the resolvent $R(\lambda,A_{\mathrm{b}})$ of $A_{\mathrm{b}}$ as above 
in \eqref{equ:SpectrumResolventsetResolventLp}.

In the following Corollary \ref{cor:OrnsteinUhlenbeckCrubSolvabilityUniqueness} we prove that the resolvent equation $(\lambda I-A_{\mathrm{b}})v=g$ has a unique 
solution in $\D(A_{\mathrm{b}})$. This can be deduced as in the $L^p$-case and requires an estimate of the form
\begin{align}
  \exists\,\omega_{\mathrm{b}}\in\R\;\wedge\;\exists\,M_{\mathrm{b}}\geqslant 1:\;\left\|T(t)\right\|_{\mathcal{L}(\Cb,\Cb)}\leqslant M_{\mathrm{b}} e^{\omega_{\mathrm{b}} t}\;\forall\,t\geqslant 0.
  \label{equ:OrnsteinUhlenbeckBoundednessOfTinCb}
\end{align}
Since we only consider the unweighted $\Cb$-case, this follows directly from the estimate \eqref{equ:OrnsteinUhlenbeckCbBoundednessOfT} in Theorem \ref{thm:OrnsteinUhlenbeckCbBoundedness} 
(with $\theta\equiv 1$, $C_{\theta}=1$ and $\eta=\nu=0$) and leads to
\begin{align}
  \left\|T(t)\right\|_{\L(\Cb,\Cb)}\leqslant M_{\mathrm{b}} e^{\omega_{\mathrm{b}} t}\;\forall\,t\geqslant 0\quad\text{with}\quad M_{\mathrm{b}}:=\kappa\aone^{\frac{d}{2}},\quad\omega_{\mathrm{b}}:=-\bzero,
  \label{equ:OrnsteinUhlenbeckBoundednessOfTinCb2}
\end{align}
where $\aone:=\frac{\amax^2}{\amin\azero}\geqslant 1$, $\kappa:=\cond(Y)$ with $Y$ from \eqref{cond:A8B} and $\bzero:=s(-B)$. From \cite[II.1.10 Theorem]{EngelNagel2000} we obtain

\begin{corollary}[Solvability and uniqueness in $\Crub(\R^d,\C^N)$]\label{cor:OrnsteinUhlenbeckCrubSolvabilityUniqueness}
  Let the assumptions \eqref{cond:A8B}--\eqref{cond:A5} be satisfied. Moreover, let $\lambda\in\C$ with $\Re\lambda>\omega_{\mathrm{b}}$ 
  (see \eqref{equ:OrnsteinUhlenbeckBoundednessOfTinCb2}). Then for every $g\in\Crub(\R^d,\C^N)$ the resolvent equation
  \begin{align*}
    \left(\lambda I-A_{\mathrm{b}}\right)v = g
  \end{align*}
  admits a unique solution $v_{\star}\in\D(A_{\mathrm{b},\theta})$, which is given by the integral expression
  \begin{align}
    v_{\star} = R(\lambda)g = \int_{0}^{\infty}e^{-\lambda t}T(t)g dt
                            = \int_{0}^{\infty}e^{-\lambda t}\int_{\R^d}H(\cdot,\xi,t)g(\xi)d\xi dt.
    \label{equ:IntegralExpressionCb}
  \end{align}
  Moreover, it holds the resolvent estimate
  \begin{align*}
    \left\|v_{\star}\right\|_{\Cb(\R^d,\C^N)}\leqslant \frac{M_{\mathrm{b}}}{\Re\lambda-\omega_{\mathrm{b}}}\left\|g\right\|_{\Cb(\R^d,\C^N)}.
  \end{align*}
\end{corollary}

Corollary \ref{cor:OrnsteinUhlenbeckCrubSolvabilityUniqueness} states that the complete right half-plane $\Re\lambda>\omega_{\mathrm{b}}$ belongs to the resolvent set 
$\rho(A_{\mathrm{b}})$. Therefore, the $\Cb$-spectrum $\sigma(A_{\mathrm{b}})$ is contained in the left half-plane $\Re\lambda\leqslant\omega_{\mathrm{b}}$. For general 
results concerning the $\Cub$-spectrum we refer to \cite[Section 6]{Metafune2001}. 

\begin{theorem}[Exponentially weighted resolvent estimates in $\Cbt(\R^d,\C^N)$]\label{thm:OrnsteinUhlenbeckAPrioriEstimatesInCrub}
  Let the assumptions \eqref{cond:A8B}--\eqref{cond:A5} be satisfied. Moreover, let $0<\vartheta<1$, $\lambda\in\C$ with $\Re\lambda>\omega_{\mathrm{b}}$ 
  (see \eqref{equ:OrnsteinUhlenbeckBoundednessOfTinCb2}) and let $\theta\in C(\R^d,\R)$ be a radial weight function of exponential growth rate $\eta\geqslant 0$ with 
  \begin{align*}
    0\leqslant\eta^2\leqslant\vartheta\frac{\azero(\Re\lambda-\omega_{\mathrm{b}})}{\amax^2},
  \end{align*}
  which satisfies the relation
  \begin{align*}
    \theta(x)\geqslant C>0\;\forall\,x\in\R^d.
  \end{align*}
  Let $g\in\Crub(\R^d,\C^N)$ and $v_{\star}\in\D(A_{\mathrm{b}})$ be the unique solution of $(\lambda I-A_{\mathrm{b}})v = g$ in $\Crub(\R^d,\C^N)$ according to 
  Corollary \ref{cor:OrnsteinUhlenbeckCrubSolvabilityUniqueness}, then $v_{\star}\in\Crub(\R^d,\C^N)\cap\Cub^1(\R^d,\C^N)$.
  If in addition $g\in\Cbt(\R^d,\C^N)$ then we have 
  \begin{align*}
    v_{\star}\in\Crub(\R^d,\C^N)\cap\Cub^1(\R^d,\C^N)\cap\Cbt^1(\R^d,\C^N)
  \end{align*}
  and the following estimates hold
  \begin{align}
        \left\|v_{\star}\right\|_{C_{\mathrm{b,\theta}}(\R^d,\C^N)} \leqslant& \frac{C_7}{\Re\lambda-\omega_{\mathrm{b}}}\left\|g\right\|_{C_{\mathrm{b,\theta}}(\R^d,\C^N)},                                             \label{equ:OrnsteinUhlenbeckExpDecStatVstarInCrub}\\
    \left\|D_i v_{\star}\right\|_{C_{\mathrm{b,\theta}}(\R^d,\C^N)} \leqslant& \frac{C_8}{\left(\Re\lambda-\omega_{\mathrm{b}}\right)^{\frac{1}{2}}}\left\|g\right\|_{C_{\mathrm{b,\theta}}(\R^d,\C^N)},\,i=1,\ldots,d,   \label{equ:OrnsteinUhlenbeckExpDecStatDiVstarInCrub}
  \end{align}
  where the $\lambda$-independent constants $C_7,C_8$ are given by Lemma \ref{lem:ImproperInt} with $p=1$.
  Moreover, the following pointwise estimates are satisfied for every $x\in\R^d$ 
  \begin{align*}
    \left|v_{\star}(x)\right|     \leqslant& \frac{C_7}{\Re\lambda-\omega_{\mathrm{b}}}\frac{1}{\theta(x)} \left\|g\right\|_{C_{\mathrm{b,\theta}}(\R^d,\C^N)},\\
    \left|D_i v_{\star}(x)\right| \leqslant& \frac{C_8}{\left(\Re\lambda-\omega_{\mathrm{b}}\right)^{\frac{1}{2}}}\frac{1}{\theta(x)} \left\|g\right\|_{C_{\mathrm{b,\theta}}(\R^d,\C^N)},\,i=1,\ldots,d.
  \end{align*}
\end{theorem}

\begin{remark}
  Moreover, if $\theta$ additionally satisfies \eqref{equ:WeightFunctionProp7} then for every $x\in\R^d$ the following pointwise estimates are satisfied
  \begin{align*}
    \left|v_{\star}(x)\right|     \leqslant& \tilde{C}_{\theta}\frac{C_7}{\Re\lambda-\omega_{\mathrm{b}}}e^{-\nu|x|}\left\|g\right\|_{C_{\mathrm{b,\theta}}(\R^d,\C^N)},\\
    \left|D_i v_{\star}(x)\right| \leqslant& \tilde{C}_{\theta}\frac{C_8}{\left(\Re\lambda-\omega_{\mathrm{b}}\right)^{\frac{1}{2}}}e^{-\nu|x|}\left\|g\right\|_{C_{\mathrm{b,\theta}}(\R^d,\C^N)},\,i=1,\ldots,d.
  \end{align*}
\end{remark}

\begin{proof}
  In the first step we show that $v_{\star}\in\Crub(\R^d,\C^N)\cap\Cub^1(\R^d,\C^N)$ and in the second step we prove the resolvent estimates.
  For the proof we use the notation from Theorem \ref{thm:OrnsteinUhlenbeckCbBoundedness}.
 
  \noindent
  1. Let $v_{\star}\in\D(A_{\mathrm{b}})$ be from Corollary \ref{cor:OrnsteinUhlenbeckCrubSolvabilityUniqueness} the unique solution of the resolvent equation for 
  some $g\in\Crub(\R^d,\C^N)$, then by definition of the maximal domain $\D(A_{\mathrm{b}})$ we have $v_{\star}\in\Crub(\R^d,\C^N)$. It remains to show $v_{\star}\in\Cub^1(\R^d,\C^N)$: 
  Since $g\in\Cub(\R^d,\C^N)$ we have
  \begin{align*}
    \forall\,\tilde{\varepsilon}>0\;\exists\,\tilde{\delta}>0\;\forall\,x,y\in\R^d\text{ with }|x-y|\leqslant\tilde{\delta}:\;\left|g(x)-g(y)\right|\leqslant\tilde{\varepsilon}.
  \end{align*}
  Let $\varepsilon>0$ be arbitrary and choose $\tilde{\varepsilon}:=\frac{\varepsilon(\Re\lambda-\omega_{\mathrm{b}})^{1-\frac{|\beta|}{2}}}{C_{7+|\beta|}}$, where 
  $C_{7+|\beta|}$ is from Lemma \ref{lem:ImproperInt} (with $\omega=\omega_{\mathrm{b}}$, $p=1$, $C_{\theta}=1$, $\eta=0$), and choose $\delta:=\tilde{\delta}>0$. 
  Then we obtain for every $|\beta|\in\{0,1\}$ and $x,y\in\R^d$ with $\left|x-y\right|\leqslant\delta$ using \eqref{equ:IntegralExpressionCb}, the transformation theorem 
  (with transformations $\Phi(\xi)=e^{tS}x-\xi$ and $\Phi(\xi)=e^{tS}y-\xi$), \eqref{equ:K}, \eqref{equ:Ki}, Lemma \ref{lem:PropK}, and Lemma \ref{lem:ImproperInt} (observe that 
  $C_{1+|\beta|}(t;\bzero,1)=C_{4+|\beta|}(t;\bzero,1)$ for $p=1$ and $\eta=0$, since $C_{\theta}=1$)
  \begin{align*}
             & \left|D_i v_{\star}(x) - D_i v_{\star}(y)\right| \\
            =& \left|\int_{0}^{\infty}e^{-\lambda t}\int_{\R^d}D_i H(x,\xi,t)g(\xi)d\xi dt-\int_{0}^{\infty}e^{-\lambda t}\int_{\R^d}D_i H(y,\xi,t)g(\xi)d\xi dt\right| \\
            =& \left|\int_{0}^{\infty}e^{-\lambda t}\int_{\R^d} K^{\beta}(\psi,t)\left(g(e^{tS}x-\psi)-g(e^{tS}y-\psi)\right)d\psi dt\right| \\
    \leqslant& \int_{0}^{\infty}e^{-\Re\lambda t}\int_{\R^d} \left|K^{\beta}(\psi,t)\right| \left|g(e^{tS}x-\psi)-g(e^{tS}y-\psi)\right|d\psi dt \\
    \leqslant& \tilde{\varepsilon} \int_{0}^{\infty}e^{-\Re\lambda t}\int_{\R^d}\left|K^{\beta}(\psi,t)\right| d\psi dt
    \leqslant  \tilde{\varepsilon} \int_{0}^{\infty}e^{-\Re\lambda t} C_{1+|\beta|}(t)dt \\
            =& \tilde{\varepsilon} \int_{0}^{\infty}e^{-\Re\lambda t} C_{4+|\beta|}(t)dt
    \leqslant  \tilde{\varepsilon}\frac{C_{7+|\beta|}}{(\Re\lambda+\bzero)^{1-\frac{|\beta|}{2}}}=\varepsilon.
  \end{align*}

  \noindent
  2. Now, let in addition $g\in\Cbt(\R^d,\C^N)$. To show \eqref{equ:OrnsteinUhlenbeckExpDecStatVstarInCrub} and \eqref{equ:OrnsteinUhlenbeckExpDecStatDiVstarInCrub}, 
  we use the integral expression \eqref{equ:IntegralExpressionCb}, the transformation theorem (with transformation $\Phi(\xi)=e^{tS}x-\xi$ in $\xi$ and $\Phi(x)=e^{tS}x-\psi$ in $x$), 
  \eqref{equ:K} and \eqref{equ:Ki}, \eqref{equ:WeightFunctionProp1}--\eqref{equ:WeightFunctionProp3}, Lemma \ref{lem:ImproperInt} (with $\omega=\omega_{\mathrm{b}}$ and $p=1$) 
  and obtain for every $\beta\in\N_0^d$ with $\left|\beta\right|\in\{0,1\}$
  \begin{align*}
             & \left\|D^{\beta}v_{\star}\right\|_{C_{\mathrm{b},\theta}}
            =  \sup_{x\in\R^d}\theta(x)\left|D^{\beta}v_{\star}(x)\right| \\
            =& \sup_{x\in\R^d}\theta(x)\left|\int_{0}^{\infty}e^{-\lambda t}\int_{\R^d}\left[D^{\beta} H(x,\xi,t)\right]g(\xi)d\xi dt\right| \\
            =& \sup_{x\in\R^d}\theta(x)\left|\int_{0}^{\infty}e^{-\lambda t}\int_{\R^d}K^{\beta}(\psi,t)g(e^{tS}x-\psi)d\psi dt\right| \\        
    \leqslant& \int_{0}^{\infty}e^{-\Re\lambda t}\sup_{x\in\R^d}\theta(x)\left|\int_{\R^d}K^{\beta}(\psi,t)g(e^{tS}x-\psi)d\psi\right| dt \\
    \leqslant& \int_{0}^{\infty}e^{-\Re\lambda t}\left(\sup_{x\in\R^d}\int_{\R^d}\theta(x)\left|K^{\beta}(\psi,t)\right| \left|g(e^{tS}x-\psi)\right|d\psi\right) dt \\
    \leqslant& \int_{0}^{\infty}e^{-\Re\lambda t}\int_{\R^d}\left|K^{\beta}(\psi,t)\right| \left(\sup_{x\in\R^d}\theta(x)\left|g(e^{tS}x-\psi)\right|\right) d\psi dt \\
            =& \int_{0}^{\infty}e^{-\Re\lambda t}\int_{\R^d}\left|K^{\beta}(\psi,t)\right| \left(\sup_{y\in\R^d}\theta(e^{-tS}(y+\psi))\left|g(y)\right|\right) d\psi dt \\
    \leqslant& \int_{0}^{\infty}e^{-\Re\lambda t}\int_{\R^d}C_{\theta} e^{\eta |\psi|}\left|K^{\beta}(\psi,t)\right| \left(\sup_{y\in\R^d}\theta(y)\left|g(y)\right|\right) d\psi dt \\
            =& \int_{0}^{\infty}e^{-\Re\lambda t} C_{\theta}\left(\int_{\R^d}e^{\eta |\psi|}\left|K^{\beta}(\psi,t)\right| d\psi\right) dt \left\|g\right\|_{C_{\mathrm{b},\theta}} \\
    \leqslant& \int_{0}^{\infty}e^{-\Re\lambda t}C_{4+|\beta|}(t;\bzero,p=1)dt \left\|g\right\|_{C_{\mathrm{b},\theta}}
    \leqslant  \frac{C_{7+|\beta|}}{\left(\Re\lambda-\omega_{\mathrm{b}}\right)^{1-\frac{\left|\beta\right|}{2}}} \left\|g\right\|_{C_{\mathrm{b},\theta}}.
  \end{align*}
\end{proof}

\textbf{Acknowledgment.} The author thanks the referee for several helpful and constructive remarks. 


\def\cprime{$'$}

\end{document}